[1994/12/01]

\documentclass{amsart}
\usepackage{graphicx}
 \usepackage{stmaryrd}
 \usepackage{amssymb}
\usepackage[dvipsnames]{xcolor}
\usepackage{enumerate}

\usepackage[colorlinks=true,
            linkcolor=violet,
            urlcolor=Aquamarine,
            citecolor=YellowOrange]{hyperref}
\definecolor{darkgreen}{rgb}{0,0.5,0}
\definecolor{darkred}{rgb}{0.7,0,0}

\def\({\left (}
\def\){\right )}
\def\<{\left\langle}
\def\>{\right\rangle}

\newtheorem{thm}{Theorem}[section]

\newtheorem{lem}[thm]{Lemma}
\newtheorem{prop}[thm]{Proposition}
\newtheorem{defn}[thm]{Definition}
\newtheorem{rem}[thm]{Remark}

\numberwithin{equation}{section}

\newcommand{\norm}[1]{\left\Vert#1\right\Vert}
\newcommand{\abs}[1]{\left\vert#1\right\vert}
\newcommand{\set}[1]{\left\{#1\right\}}
\newcommand{\Real}{\mathbb R}

\newcommand{\pfrac}[2]{\frac{\partial #1}{\partial #2}}

\begin{document}

\title{Analysis aspects of Ricci flow on conical surfaces}

\author{Hao Yin}

\address{Hao Yin,  School of Mathematical Sciences,
University of Science and Technology of China, Hefei, China}
\email{haoyin@ustc.edu.cn }
\thanks{The research is supported by NSFC 11471300.}
\maketitle
\begin{abstract}
In this paper, we establish a framework for the analysis of linear parabolic equations on conical surfaces and use them to study the conical Ricci flow. In particular, we prove the long time existence of the conical Ricci flow for general cone angle and show that this solution has the optimal regularity, namely, the time derivatives of the conformal factor are bounded and for each fixed time, the conformal factor has an explicit asymptotic expansion near the cone points.
\end{abstract}

\tableofcontents

\section{Introduction}
In this paper, we study the (normalized) Ricci flow on surfaces with conical singularities. Let $S$ be a smooth Riemann surface and $\set{p_i}$ be finitely many prescribed points on $S$. For each $p_i$, we assign a weight $\beta_i>-1$. We are interested in the class of metrics $g$ which are smooth and compatible with the conformal structure of $S$ away from $p_i$ while having a conical singularity of order $\beta_i$ at $p_i$, i.e. in a conformal coordinate chart $U$ around $p_i$, $g$ is given by
\begin{equation*}
	e^{2u} r^{2\beta} (dx^2+dy^2),
\end{equation*}
where $r=\sqrt{x^2+y^2}$ and $u$ is at least continuous around $p_i$. These metrics are incomplete, and hence there are many different ways of talking about Ricci flows starting from such metrics. For example, the conical singularities can be smoothed out immediately to become the ordinary Ricci flow on closed manifolds as in \cite{simon2001deformation,ramos2011smoothening,simon2013local}, or the conical singularities can be pushed to infinity immediately to become some Ricci flow on complete noncompact surfaces as in \cite{topping2010ricci}. Here we are interested in a special type of Ricci flow which preserves the conical singularities while deforms the smooth part into nicer and nicer geometry. 

To describe the flow, it is convenient to have a background metric $\tilde{g}$ which is exactly the cone metric of order $\beta_i$ near small neighborhood of each $p_i$. Throughout this paper, we shall have $S$, $\set{p_i}$, $\set{\beta_i}$ and $\tilde{g}$ fixed, which we denote by $(S,\beta,\tilde{g})$ for simplicity. Consider the family of metrics $g(t)$ given by $e^{2u(t)}\tilde{g}$ with $u(t)$ being `good' near $p_i$ so that $g(t)$ is still conical at $p_i$ and $u(t)$ satisfies the equation 
\begin{equation}
	\partial_t u = e^{-2u} \tilde{\triangle}u + \frac{r}{2} -e^{-2u} \tilde{K}
	\label{eqn:rfu}
\end{equation}
on $S\setminus \set{p_i}$ where $\tilde{\triangle}$ and $\tilde{K}$ are the Laplacian and the Gauss curvature of $\tilde{g}$ and $r$ is some normalization constant. The exact meaning of `being good' depends on the approach we take to study the problem and is also a central theme of the present paper. 

This singularity-preserving flow was first studied by the author in \cite{yin2010ricci}, where a local existence result was claimed. The study was continued in \cite{yin2013ricci} which proves the long time existence of the flow if $\beta_i\in (-1,0)$ and gives some convergence results in certain cases. Unfortunately, the proofs in both papers \cite{yin2010ricci,yin2013ricci} contain gaps. We refer the interested readers to the historical remark near the end of this section. 

In the mean time, the topic of Ricci flow with conical singularity attracts the attention of many authors. Mazzeo, Rubinstein and Sesum \cite{mazzeo2013ricci}, backed up by the mircolocal analysis method developed by Mazzeo \cite{mazzeo1991elliptic}, Bahuaud and Vertman \cite{bahuaud2014yamabe}, Jeffres and Loya \cite{jeffres2003regularity}, Mooers \cite{mooers1999heat} and others, proved the long time existence of \eqref{eqn:rfu} when $\beta_i\in (-1,0)$, showed the convergence to constant curvature metric if some stability condition is satisfied and speculated an interesting phenomenon when there is no constant curvature metric in the conformal class. This last phenomenon, which is called the one dimensional Hamilton-Tian conjecture by the authors in \cite{mazzeo2013ricci}, was proved by Phong, Song, Strum and Wang \cite{phong2014ricci,phong2015convergence}.

We should also mention that the higher dimensional generalization of this flow, i.e. the K\"ahler Ricci flow singular along a smooth divisor was studied by Chen and Wang \cite{chen2015bessel,chen2014long} and Wang \cite{wang2014smooth}. They developed in \cite{chen2015bessel} a parabolic version of Donaldson's $C^{2,\alpha}$ estimate \cite{donaldson2012kahler}. Naturally, their proof can be reduced to complex dimension one to give some results.

In this paper, we discuss the PDE aspect of the conical Ricci flow by developing a weak solution theory to the linear parabolic equation on surfaces with conical singularities. This includes a new space of functions in which the weak solution lies, various linear estimates using either the maximum principle or the energy method, some smoothing estimates. We refer to Section \ref{sec:weaksolution} and Section \ref{sec:smoothing} for the exact statement of these results. Using these PDE results, we are able to prove the following main theorems.

The first is an existence result. 
\begin{thm}\label{thm:main1}
	Suppose $(S,\beta,\tilde{g})$ is a conical surface with a background metric fixed. For any function $u_0:S\to \Real$ satisfying that $u_0$ and $\tilde{\triangle} u_0$ lie in $\mathcal W^{2,\alpha}$, there exist $T>0$ depending only on the $\mathcal W^{2,\alpha}$ norm of $u_0$ and $\tilde{\triangle} u_0$ and $u(t)\in \mathcal V^{2,\alpha,[0,T]}$ solving \eqref{eqn:rfu}. If $r$ is chosen so that $r/2$ is the average of Gauss curvature of the metric $e^{2u_0} \tilde{g}$, then $u(t)$ is defined on $[0,\infty)$ and for any $T\in [0,\infty)$, $u\in \mathcal V^{2,\alpha,[0,T]}$.
\end{thm}
For the exact definition of $\mathcal W^{2,\alpha}$ and $\mathcal V^{2,\alpha,[0,T]}$, see Section \ref{sec:weaksolution}. It suffices for now to remark that these spaces consist of functions which are $C^{2,\alpha}$ away from the singularities and which have some bounded integral norm near the singularities so that it behaves like the weak solution in the standard theory of parabolic equation on smooth domains (see Chapter III of \cite{ladyzhenskaya1968linear} for example).

Theorem \ref{thm:main1} holds for any $\beta_i>-1$, while the long time existence results in \cite{yin2013ricci} and \cite{mazzeo2013ricci} are both restricted to the case of sharp cone angles, i.e. $\beta_i\in (-1,0)$ (note that this is equivalent to $\beta_i \in (0,1)$ in \cite{mazzeo2013ricci}). The reason behind this is that the method used by Hamilton \cite{hamilton1988ricci} requires the gradient of some potential function to be bounded, which may not be true if some $\beta_i>0$. The proof here focuses on the conformal factor $u$ instead of the curvature $K$. We get an apriori bound for the $C^0$ norm of $u$ first by using the method from K\"ahler geometry and then turn it into a bound of $K$ by pure PDE methods for quasilinear parabolic equations.  

Given the existence of the solution in Theorem \ref{thm:main1}, it is natural to ask how smooth the solution is near a singular point. Usually, the higher order regularity of a nonlinear PDE is proved by successively taking derivatives. In our case, we can not take spacial derivative due to the singularity and we study first the time derivatives.

\begin{thm}
	\label{thm:main2}
	Let $u(t)$ be the solution given in Theorem \ref{thm:main1}. For all $k\in \mathbb N$ and $0<\delta<T<\infty$, $\partial_t^k u$ lies in $\mathcal V^{2,\alpha,[\delta,T]}$. In particular, $\partial_t^k u$ is bounded in $S\times [\delta,T]$.
\end{thm}

To describe the regularity property of $u(t)$ near a singular point, we take the polar coordinates $(\rho,\theta)$ around a singular point such that
\begin{equation*}
	\tilde{g}= d\rho^2 + \rho^2(\beta+1)^2 d\theta^2,
\end{equation*}
which is the standard cone metric of order $\beta$. Here $\rho$ is the Riemannian distance to the cone point with respect to $\tilde{g}$. We prove
\begin{thm}
	\label{thm:main3}
	Let $u(t)$ be the solution given in Theorem \ref{thm:main1}. For any $t>0$ and $q>0$ fixed, we have
	\begin{equation*}
		u(t)= \sum_{v\in \mathcal T^q} a_v v + \tilde{O}(q)
	\end{equation*}
	in a neighborhood of the singular point for some real numbers $a_{v}$,
	where 
	\begin{equation*}
	\mathcal T^q =\set{ \rho^{2j+\frac{k}{\beta+1}}\cos l \theta, \rho^{2j+\frac{k}{\beta+1}}\sin l \theta \,|\, l,j,k,\frac{k-l}{2}\in \mathbb N \cup \set{0}; 2j+\frac{k}{\beta+1}< q}
	\end{equation*}
	and $\tilde{O}(q)$ is some error term satisfying that for any $k_1,k_2 \in \mathbb N\cup\set{0}$,
	\begin{equation*}
		\abs{(\rho\partial_\rho)^{k_1} \partial_\theta^{k_2} \tilde{O}(q)} \leq C(k_1,k_2) \rho^q
	\end{equation*}
	in the same neighborhood as above.
\end{thm}

In Theorem 2.1 of \cite{mazzeo2013ricci}, an asymptotic expansion of the metric $g(t)$ is given in the form
\begin{equation*}
	g\sim \left( \sum_{j,k\geq 0}\sum_{l=0}^{N_{j,k}} a_{jkl}(y) r^{j+k/\beta}(\log r)^l \right) \abs{z}^{2\beta-2} \abs{dz}^2.
\end{equation*}
Please note that the $\beta$ and $r$ in the above equation are $\beta+1$ and $\rho$ in this paper. Theorem \ref{thm:main3} provides more information by putting more restrictions on the terms that appear in the expansion. In particular, it is proved that there is no log term involved. Moreover, as one can always expand $a_{jk0}$ in the trigonometric series
\begin{equation*}
	a_{jk0}(\theta)\sim \sum_{l} a_l \cos l\theta + b_l \sin l\theta,
\end{equation*}
Theorem \ref{thm:main3} also puts restrictions on the possible value of $k$ and $l$. 

As shown in the proof of Theorem \ref{thm:main3}, this refined information of expansion is related to both the nature of the singularity and the structure of the equation. In a forthcoming paper \cite{yin2016}, the authors show that a more complicated nonlinear structure of the equation leads to more singular terms in the expansion. In particular, we shall see $\log$ terms there.

The paper is organized as follows. In Section \ref{sec:weaksolution}, we define the function space in which the weak solution lies, construct the weak solution via approximation and prove various estimates for it. The estimates presented in this section are of the kind that the control over the solution is inherited from the corresponding property of the initial data. In Section \ref{sec:smoothing}, we show how a linear parabolic equation can `create' regularity. These results are key to the proof of higher regularity later. The results in these two sections are somewhat independent from the application. A more detailed introduction is included at the beginning of each section. The rest three sections are devoted to the proof of three main theorems respectively.

To conclude the introduction, we discuss the conventions and the notations used throughout the entire paper. We assume without loss of generality that there is only one singular point of order $\beta$, which is denoted by $p$. Around $p$, there is a conformal coordinate system $(x,y)$ of $S$, which is defined for all $\set{x^2+y^2\leq 1}$. We fix a background metric $\tilde{g}$, which is smooth away from the singular point and is the standard cone metric
\begin{equation*}
	\tilde{g}= (x^2+y^2)^\beta (dx^2+dy^2) 
\end{equation*}
in the above mentioned coordinate neighborhood of $p$.
We use a tilde to indicate quantities related to this metric, such as $\tilde{\nabla}$, $\tilde{\triangle}$ and so on.  

By `polar coordinates', we mean $(\rho,\theta)$ defined by
\begin{equation*}
	x=r\cos \theta, \quad y=r\sin \theta
\end{equation*}
and
\begin{equation*}
	\rho=\frac{1}{\beta+1} r^{\beta+1}.
\end{equation*}
It is not hard to see that $\rho$ is the Riemannian distance to $p$ measured by $\tilde{g}$. Throughout this paper, we write $B_*$ for the subset $\set{(\rho,\theta)\,|\, \rho<*}$ and in the case that $*$ is $1$, we simply write $B$.

\begin{rem}
	This paper follows the basic strategy taken by the author's previous papers \cite{yin2010ricci, yin2013ricci}. The proof of the main theorem in \cite{yin2010ricci} contains a gap and therefore the local existence result claimed there (with a rather weak initial data) is not proved. Under stronger assumptions of the initial data, a local existence proof was given in \cite{yin2013ricci}. However, there are other problems in that paper. The second version of \cite{yin2013ricci} added an appendix fixing one gap caused by the compatibility condition of initial and boundary data of parabolic equations. Recently, another gap was found and the attempt to fix it directly motivates the definition of weak solution used in this paper. See Remark \ref{rem:gap} for details.

	To make things clear, we give a detailed discussion of the whole problem totally independent of \cite{yin2010ricci,yin2013ricci}. It turns out that many seemingly important technical problems in \cite{yin2013ricci} disappear naturally in this new exposition. This paper is not just a revised version of \cite{yin2013ricci} and much stronger results are obtained here.
\end{rem}

\vskip 1cm
\noindent
{\it Acknowledgment.}
A major part of this paper was finished during the author's visit in Warwick university in 2015. He would like to thank the Mathematics Institute for the wonderful working environment and Professor Topping for making it available to him.

\section{Weak solution of linear equation}
\label{sec:weaksolution}
In this section, we develop a weak solution theory of linear equation of the type
\begin{equation*}
	\partial_t u = a(x,t) \tilde{\triangle} u + b(x,t) u + f(x,t).
\end{equation*}
We will start in Section \ref{subsec:functionspace} with the definition of spaces of functions in which the weak solution lies. Briefly speaking, it is a combination of the H\"older space and the Sobolev space. More precisely, it is the usual $C^{k,\alpha}$ space away from the singularity and equipped with a H\"older norm weighted naturally by the distance to the singular point, while near the singularity, it is similar (a little stronger) to the Sobolev space $V_2^{1,0}$ used in the book \cite{ladyzhenskaya1968linear}. The definition works well with both the maximum principle and the energy estimates.

In Section \ref{subsec:approximation}, we construct a weak solution via approximation by surfaces with boundary using the estimates proved in Section \ref{subsec:bvp}. In Section \ref{subsec:maximum}, we prove a maximum principle, which extend the estimates in Theorem \ref{thm:linear} obtained for the weak solution obtained by approximation to other weak solutions. Finally, in Section \ref{subsec:timederivative}, we prove stronger estimates for linear equation with stronger assumptions on the initial data and the coefficients of the equation. This estimate is to be used in the proof of local existence of the Ricci flow.

\subsection{Functions spaces}\label{subsec:functionspace}
We start by recalling some weighted H\"older space defined in \cite{yin2010ricci}. The elliptic version is $C^{l,\alpha}(S,\beta)$ and the parabolic version is $C^{l,\alpha}( (S,\beta)\times [0,T])$. In this paper, for simplicity, we denote them by $\mathcal E^{l,\alpha}$ and $\mathcal P^{l,\alpha,[0,T]}$ respectively. 

Let $(\rho,\theta)$ be the polar coordinates defined in the introduction and $U$ be a neighborhood of $\overline{S\setminus B}$. We define the $\mathcal E^{l,\alpha}$ norm to be 
\begin{equation*}
	\norm{f}_{\mathcal E^{l,\alpha}}= \sup_{k=0 \ldots \infty} \norm{f(2^{-k} \rho,\theta)}_{C^{l,\alpha}(B_1\setminus B_{1/2})} + \norm{f}_{C^{l,\alpha}(U)}.
\end{equation*}
Here $B_r$ is $\set{(\rho,\theta)\in B \,|\, \rho<r}$, $C^{l,\alpha}(B_1\setminus B_{1/2})$ and $C^{l,\alpha}(U)$ are just the usual H\"older norms. Similarly, if $f$ is a function defined on space time $S\times [0,T]$, we define (by regarding $f$ as a function of $(\rho,\theta,t)$ in $B$)
\begin{equation*}
	\norm{f}_{\mathcal P^{l,\alpha,[0,T]}}= \sup_{k=0 \ldots \infty} \norm{f(2^{-k} \rho,\theta,4^{-k} t)}_{C^{l,\alpha}(B_1\setminus B_{1/2}\times [0,4^k T])} + \norm{f}_{C^{l,\alpha}(U\times [0,T])}.
\end{equation*}
To see that the definition is independent of our choice of coordinate system $(x,y)$, we refer to Section 2 of \cite{yin2010ricci}.

In spite of the tedious definition, it is not difficult to understand the meaning of these weighted H\"older space. Away from the singularity, they are just the normal H\"older space. Near a singularity, the $\mathcal E^{l,\alpha}$ norm is the bound for up to $l-$th derivatives which one may obtain for a bounded harmonic function via applying the interior estimate on a ball away from the singularity. A similar characterization is true for $\mathcal P^{l,\alpha,[0,T]}$ if we replace the harmonic function by a solution to the linear heat equation defined on $S\times [0,T]$.

These spaces are too weak for a useful discussion of the Ricci flow equation because they contain almost no information at the singular point. We define stronger function spaces by requiring some integral norm to be bounded. For time-independent functions, we define $\mathcal W^{l,\alpha}$ to be the set of function $u$ in $\mathcal E^{l,\alpha}$ satisfying
\begin{equation*}
	\abs{u}_{\mathcal W}:=\left( \int_S \abs{\tilde{\nabla} u}^2 d\tilde{V}\right)^{1/2} < +\infty.
\end{equation*}
$\mathcal W^{l,\alpha}$ is a Banach space equipped with the norm
\begin{equation*}
	\norm{u}_{\mathcal W^{l,\alpha}} := \norm{u}_{\mathcal E^{l,\alpha}} + \abs{u}_{\mathcal W}.
\end{equation*}

For functions defined on $S\times [0,T]$, we define $\mathcal V^{l,\alpha,[0,T]}$ to be the set of function $u$ in $\mathcal P^{l,\alpha,[0,T]}$ satisfying 
\begin{equation*}
	\abs{u}_{\mathcal V^{[0,T]}}:=\max_{t\in [0,T]} \abs{u(t)}_{\mathcal W} +  \left( \iint_{S\times [0,T]} \abs{\partial_t u}^2 dtd\tilde{V} \right)^{1/2}< +\infty.
\end{equation*}
The norm of $\mathcal V^{l,\alpha,[0,T]}$ is defined to be
\begin{equation*}
	\norm{u}_{\mathcal V^{l,\alpha,[0,T]}}= \norm{u}_{\mathcal P^{l,\alpha,[0,T]}} + \abs{u}_{\mathcal V^{[0,T]}}.
\end{equation*}

	For later reference, we need the following variations of $\mathcal W^{l,\alpha}$ and $\mathcal V^{l,\alpha,[0,T]}$:
	\begin{itemize}
		\item 
	If $\Omega$ is a domain in $S$ containing $p$, or a part of the infinite cone $\Real^+\times S^1$ with standard cone metric $\tilde{g}= d\rho^2 + \rho^2(\beta+1)^2 d\theta^2$ containing the cone tip, we can define $\mathcal W^{l,\alpha}(\Omega)$ and $\mathcal V^{l,\alpha,[0,T]}(\Omega)$ similarly. 
\item
	For any open interval $(t_1,t_2)$, $\mathcal V^{l,\alpha,(t_1,t_2)}$ is the set of function $f$, which belongs to $\mathcal V^{l,\alpha,[t_1+\delta,t_2-\delta]}$ for any $\delta>0$. Similar convention holds for $[t_1,t_2)$ and $(t_1,t_2]$.
	\end{itemize}

\begin{defn}\label{defn:weak}
	Suppose a function $u$ is defined on $S\times [t_1,t_2]$ which is at least $C^2$ away from the singularity. If it satisfies some linear equation or the Ricci flow equation classically away from the singularity and $\abs{u}_{\mathcal V^{[0,T]}}$ is finite, then we call it a 'weak' solution.
\end{defn}

In the rest of this section, we discuss the implications of being a weak solution. The next two lemmas show the advantage of having $\abs{u}_{\mathcal W}$ and $\abs{u}_{\mathcal V^{[0,T]}}$ bounded respectively. 

\begin{lem}
	\label{lem:basic}
	Suppose that $u$ and $v$ are two functions in $\mathcal W^{2,\alpha}$. If $\tilde{\triangle} u$ is integrable on $S$, or $v\cdot \tilde{\triangle} u$ is bounded from below (or above) by some integrable function, then
	\begin{equation*}
		\int_{S} {\tilde\triangle} u=0 \quad \mbox{and} \quad \int_{S} v\cdot{\tilde \triangle u} =-\int_{S} \tilde\nabla u\cdot \tilde\nabla v.
	\end{equation*}
\end{lem}

\begin{proof}
	Since the first inequality follows from the second one by letting $v\equiv 1$, we prove the second one only.
	By the definition of $\mathcal W^{2,\alpha}$, there is small $\delta>0$ with
	\begin{equation*}
    \int_0^\delta \int_{\partial B_\rho} \abs{\tilde{\nabla} u}^2 d\sigma d\rho <\infty.
  \end{equation*}
  Recall that here $B_\rho$ is the ball of radius $\rho$ centered at the singularity measured with respect to the cone metric $\tilde{g}$.
  For any $\varepsilon>0$, we claim that there is a sequence $\rho_i$ going to zero such that
  \begin{equation}\label{eqn:rhoi}
	  \int_{\partial B_{\rho_i}} \abs{\tilde{\nabla} u}^2 d\sigma \leq \frac{\varepsilon}{\rho_i}.
  \end{equation}
  If the claim is not true, then there exists some $\varepsilon>0$ and $\delta_0>0$ such that for any $\rho\leq \delta_0$ ,
  \begin{equation*}
	  \int_{\partial B_\rho} \abs{\tilde{\nabla} u}^2 d\sigma>\frac{\varepsilon}{\rho},
  \end{equation*}
  which is a contradiction to the finiteness of $\int_{B_{\delta_0}} \abs{\tilde{\nabla} u}^2 d\tilde{V}$.

  For each $\rho_i$, the integration by parts gives
  \begin{equation}\label{eqn:ibp}
	  \int_{S\setminus B_{\rho_i}} v \cdot \tilde{\triangle} u + \tilde{\nabla} u \cdot \tilde{\nabla} v d\tilde{V} = \int_{\partial B_{\rho_i}} v\cdot \pfrac{u}{\nu} d\sigma,
  \end{equation}
  where $\nu$ is the outward normal vector to $\partial B_{\rho_i}$.

  By \eqref{eqn:rhoi} and the H\"older inequality, we get
  \begin{equation*}
	  \int_{\partial B_{\rho_i}} \abs{\tilde{\nabla} u}d\sigma \leq C \left( \int_{\partial B_{\rho_i}} \abs{\tilde{\nabla} u}^2 d\sigma\right)^{1/2} \rho_i^{1/2}<\varepsilon.
  \end{equation*}
  Together with \eqref{eqn:ibp} and the boundedness of $v$, it implies that
  \begin{equation*}
	  \lim_{i\to \infty} \abs{\int_{S\setminus B_{\rho_i}} v \cdot \tilde{\triangle} u + \tilde{\nabla} u \cdot \tilde{\nabla} v d\tilde{V}} \leq \varepsilon.
  \end{equation*}
  The assumptions of the lemma imply that the integrand $v\cdot \tilde{\triangle} u + \tilde{\nabla} u\cdot \tilde{\nabla} v$ is bounded from below (or above) by some integrable function, which allows us to conclude that
  \begin{equation*}
	 \int_{S} v \cdot \tilde{\triangle} u + \tilde{\nabla} u \cdot \tilde{\nabla} v d\tilde{V} =0.
  \end{equation*}
\end{proof}

\begin{lem}
	\label{lem:basic2}
	Suppose that $u\in \mathcal P^{2,\alpha,[0,T]}$ and that
	\begin{equation}\label{eqn:basica}
		\iint_{S\times [0,T]} \abs{\partial_t u} d\tilde{V} ds < +\infty.
	\end{equation}
	Then for any $C^1$ functions $\Psi:\Real\to \Real$ and $\varphi(x,t):S\times [0,T]\to \Real$ satisfying
	\begin{equation}\label{eqn:basicb}
		\max_{S\times[0,T]} \abs{\varphi}+\abs{\partial_t \varphi}< +\infty,
	\end{equation}
	$\int_S \Psi(u) \varphi d\tilde{V}$ is an absolutely continuous function of $t$ and 
	\begin{equation}\label{eqn:basic2}
		\frac{d}{dt}\int_S \Psi(u)\varphi d\tilde{V} = \int_S \Psi'(u) \partial_t u \varphi d\tilde{V} + \int_S \Psi(u) \partial_t \varphi d\tilde{V}
	\end{equation}
	for almost every $t$. In particular, the result holds for $u\in \mathcal V^{2,\alpha,[0,T]}$.
\end{lem}
\begin{proof}
For any $0\leq t_1<t_2\leq T$, the Fubini theorem implies that
\begin{eqnarray*}
	&& \int_{t_1}^{t_2} \int_S \Psi'(u) \partial_t u \varphi + \Psi(u) \partial_t \varphi d\tilde{V} dt \\
	&=&   \int_S \left( \int_{t_1}^{t_2} \partial_t (\Psi(u)\varphi(x,t)) dt \right) d\tilde{V} \\
	&=&  \left. \int_S \Psi(u)\varphi(x,t) d\tilde{V}\right|_{t_1}^{t_2},
\end{eqnarray*}
where the first line is absolutely integrable by \eqref{eqn:basica}, \eqref{eqn:basicb} and the boundedness of $u$.
This implies that $\int_S \Psi(u)\varphi(x,t) d\tilde{V}$ is absolutely continuous and \eqref{eqn:basic2} holds.
\end{proof}

\begin{rem}
	\label{rem:gap} In \cite{yin2013ricci}, when $u$ is a solution to the linear parabolic equation or the Ricci flow equation, $\int_S (u_+)^2 d\tilde{V}$ is taken as a differentiable function of $t$ for many times without further justification. By Lemma \ref{lem:basic2}, this is not a problem as long as one has $\iint_{S\times [0,T]} \abs{\partial_t u}^2 d\tilde{V}dt< \infty$. Since this quantity appears naturally in the energy estimate of linear parabolic equations, we add it to the definition of weak solution.
\end{rem}

\subsection{Estimates of boundary value problems}\label{subsec:bvp}
In this section, $M$ is a compact surface with nonempty boundary and a Riemannian metric $\tilde{g}$. The reuse of $\tilde{g}$ is not likely to cause confusion because when we apply the results of this section, $(M,\tilde{g})$ will be a part of $(S,\tilde{g})$. We prove two apriori estimates for the $C^{2,\alpha}$ solutions of linear parabolic equations on $M$. One of them is a $C^0$ estimate implied by the maximum principle and the other is an inequality involving the integrals used in Definition \ref{defn:weak}. The key point is that these estimates are independent of the geometry of $\partial M$. In the next section, we consider a sequence of surfaces with boundary which approximates the conical surfaces and the geometry of this sequence is not uniform.

Consider the linear boundary value problem
\begin{equation}
\left\{
	\begin{array}[]{ll}
		\partial_t u = a(x,t) \tilde{\triangle} u + b(x,t) u + f(x,t) & \quad \mbox{on } M\times [0,T] \\
		u(0)= u_0 &\quad \mbox{on } M\\
		\partial_\nu u =0 &\quad \mbox{on } \partial M.
	\end{array}
	\right.
	\label{eqn:Sk}
\end{equation}

We also require a compatibility condition on the initial and boundary data.
\begin{equation}
	\partial_\nu u_0=0 \quad \mbox{on } \partial M.
	\label{eqn:compatible}
\end{equation}

\begin{prop}\label{prop:bvp}
	For $a,b,f$ in $C^\alpha(M\times [0,T])$ with $0<\lambda<\min_{M\times [0,T]} a$ and $u_0\in C^{2,\alpha}(M)$ satisfying \eqref{eqn:compatible}, there is a unique $u(x,t)$ in $C^{2,\alpha}(M\times [0,T])$ solving \eqref{eqn:bvp} such that for $t\in [0,T]$,
	\begin{equation}
		\norm{u(t)}_{C^0(M)}\leq e^{C_1 t} \left( \norm{u_0}_{C^0(M)} + \int_0^t e^{-C_1 s}  C_2 ds\right),
		\label{eqn:Mc0}
	\end{equation}
	where $C_1=\norm{b}_{C^0(S\times [0,T])}$ and $C_2=\norm{f}_{C^0(S\times [0,T])}$,
	and
\begin{equation} 
	\label{eqn:Menergy} 
	 \int_M \abs{\tilde{\nabla} u}^2(t) d\tilde{V} + \int_0^t \int_M \abs{\partial_t u}^2 dsd\tilde{V}
	\leq  C_3 \int_M \abs{\tilde{\nabla} u_0}^2 d\tilde{V} + C_4 t,
\end{equation} 
where $C_3$ depends on $\norm{a}_{C^0(S\times [0,T])}$ and $C_4$ depends on the $C^0$ norm of $a, b, f, u$ and $\lambda$.
\end{prop}

\begin{proof}
	The existence and uniqueness of the solution $u$ in $C^{2,\alpha}(M\times [0,T])$ is well known (see Theorem 5.18 of \cite{lieberman1996second}). 
	For \eqref{eqn:Mc0}, we notice that the right hand side of \eqref{eqn:Mc0} is the solution of the ODE
	\begin{equation}
		\label{eqn:bvpode}
		\frac{dh}{dt}= C_1 h+C_2\quad \mbox{and} \quad h(0)=\norm{u_0}_{C^0(M)}.
	\end{equation}
	The proof of \eqref{eqn:Mc0} then reduces to the claim that $\abs{u}\leq h$. Subtracting \eqref{eqn:bvpode} from the equation of $u$ yields 
	\begin{equation*}
		\partial_t(u-h) = a(x,t) \tilde{\triangle} (u-h) + b(x,t)u - C_1 h + f-C_2.
	\end{equation*}
	By the definition of $C_1$ and $C_2$ and the fact that $h\geq 0$, we have
	\begin{equation*}
		\partial_t(u-h) \leq a(x,t) \tilde{\triangle} (u-h) + b(x,t) (u-h).
	\end{equation*}
	Moreover, $\partial_\nu(u-h)|_{\partial M}=0$. Since $u_0\leq h(0)$, the classical maximum principle for the linear parabolic equation on manifolds with boundary gives $u\leq h$ on $M\times [0,T]$. The lower bound is proved similarly.
	For the proof of \eqref{eqn:Menergy}, we consider the weighted integral
	\begin{equation*}
		\iint_{M\times [0,t]} \abs{\partial_t u}^2  a^{-1} ds d\tilde{V} 
		= \iint_{M\times [0,t]} \partial_t u \left( \tilde{\triangle} u + a^{-1} b u + a^{-1} f \right) ds d\tilde{V}
	\end{equation*}
	By the Young's inequality, there is a constant $C$ depending on $C^0$ norm of $b,f,u$ and $a^{-1}$ such that
	\begin{equation}\label{eqn:young}
		\frac{1}{2}\iint_{M\times [0,t]} \abs{\partial_t u}^2  a^{-1} ds d\tilde{V} 
		\leq \iint_{M\times [0,t]} \partial_t u \tilde{\triangle} u ds d\tilde{V} + C t. 
	\end{equation}
	Using integration by parts and the Fubini's Theorem, we have
	\begin{eqnarray}\label{eqn:iby}
		&&\iint_{M\times [0,t]} \partial_t u \tilde{\triangle} u ds d\tilde{V} \\ \nonumber  
		&=& -\int_0^t \int_M \tilde \nabla \partial_t u \cdot \tilde{\nabla} u d\tilde{V} ds \\ \nonumber
		&=& - \int_M \int_0^t \frac{1}{2} \partial_t \abs{\tilde{\nabla}u}^2 ds \tilde{V}\\\nonumber
		&=& \frac{1}{2} \int_M \abs{\tilde{\nabla }u}^2(0) d\tilde{V} -  \frac{1}{2} \int_M \abs{\tilde{\nabla }u}^2(t) d\tilde{V}.
	\end{eqnarray}
	The boundary term of integration by parts vanishes because of the boundary condition $\partial_\nu u|_{\partial M} =0$. The computation above involves some higher derivative of $u$ which does not exist for a function in $C^{2,\alpha}$. However, the first line and the last line make perfect sense for $C^{2,\alpha}$ functions. Therefore, the computation can be justified by smooth approximations. \eqref{eqn:Menergy} follows from \eqref{eqn:young} and \eqref{eqn:iby}
\end{proof}

\subsection{Weak solution of the linear equation via approximations}\label{subsec:approximation}
In this section, we construct a weak solution to the linear parabolic equation on conical surfaces and prove some estimates. The idea is to approximate $(S,\tilde{g})$ by a sequence of surfaces with boundary, solve a sequence of linear parabolic equations with Neumann boundary condition and take the limit of the sequence of solutions.

\begin{rem}
	Note that this is the only place in the paper where we use this approximation method. This is different from \cite{yin2013ricci} where the same type of approximation was used again and again. The idea is that all special properties we need are coded into the definition of weak solution, i.e. the finiteness of $\abs{\cdot}_{\mathcal V^{[0,T]}}$, which will be the starting point of many later discussions.
\end{rem}

\begin{thm}
	\label{thm:linear}
	For $a,b,f$ in $\mathcal V^{0,\alpha,[0,T]}$ with $0<\lambda<\min_{S\times[0,T]} a$ and $u_0\in \mathcal W^{2,\alpha}$, there exists a weak solution $u\in \mathcal V^{2,\alpha,[0,T]}$ to 
	\begin{equation}\label{eqn:linear1}
		\partial_t u = a(x,t)\tilde{\triangle} u + b(x,t) u + f(x,t)
	\end{equation}
	with $u(0)=u_0$ such that 
	\begin{equation*}
		\norm{u}_{\mathcal V^{2,\alpha,[0,T]}}\leq C
	\end{equation*}
	for a constant $C$ depending on $\lambda$, $\mathcal V^{0,\alpha,[0,T]}$ norms of $a,b,f$ and $\norm{u_0}_{\mathcal W^{2,\alpha}}$.
Moreover, we have
\begin{equation}
		\norm{u(t)}_{C^0(S)}\leq e^{C_1 t} \left( \norm{u_0}_{C^0(S)} +  \int_0^t e^{-C_1 s}  C_2 ds\right),
	\label{eqn:linear1c0}
\end{equation}
where $C_1=\norm{b}_{C^0(S\times [0,T])}$ and $C_2=\norm{f}_{C^0(S\times [0,T])}$, and
\begin{equation}
	\abs{u}_{\mathcal V^{[0,T]}} \leq C_3 \int_S \abs{\tilde{\nabla} u_0}^2 d\tilde{V} +  T C_4,
	\label{eqn:linear1energy}
\end{equation}
where $C_3$ depends on the $\norm{a}_{C^0(S\times [0,T])}$ and $C_4$ depends on the $C^0(S\times [0,T])$ norm of $a,b,f,u$ and $\lambda$.
\end{thm}

Recall that we have assumed that $p$ is the only singular point in $S$ and $(\rho,\theta)$ is the polar coordinates around $p$.  Define 
\begin{equation*}
	S_k:=S \setminus \set{(\rho,\theta)|\, \rho<\frac{1}{k}}.
\end{equation*}
The restriction of $\tilde{g}$ to $S_k$ is still denoted by $\tilde{g}$.

To use the linear estimate for surfaces with boundary in Section \ref{subsec:bvp}, the initial data of the approximating problem must satisfy the compatibility condition \eqref{eqn:compatible}. Therefore, we must modify the initial data so that it satisfies the Neumann boundary condition on $\partial S_k$ while keeping various norms of the modification under control.   

By the definition of $S_k$, we know that $(\rho,\theta)$ parametrizes a neighborhood of $\partial S_k$ for $\rho\in [\frac{1}{k},1)$ and $\theta\in S^1$. For any $\epsilon>0$ small, let $\eta_\epsilon: [1/k,1)\to [1/k,1)$ be some smooth increasing function satisfying:
	\begin{enumerate}[(a)]
		\item $\eta_\epsilon(\rho)=\frac{1}{k}$ for $\rho\in [1/k,1/k+\epsilon]$;
		\item $\eta_\epsilon(\rho)=\rho$ for $\rho\in [1/k+2\epsilon,1)$;
		\item $\eta_\epsilon'(\rho)\leq 3$ for all $\rho\in [1/k,1)$.
	\end{enumerate}
\begin{lem}
	For each $k$ fixed, there is $\epsilon$ so small that if $u_{0,k}:S_k\to \Real$ is obtained from $u_0$ by a modification near $\partial S_k$ given by
\begin{equation*}
	u_{0,k}(\rho,\theta)= u_0(\eta_\epsilon(\rho),\theta),
\end{equation*}
then
\begin{equation*}
	\norm{u_{0,k}}_{C^0(S_k)} \leq \norm{u_0}_{C^0(S)}
\end{equation*}
and 
\begin{equation*}
	\norm{\tilde{\nabla} u_{0,k}}_{L^2(S_k,\tilde{g})} \leq 2\norm{\tilde{\nabla} u_0}_{L^2(S,\tilde{g})}. 
\end{equation*}
Moreover, we may take $\epsilon\to 0$ as $k\to \infty$.
\end{lem}
\begin{proof}
	The control of $C^0$ norm is obvious. For the second inequality, we assume $\norm{\tilde{\nabla} u_0}_{L^2(S,\tilde{g})}>0$, otherwise there is nothing to prove. It suffices to show
	\begin{equation*}
		\norm{\tilde{\nabla} (u_{0,k}-u_0)}_{L^2(S_k,\tilde{g})} <  \norm{\tilde{\nabla} u_0}_{L^2(S,\tilde{g})}.
	\end{equation*}
	Notice that $\tilde{\nabla} (u_{0,k}-u_0)$ is supported in 
	\begin{equation*}
		\set{(\rho,\theta)\,|\, \frac{1}{k}\leq \rho\leq \frac{1}{k}+2\epsilon}
	\end{equation*}
	and $\abs{\tilde{\nabla} u_0}$, hence $\abs{\tilde{\nabla} u_{0,k}}$ is bounded (by a constant depending on $u_0$ and $k$). Hence, we can make $\norm{\tilde{\nabla} (u_{0,k}-u_0)}_{L^2(S_k,\tilde{g})}$ as small as we want by choosing $\epsilon$ small.
\end{proof}

Now we can define an initial-boundary value problem
\begin{equation}
	\left\{
		\begin{array}[]{ll}
			\partial_t u = a(x,t)\tilde{\triangle} u + b(x,t) u + f(x,t) &\quad \mbox{on } S_k\times [0,T] \\
			u(0)= u_{0,k} &\quad \mbox{on } S_k \\
			\partial_\nu u=0 & \quad \mbox{on }\partial S_k \times [0,T].
		\end{array}
	\right.
	\label{eqn:bvp}
\end{equation}

Proposition \ref{prop:bvp} implies that we have a unique solution $u_k\in C^{2,\alpha}(S_k\times [0,T])$ to \eqref{eqn:bvp} satisfying
	\begin{equation}
		\norm{u_k(t)}_{C^0(S_k)}\leq e^{C_1 t}\left( \norm{u_0}_{C^0(S)} +  \int_0^t e^{-C_1 s}  C_2 ds\right),
		\label{eqn:Skc0}
	\end{equation}
	where $C_1=\norm{b}_{C^0(S_k\times [0,T])}$ and $C_2=\norm{f}_{C^0(S_k\times [0,T])}$, and
	\begin{equation}
		\label{eqn:Skenergy}
		 \int_{S_k} \abs{\tilde{\nabla} u_k}^2(t) d\tilde{V} + \int_0^t \int_{S_k} \abs{\partial_t u_k}^2 dsd\tilde{V}
		\leq C_3 \int_{S_k} \abs{\tilde{\nabla} u_0}^2 d\tilde{V} + C_4 t,
	\end{equation}
	where $C_3$ depends on $\norm{a}_{C^0(S_k\times [0,T])}$ and $C_4$ depends on the $C^0$ norm of $a,b,f,u_k$ and $\lambda$.

	With \eqref{eqn:Skc0} and the Schauder estimate, we obtain some uniform estimate for $u_k$ on $\Omega\times [0,T]$ (for $k$ sufficiently large), for any fixed compact set $\Omega \subset S\setminus \set{p}$. These estimates imply that after taking subsequence if necessary, $u_k$ converges to a solution $u$ to \eqref{eqn:linear1} (defined on $S\times [0,T]$) and the convergence is in $C^{2,\alpha'}$ on $\Omega\times [0,T]$ for $\alpha'\in (0,\alpha)$. 

	\eqref{eqn:Skc0} and \eqref{eqn:Skenergy} pass on to the limit to give \eqref{eqn:linear1c0} and \eqref{eqn:linear1energy} respectively. Moreover, the $\mathcal P^{2,\alpha,T}$ norm of $u$ follows from the equation \eqref{eqn:linear1} and the Schauder interior estimates. This concludes the proof of Theorem \ref{thm:linear}.

\subsection{Maximum principle and the uniqueness of weak solution}\label{subsec:maximum}
In this section, we prove a maximum principle for the weak solution of linear equations and as a corollary, we show that the weak solution obtained in Theorem \ref{thm:linear} is unique among all weak solutions if the time derivative of $a$ is bounded.

The maximum principle proved here uses the so-called energy method, which is well known in the weak solution theory of parabolic equations. There are other types of maximum principles in the literature using barrier functions. One of them is due to Jeffres \cite{jeffres2000uniqueness} and we refer to Section 5 of \cite{jeffres2011k} and Lemma 11.4 in \cite{chen2015bessel}. The other is due to Chen and Wang \cite{chen2014long}, see Theorem 6.2 there.

\begin{lem}\label{lem:maximum}
	Suppose that $u\in \mathcal P^{2,\alpha,[0,T]}$ satisfies pointwisely on $S\setminus \set{p}\times [0,T]$
	\begin{equation*}
		\partial_t u \leq a(x,t) \tilde{\triangle} u + b(x,t) u,
	\end{equation*}
	whose coefficients satisfy
	\begin{equation}\label{eqn:maxab}
		0< \lambda< a(x,t)< \lambda^{-1} \quad \mbox{and} \quad \abs{b},\abs{\partial_t a} < \lambda^{-1}
	\end{equation}
	for some $\lambda>0$.
	Assume that
	\begin{equation}\label{eqn:assmax}
		\max_{t\in [0,T]} \norm{\tilde{\nabla} u}_{L^2(S,\tilde{g})} + \iint_{S\times [0,T]} \abs{\partial_t u} d\tilde{V} dt < \infty.
	\end{equation}
	If $u(0)\leq 0$, then $u(x,t)\leq 0$ for all $t\in [0,T]$.
\end{lem}

\begin{rem}
	For $u\in \mathcal V^{2,\alpha,[0,T]}$, the assumption \eqref{eqn:assmax} is trivially satisfied.
\end{rem}

\begin{proof}
	By setting $\Psi(z)=(\max\set{z,0})^2$ and $\varphi(x,t)= a^{-1}$, Lemma \ref{lem:basic2} implies that  $\int_S (u_+)^2 a^{-1} d\tilde{V}$ is absolutely continuous so that we can compute
	\begin{eqnarray*}
		\frac{d}{dt} \int_S (u_+)^2 a^{-1} d\tilde{V} &=&  \int_S 2 u_+ (\partial_t u) a^{-1}d\tilde{V} + \int_S (u_+)^2 \partial_t (a^{-1}) d\tilde{V} \\
		&\leq& \int_S 2 u_+ \tilde{\triangle} u d\tilde{V} + C \int_S (u_+)^2 d\tilde{V} \\
		&\leq & -2 \int_S  \abs{\tilde\nabla u_+}^2  d\tilde{V} + C \int_S (u_+)^2  a^{-1} d\tilde{V}. 
	\end{eqnarray*}
	Here in the above computation, we used Lemma \ref{lem:basic} to justify the integration by parts.
	\footnote{
		A little more effort is necessary here because $u_+$ is not in $\mathcal W^{2,\alpha}$. 
		
		We note that by Lemma 7.6 of \cite{trudinger1983elliptic} $\int_S \abs{\tilde{\nabla} u_+}^2 d\tilde{V}$ is finite, which implies that there exists a sequence of nonnegative $u_i\in \mathcal W^{2,\alpha}$ such that $u_i$ converges to $u_+$ in the sense that
		\begin{equation}\label{eqn:goodu}
			\lim_{i\to \infty} \sup_S \abs{u_i-u_+} + \int_S \abs{\tilde{\nabla} (u_i-u_+)}^2 d\tilde{V} =0.
		\end{equation}
		In fact, by setting $\rho=e^{s}$, we can regard $u_+$ as a function of $(s,\theta)$, which (by the definition of $\mathcal W^{2,\alpha}$) satisfies
		\begin{equation}\label{eqn:goodcylinder}
			\sup_{k\in \mathbb N} \norm{u_+(s-k,\theta)}_{C^{2,\alpha}([0,1]\times S^1)} + \int_{(-\infty,0]\times S^1} \abs{\partial_s u_+}^2 + \abs{\partial_\theta u_+}^2 dtd\theta <C.
		\end{equation}
	The standard mollification on the infinite cylinder $(-\infty,0]\times S^1$ gives nonnegative $u_i$, which approximates $u_+$ in the $C^0$ norm and the $L^2$ norm of gradient. It is exactly \eqref{eqn:goodu} if we regard $u_i$ as a function of $(\rho,\theta)$ (by the conformal invariance of the Dirichlet energy). Moreover, $u_i\in \mathcal W^{2,\alpha}$ because it also satisfies \eqref{eqn:goodcylinder}. By the Fubini theorem, \eqref{eqn:assmax} and the inequality satisfied by $u$, $\tilde{\triangle} u$ is bounded from below by some integrable function for almost every $t$.

	Since $u_i$ is nonnegative and uniformly bounded, we know $u_i\tilde{\triangle} u$ is bounded from below by some integrable function so that we can apply Lemma \ref{lem:basic} to get
	\begin{equation*}
		\int_S 2 u_i \tilde{\triangle} u + 2 \tilde{\nabla} u_i \cdot \tilde{\nabla} u d\tilde{V}=0.
	\end{equation*}
	Taking $i\to \infty$ yields
	\begin{equation*}
		\int_S 2 u_+ \tilde{\triangle} u + 2 \tilde{\nabla} u_+ \cdot \tilde{\nabla} u d\tilde{V}\leq 0.
	\end{equation*}
	Here we used the Fatou's lemma for the sequence $u_i \tilde{\triangle} u$, which is uniformly bounded from below by some integrable function.
}
	Since $\int_S (u_+)^2 a^{-1} d\tilde{V}$ is zero when $t=0$, it is zero for all $t\in [0,T]$.
\end{proof}

As a corollary of Lemma \ref{lem:maximum}, we have the following ODE comparison lemma.
\begin{lem}\label{lem:odecompare}
	Suppose that $u\in \mathcal P^{2,\alpha,[0,T]}$ satisfies pointwisely on $S\setminus \set{p}\times [0,T]$
	\begin{equation}\label{eqn:odeu}
		\partial_t u = a(x,t) \tilde{\triangle} u + b(x,t) u + f(x,t),
	\end{equation}
	whose coefficients satisfy \eqref{eqn:maxab}.
	If $u$ satisfies \eqref{eqn:assmax}, then for $t\in [0,T]$,
	\begin{equation*}
		\norm{u(t)}_{C^0(S)}\leq e^{C_1 t} \left( \norm{u(0)}_{C^0(S)} + \int_0^t e^{-C_1 s} \norm{f(s)}_{C^0(S)} ds \right)
	\end{equation*}
	where $C_1= \norm{b}_{C^0(S\times [0,T])}$.
\end{lem}
\begin{proof}
	The proof is similar to that of Proposition \ref{prop:bvp}. Let $h(t)$ be the solution to the ODE
	\begin{equation}\label{eqn:odeh}
		\frac{d}{dt}h= C_1 h(t) + \norm{f(t)}_{C^0(S)} \quad \mbox{and} \quad h(0)= \norm{u(0)}_{C^2(S)}
	\end{equation}
	and notice that $h(t)$ is exactly the right hand side of the inequality we want to prove.
	Subtracting \eqref{eqn:odeh} from \eqref{eqn:odeu} yields
	\begin{equation*}
		\partial_t (u-h) \leq a \tilde{\triangle} (u-h) + b(u-h).
	\end{equation*}
	Applying Lemma \ref{lem:maximum} to the above inequality shows that $u\leq h$ on $S\times [0,T]$. The other side of the inequality can be proved similarly.
\end{proof}

Next, we give a theorem about the uniqueness of the weak solution. Once we know a weak solution is `the' weak solution, then it coincides with the solution given by Theorem \ref{thm:linear} and satisfies some linear estimates.
\begin{thm}
	\label{thm:unique} Suppose that $u$ is a weak solution on $S\times [0,T]$ to the equation
	\begin{equation}\label{eqn:linear2}
		\partial_t u = a(x,t)\tilde{\triangle} u + b(x,t) u + f(x,t)
	\end{equation}
	with $u(0)=u_0\in \mathcal W^{2,\alpha}$. Assume that $a,b,f$ are in $\mathcal V^{0,\alpha,[0,T]}$ and 
	\begin{equation}\label{eqn:moreover}
		\max_{S\times [0,T]} \abs{\partial_t a}< \infty.
	\end{equation}
	Then $u$ is the same as the solution given in Theorem \ref{thm:linear}. In particular, \eqref{eqn:linear1c0} and \eqref{eqn:linear1energy} holds for $u$.
\end{thm}

The last part of the above theorem is very useful in later proofs of this paper. We use it to obtain $C^0$ and $\abs{\cdot}_{\mathcal V^{[0,T]}}$ estimates for weak solutions as long as \eqref{eqn:moreover} holds.

\begin{proof}
	Let $\tilde{u}$ be the weak solution to \eqref{eqn:linear2} given by Theorem \ref{thm:linear}. Then $u-\tilde{u}$ is a weak solution to the homogeneous equation
	\begin{equation*}
		\partial_t (u-\tilde{u}) =a(x,t) \tilde{\triangle} (u-\tilde{u}) + b(x,t)(u-\tilde{u})
	\end{equation*}
	with $(u-\tilde{u})(0)=0$. Applying Lemma \ref{lem:maximum} to both $u-\tilde{u}$ and $\tilde{u}-u$ proves the theorem. 
\end{proof}

For future use, we prove an estimate on the growth of $L^p$ norm for a weak solution to the linear equation.
	\begin{lem}
	\label{lem:fora}
	Suppose $a,b,f,u_0$ satisfy the same assumption as in Theorem \ref{thm:linear} and we further assume that
	\begin{equation}
		\max_{S\times [0,T]} \abs{\partial_t a}< \infty.
		\label{eqn:fora1}
	\end{equation}
	If $u$ is a weak solution to \eqref{eqn:linear1}, then for any $p>1$, we have
	\begin{equation*}
		\int_S \abs{u}^p(t) d\tilde{V} \leq C, \quad \forall t\in [0,T] 
	\end{equation*}
	where $C$ depends on $p$, $T$, the $C^0$ norm of $a, a^{-1},b,f,\partial_t a$ and $\int_S \abs{u_0}^p d\tilde{V}$.
\end{lem}
\begin{proof}
	We will split $u$ into three parts, $u^+$, $u^-$ and $u_f$ and estimate the $L^p$ norm of them separately. To see this, we let $u_f$ be the solution to
	\begin{equation*}
		\partial_t u_f = a \tilde{\triangle} u_f + b u_f + f
	\end{equation*}
	with $u_f(0)=0$ given by Theorem \ref{thm:linear}. For the initial data $u_0$, we write
	\begin{equation*}
		u_0= u_0^+ - u_0^-,
	\end{equation*}
	where $u_0^+$ and $u_0^-$ are two functions in $\mathcal W^{2,\alpha}$ such that they are bounded from below by some positive number and 
	\begin{equation*}
		\int_S \abs{u^{\pm}_0}^p d\tilde{V} \leq  \int_S \abs{u}^p d\tilde{V} +1.
	\end{equation*}
	We can obtain $u_0^\pm$ by splitting $u_0$ into its positive and negative parts, adding a small positive constant and smoothing them out if necessary. Theorem \ref{thm:linear} gives us $u^{\pm}(t)$ satisfying
	\begin{equation*}
		\partial_t u^\pm(t) = a \tilde{\triangle} u^\pm + b u^\pm
	\end{equation*}
	with $u^\pm(0)=u_0^\pm$. By Theorem \ref{thm:unique}, we know $u=u^+-u^-+u_f$. Moreover, $u_f$ is bounded on $S\times [0,T]$, so it remains to bound the $L^p$ norm of $u^{\pm}(t)$. Since the proof is the same, we give only the proof of $u^+$. It is a consequence of Lemma \ref{lem:maximum} that $u^+(t)$ is nonnegative. Direct computation shows away from $p$
	\begin{equation}\label{eqn:uplusp}
		\partial_t (u^+)^p \leq a \tilde{\triangle} (u^+)^p + p b (u^+)^p.
	\end{equation}
	By the fact that $u^+\in \mathcal V^{2,\alpha,[0,T]}$ and that $\partial_t a$ is bounded, we have
	\begin{equation} \label{eqn:uplust}
		\iint_{S\times [0,T]} \abs{\partial_t ( (u^+)^p a^{-1})} d\tilde{V} ds < +\infty,
	\end{equation}
	which justifies the following computation
	\begin{eqnarray*}
		\frac{d}{dt} \int_S (u^+)^p a^{-1} d\tilde{V} &=& \int_S \partial_t (u^+)^p a^{-1} + (u^+)^p \partial_t (a^{-1})d\tilde{V}  \\
		&\leq& \int_S \tilde{\triangle} (u^+)^p d\tilde{V} + C \int_S (u^+)^p a^{-1} d\tilde{V} \\
		&\leq&  C \int_S (u^+)^p a^{-1} d\tilde{V}.
	\end{eqnarray*}
	Here in the last line above, we used Lemma \ref{lem:basic} to see that $\int_S \tilde{\triangle} (u^+)^p d\tilde{V}$ vanishes.\footnote{In this case, $u^+$ and hence $(u^+)^p$ is in $\mathcal W^{2,\alpha}$ and by \eqref{eqn:uplusp} and \eqref{eqn:uplust}, $\tilde{\triangle} (u^+)^p$ is bounded from below by an integrable function for almost every $t$.} The control over $\int_S (u^+)^p d\tilde{V}$ follows from the above inequality by integration over time.
\end{proof}

\subsection{Estimates of the time derivative of solution}\label{subsec:timederivative}

In the study of PDE, when we assume that the coefficients in the equation and the initial boundary data are smoother (or have stronger estimates), we naturally expect that the solution is smoother(or have stronger estimates). This is usually proved by taking derivative of the equation. In the conical setting, there is no natural spacial derivatives to take near a singular point. However, we can still take the time derivative.

For example, if we take the time derivative of the equation
\begin{equation}
	\partial_t u = a(x,t) \tilde{\triangle} u +  f(x,t)
	\label{eqn:nob}
\end{equation}
and write $w$ for $\partial_t u$, we obtain
\begin{equation}
	\partial_t w = a \tilde{\triangle} w + \partial_t a \frac{w-f}{a} + \partial_t f.
	\label{eqn:dt}
\end{equation}

\begin{thm}
	\label{thm:stronglinear}
	In addition to the assumptions that $a,f\in \mathcal V^{0,\alpha,[0,T]}$, $u_0\in \mathcal W^{2,\alpha}$ and $\norm{a^{-1}}_{C^0(S\times [0,T])} \leq C$, which is assumed in Theorem \ref{thm:linear}, if $\partial_t a$, $\partial_t f$ are in $\mathcal V^{0,\alpha,[0,T]}$ and 
	\begin{equation*}
		w_0:=a(x,0)\tilde{\triangle} u_0 + f(x,0) \in \mathcal W^{2,\alpha},
	\end{equation*}
	then $\partial_t u$ is a weak solution to \eqref{eqn:dt}. Moreover, it is the unique weak solution to \eqref{eqn:dt} and 
	$\norm{\partial_t u}_{\mathcal V^{2,\alpha,[0,T]}}$ is bounded by a constant depending on the $\mathcal V^{0,\alpha,[0,T]}$ norm of $a,f,\partial_t a,\partial_t f$, the $\mathcal W^{2,\alpha}$ norm of $u_0$ and $w_0$ and $\norm{a^{-1}}_{C^0(S\times [0,T])}$. 
\end{thm}
\begin{proof}
	We start by observing that all the rest conclusions in Theorem \ref{thm:stronglinear} follow from the claim that $\partial_t u$ is a weak solution to \eqref{eqn:dt}. Since $\partial_t a$ is in $\mathcal V^{0,\alpha,[0,T]}$ (hence is bounded), Theorem \ref{thm:unique} applies directly to show $\partial_t u$ is the unique weak solution so that we have the bound of $\norm{\partial_t u}_{\mathcal V^{2,\alpha,[0,T]}}$ as given in Theorem \ref{thm:linear}.

	To see the claim holds, let $\tilde{w}$ be the weak solution to \eqref{eqn:dt} (given by Theorem \ref{thm:linear}) with the initial data $w_0$ and set
	\begin{equation}\label{eqn:utw}
		\tilde{u}(t)= u(0)+ \int_0^t \tilde{w}(s)ds.
	\end{equation}
	Theorem \ref{thm:stronglinear} is proved if we can show $u(t)=\tilde{u}(t)$, which implies that $\partial_t u$ (being the same as $\partial_t \tilde{u} = \tilde{w}$) is a weak solution to \eqref{eqn:dt}.

	The aim of showing $u(t)=\tilde{u}(t)$ is further reduced to checking that $\tilde{u}(t)$ is a weak solution to \eqref{eqn:nob}, because $u(t)$ and $\tilde{u}(t)$ will then be two weak solutions with the same initial data so that Theorem \ref{thm:unique} can be applied. Here we used the fact that $\partial_t a$ is bounded.

	The rest of the proof is devoted to proving $\tilde{u}(t)$ is a weak solution to \eqref{eqn:nob}. First, we show that it satisfies \eqref{eqn:nob} pointwisely away from the singularity. At $t=0$,
	\begin{equation}\label{eqn:H0}
		\partial_t \tilde{u}|_{t=0}= \tilde{w}(0)= a(x,0)\tilde{\triangle} u_0 + f(x,0)= a(x,0) \tilde{\triangle} \tilde{u}_0 + f(x,0).
	\end{equation}
	For $t>0$, we set
	\begin{equation*}
		H=\partial_t \tilde{u} - a \tilde{\triangle} \tilde{u} -f=\tilde{w}-a\tilde{\triangle} \tilde{u} -f
	\end{equation*}
	and compute
	\begin{eqnarray*}
		\partial_t H &=& \partial_t \tilde{w} -\partial_t a \tilde{\triangle} \tilde{u} - a \tilde{\triangle} \tilde{w} - \partial_t f\\
		&=&  \partial_t \tilde{w} -\partial_t a a^{-1} (\tilde{w}-f -H)- a \tilde{\triangle} \tilde{w} - \partial_t f\\
		&=& \partial_t a a^{-1} H.
	\end{eqnarray*}
	Here in the last line, we used the equation satisfied by $\tilde{w}$. Now, since $H(0)=0$ by \eqref{eqn:H0}, $H\equiv 0$ for $t\geq 0$.

	To see that $\tilde{u}$ is a weak solution, we notice first that
	\begin{equation}\label{eqn:tildeu2}
		\int_0^T \int_S \abs{\partial_t \tilde{u}}^2 d\tilde{V}ds = \int_0^T \int_S \abs{\tilde{w}^2} d\tilde{V}ds<\infty,
	\end{equation}
	because $\tilde{w}$ is bounded. For the Dirichlet energy bound of $\tilde{u}$, we take any compact domain $W$ in $S$ away from the singular point and use the dominated convergence theorem and the fact that $\tilde{\nabla} \tilde{w}$ is bounded on $W\times [0,t]$ to get
	\begin{equation*}
		\tilde{\nabla} \tilde{u} = \tilde{\nabla} u_0 + \int_0^t \tilde{\nabla} \tilde{w} \quad \mbox{in} \quad W.
	\end{equation*}
	The Schwartz inequality, the H\"older inequality and the Fubini theorem imply that
	\begin{equation}\label{eqn:tildeu1}
		\int_S \abs{\tilde{\nabla} \tilde{u}}^2 (t) d\tilde{V} \leq 2 \int_S \abs{\tilde{\nabla} u_0}^2 d\tilde{V} +  2 t \int_0^t \int_S \abs{\tilde{\nabla} \tilde{w}}^2 d\tilde{V} dt. 
	\end{equation}

	\eqref{eqn:tildeu1} and \eqref{eqn:tildeu2} together proves that $\tilde{u}$ is a weak solution to \eqref{eqn:nob} and hence concludes the proof of the theorem.
\end{proof}

\section{Smoothing effect of linear equation}\label{sec:smoothing}
The estimates proved in previous section are as good as the initial data. For example, a weak solution $u$ is bounded because the initial data $u_0$ is bounded and to show that $\partial_t u$ is in $\mathcal V^{2,\alpha,[0,T]}$, we need to assume that $\partial_t u|_{t=0}$ is in $\mathcal W^{2,\alpha}$. However, as is well known for the linear parabolic equation on smooth manifolds, rough initial data can be smoothed out. Of course, the regularity of the solution is still restricted by the regularity of the coefficients of the equation. In this section, we collect a few results in this direction for the linear parabolic equation on conical surfaces. They play essential roles in proving higher regularity of the conical Ricci flow.

\subsection{H\"older regularity for bounded solution}\label{subsection:calpha}
We first define the H\"older spaces on a conical surfaces. Recall that we have a background metric $\tilde{g}$ on $S$, hence, for any $x,y\in S$ (including the singular point), we have a well defined distance function $\tilde{d}(x,y)$, which is the infimum of the lengths of all smooth paths connecting $x$ and $y$. For any $\alpha\in (0,1)$, we define $C^\alpha(S)$ to be the set of bounded functions $u$ satisfying
\begin{equation*}
	\abs{u}_{C^\alpha(S)} := \sup_{x,y\in S} \frac{\abs{u(x)-u(y)}}{ \tilde{d}(x,y)^\alpha} < +\infty
\end{equation*}
and
\begin{equation*}
	\norm{u}_{C^\alpha(S)}:= \norm{u}_{C^0(S)} + \abs{u}_{C^\alpha(S)}
\end{equation*}
is defined to be the $C^\alpha$ norm. Similarly, we have a parabolic version for functions defined on $S\times [0,T]$ for some $T>0$. For $(x,t)$ and $(y,s)$ in $S\times [0,T]$, we define the space-time H\"older space $C^\alpha(S\times [0,T])$ to be the set of bounded function $u$ such that
\begin{equation*}
	\abs{u}_{C^{\alpha}(S\times [0,T])}:= \sup_{(x,t),(y,s)\in S\times [0,T]} \frac{\abs{u(x,t)-u(y,s)}}{ (\tilde{d}(x,y)+\sqrt{t-s})^{\alpha} }
\end{equation*}
and
\begin{equation*}
	\norm{u}_{C^\alpha(S\times [0,T])}:= \norm{u}_{C^0(S\times[0,T])} + \abs{u}_{C^\alpha(S\times[0,T])}.
\end{equation*}

\begin{rem}
	The definition above depends on $\tilde{g}$.
\end{rem}

Our first result is the following H\"older regularity result. Since it is a local result, we state it in a neighborhood of the singular point $p$. Of course, the result also holds in a ball away from the singularity. In fact, the proof given below is a modification of the proof in the smooth case which is well known. 

\begin{thm}\label{thm:dg}
	Suppose that $u\in \mathcal V^{2,\alpha,[0,T]}(B)$ is a weak solution to 
\begin{equation}
	\partial_t u = a(x,t) \tilde{\triangle} u +b(x,t)u+ f(x,t).
	\label{eqn:dg1}
\end{equation}
	If there is $C_1>0$ such that 
	\begin{equation*}
		a(x,t)+\frac{1}{a(x,t)}+ \abs{\partial_t a}+ \abs{b(x,t)} + \abs{f(x,t)}\leq C_1
	\end{equation*}
	on $B\times [0,T]$, then for any $\delta>0$, there is some $\alpha'>0$ depending on $\beta$ and $C^0$ norm of $a$ and $a^{-1}$ such that
	\begin{equation*}
		\norm{u}_{C^{\alpha'}(B_{1-\delta}\times [\delta,T])}\leq C_2
	\end{equation*}
	for some $C_2>0$ depending on $\delta, C_1$ and $C^0$ norm of $u$.
\end{thm}

It is closely related to the well known H\"older regularity result for linear parabolic equations of divergence form with bounded coefficients, for example, Section III.10 of \cite{ladyzhenskaya1968linear}, which is proved by Di Giorgi-Nash-Moser iteration. The idea is that in some natural coordinates, the linear equation here can be shown to have bounded (but not continuous) coefficients. This feature of conical singularity has been observed and utilized by Chen and Wang \cite{chen2014long} in the elliptic case. In the parabolic case, we have some extra difficulty caused by the fact that \eqref{eqn:dg1} is not of a divergence form. This is why we assume $\norm{\partial_t a}_{C^0(S\times [0,T])}$ is finite.

For the proof of Theorem \ref{thm:dg}, by using some natural coordinates, we first transform Theorem \ref{thm:dg} into Theorem \ref{thm:dg1} below which has nothing to do with the conical surface. For this purpose, let $(x_1,x_2)$ be $(\rho\cos \theta, \rho\sin \theta)$ and compute $\tilde{g}$ in terms of $(x_1,x_2)$ to see
\begin{equation*}
	\tilde{g}= \frac{1}{\rho^2} \left( (x_1^2+(\beta+1)^2 x_2^2) dx_1^2 - 2(\beta^2+2\beta)x_1x_2 dx_1dx_2 + (x_2^2+(\beta+1)^2 x_1^2) dx_2^2 \right).	
\end{equation*}
The observation is that these coefficients are bounded (not continuous at $(0,0)$). If we use $g_{ij}$ to denote the coefficients of $a^{-1}\tilde{g}$, we can rewrite (\ref{eqn:dg1}) in coordinates $(x_1,x_2)$ as
\begin{equation}\label{eqn:new}
	\partial_t u = \frac{1}{\sqrt{g}}\partial_i \left( g^{ij}\sqrt{g} \partial_j u \right) +bu + f.
\end{equation}
Here $g=\mathrm{det} g_{ij}$.
Since the identity map from $(B, \tilde{g}_{ij}dx_idx_j)$ to $(B,{\delta_{ij}}dx_idx_j)$ is a bi-Lipschitz map,\footnote{The best way to see this is to notice that in the polar coordinates,
	\begin{equation*}
		\tilde{g}_{ij} dx_i dx_j = d\rho^2 + \rho^2(\beta+1)^2 d\theta^2 \quad \mbox{and} \quad \delta_{ij} dx_i dx_j = d\rho^2 + \rho^2 d\theta^2.
	\end{equation*} } the assumption that $u$ is a weak solution becomes
\begin{equation}\label{eqn:weak}
	\max_{t\in [0,T]} \int_B \abs{\partial_i u}^2 dx + \int_0^T \int_B \abs{\partial_t u}^2 dx dt<  +\infty.
\end{equation}
Moreover, the H\"older space (and norm) defined with these two metrics are equivalent. In summary, to show Theorem \ref{thm:dg}, it suffices to prove

\begin{thm}
	\label{thm:dg1}
	Suppose that $u:C^{2,\alpha}(B\setminus \set{0}\times [0,T])\to \Real$ is a classical solution to \eqref{eqn:new} satisfying \eqref{eqn:weak}. Assume that the coefficients of \eqref{eqn:new} satisfy 
	\begin{enumerate}
		\item $\frac{1}{\lambda} \delta_{ij}\leq g_{ij}\leq \lambda \delta_{ij}$ for some $\lambda>0$;
		\item $\abs{\partial_t g}+\abs{b}+ \abs{f}\leq C_1$ for some $C_1$.
	\end{enumerate}
	If $u$ is bounded on $B\setminus \set{0} \times [0,T]$, then for each $\sigma>0$, there is $\alpha'>0$ depending only on $\lambda$ and $C_2$ depending on $\sigma$, $\lambda$, $T$ and $\norm{u}_{C^0(B\setminus\set{0}\times [0,T])}$ such that
	\begin{equation*}
		\norm{u}_{C^{\alpha'}(B_{1-\sigma}\times [\sigma,T])}\leq C_2.
	\end{equation*}
\end{thm}

This is almost a special case of Theorem 10.1 in Chapter III of \cite{ladyzhenskaya1968linear} except that the principal part of \eqref{eqn:new} is not a divergence form. We will show that this is not a problem if we assume $\partial_t g$ is bounded as above. The idea is that instead of multiplying the equation by some test function $v$, we multiply the equation by $(\sqrt{g})v$, which cancels the $(\sqrt{g})^{-1}$ in front of $\partial_i(g^{ij}\sqrt{g}\partial_j u)$ so that the integration by parts works (using \eqref{eqn:weak}) as if the equation is of the divergence form. Of cause, the price of doing so is an extra term involving $\partial_t(\sqrt{g})$, which is assumed to be bounded. Given this observation, some routine computation shows that the proof in \cite{ladyzhenskaya1968linear} still works. We shall give complete details to this in the appendix and for now, let's assume this theorem, hence Theorem \ref{thm:dg}.

As a corollary of the above result, we have the following Liouville type theorem.
\begin{lem}
	\label{lem:liouville}
	Let $\Real^+\times S^1$ be the infinite cone with metric $\tilde{g}=d\rho^2 + (\beta+1)^2 \rho^2 d\theta^2$. Suppose $u$ is a bounded solution to the standard heat equation 
	\begin{equation}\label{eqn:heat}
		\partial_t u =\triangle_{\tilde{g}} u
	\end{equation}
	on $(\Real^+\times S^1)\times (-\infty,0)$. If  
	\begin{equation*}
		\max_{t\in [-T,0]} \norm{\tilde{\nabla} u}_{L^2(\set{\rho<R},\tilde{g})} + \left( \iint_{\set{\rho<R}\times [-T,0]} \abs{\partial_t u}^2 d\tilde{V}dt \right)^{1/2} < +\infty
	\end{equation*}
	for any $R>0$ and $T>0$, then $u$ is a constant.
\end{lem}
\begin{proof}
	Set
	\begin{equation*}
		u_n(\rho,\theta,t)= u(n\rho,\theta,n^2 t).
	\end{equation*}
	$u_n$ satisfies \eqref{eqn:heat} and is uniformly bounded on $B\times [-1,0]$. Theorem \ref{thm:dg} then implies a uniform $C^{\alpha'}$ norm on $B_{1/2}\times [-1/2,0]$. This is impossible for large $n$ unless $u$ is a constant.
\end{proof}

The application of Theorem \ref{thm:dg} is restricted by the assumption $\partial_t a$ being bounded. In particular, in \eqref{eqn:rfu}, we have $a=e^{-2u}$ and therefore unless we know $\partial_t u$ is bounded, we should not assume $\partial_t a$ is bounded. Fortunately, we have
\begin{thm}
	\label{thm:dgspecial}
	Suppose that $u\in \mathcal V^{2,\alpha,[0,T]}(B)$ is a weak solution to
	\begin{equation}
		\partial_t u = e^{-2u} \tilde{\triangle}u + f(x,t)
		\label{eqn:special}
	\end{equation}
	with 
	\begin{equation*}
		\norm{f}_{C^0(B\times [0,T])}+ \norm{u}_{C^0}(B\times [0,T])\leq C_1.
	\end{equation*}
	Then for any $\sigma>0$, there is some $\alpha'>0$ depending on $C_1$ and $C_2$ depending on $C_1$ and $\sigma$ such that
	\begin{equation*}
		\norm{u}_{C^{\alpha'}(B_{1-\sigma}\times [\sigma,T])}\leq C_2.
	\end{equation*}
\end{thm}
\begin{proof}
	Setting $v=e^{2u}$, \eqref{eqn:special} becomes
	\begin{equation*}
		\partial_t v = \tilde{\triangle} \log v + f.
	\end{equation*}
	Taking $(\rho,\theta)$ coordinates, and setting $(x_1,x_2)$ as before, we can rewrite the above equation as
	\begin{equation*}
		\partial_t v = \frac{1}{\sqrt{\tilde{g}}} \partial_i \left(  \tilde{g}^{ij} \sqrt{\tilde{g}} \frac{1}{v} \partial_j v \right) + f.
	\end{equation*}
	This is only slightly different from \eqref{eqn:new} with an additional $1/v$. We still have $\lambda$ depending on $C^0$ norm of $u$ such that
	\begin{equation*}
		\frac{1}{\lambda}\delta_{ij}\leq \tilde{g}^{ij} \sqrt{\tilde{g}} \frac{1}{v} \leq \lambda \delta_{ij}
	\end{equation*}
	and $\partial_t \tilde{g}=0$. It is then evident from here that Theorem \ref{thm:dgspecial} follows from (the proof of) Theorem \ref{thm:dg1}\footnote{In fact, the actual structure of $g^{ij}\sqrt{g}$ in \eqref{eqn:new} is not important. What we need is only the fact that this matrix is comparable with $\delta_{ij}$.}.
\end{proof}

\subsection{Smoothing estimate for rough initial data}\label{subsec:rough}
The main result of this section strengthens Theorem \ref{thm:stronglinear} in the sense that we drop the assumption of $\tilde{\triangle} u_0$ there and show that the regularity and estimate of $\partial_t u$ remain true on $S\times [t_0,T]$ for any $t_0>0$.

\begin{thm}\label{thm:smoothing}
Suppose that 
\begin{enumerate}
	\item the $\mathcal V^{0,\alpha,[0,T]}$ norms of $a,b,f,\partial_t a$, $\partial_t b$ and $\partial_t f$ are bounded by $C_1$;
	\item $\norm{u_0}_{\mathcal W^{2,\alpha}}\leq C_2$;
	\item $a>\lambda>0$ on $S\times [0,T]$.
\end{enumerate}
Let $u$ be the weak solution to 
\begin{equation}\label{eqn:linear}
	\partial_t u = a(x,t)\tilde{\triangle} u + b(x,t) u +  f(x,t)
\end{equation}
with $u(0)=u_0$ given by Theorem \ref{thm:linear}. Then for each $t_0>0$, $w=\partial_t u$ is a weak solution to 
\begin{equation}
	\partial_t w = a \tilde{\triangle} w + (\partial_t a a^{-1} +b) w + \left( \partial_t f + \partial_t b u - \partial_t a a^{-1} (bu +f) \right) 
	\label{eqn:lineardt}
\end{equation}
on $[t_0,T]$ and that
\begin{equation*}
	\norm{\partial_t u}_{\mathcal V^{2,\alpha,[t_0,T]}} \leq C(C_1,C_2,\lambda,T,t_0).
\end{equation*}
\end{thm}

The proof consists of three steps. We state the goal of each step in the form of a lemma. Note that for each lemma, we assume the same assumptions as in Theorem \ref{thm:smoothing}. The first step is 

\begin{lem}
	\label{lem:lp}
	For $t_1=t_0/2>0$, there exists $q>1$ such that 
	\begin{equation*}
		\norm{\partial_t u (t_1)}_{L^q(S,\tilde{g})}\leq C (C_1,C_2,\lambda,t_1).
	\end{equation*}
	Moreover, by the H\"older inequality, we may always assume that $1<q<2$.
\end{lem}

The second step is to construct a solution to \eqref{eqn:lineardt} with initial data $\partial_t u(t_1)$, which lies in $\mathcal V^{2,\alpha,[t_1+\delta,T]}$ for all $\delta>0$. Since the initial data is only in $L^q$, we can not expect the $\mathcal W^{2,\alpha}$ norm of the solution to be bounded up to $t=t_1$. However, we shall have a growth estimate of it in terms of $t-t_1$. Precisely,
\begin{lem} \label{lem:growth}
	There is a solution $\tilde{w}$ to \eqref{eqn:lineardt} such that
	\begin{enumerate}
		\item for any $\delta>0$,
			\begin{equation*}
				\norm{\tilde{w}}_{\mathcal V^{2,\alpha,[t_1+\delta,T]}}\leq C (C_1,C_2,\lambda,T,t_1,\delta).
			\end{equation*}
		\item for $q>1$ in Lemma \ref{lem:lp}, 
			\begin{equation*}
				\norm{\tilde{w}(t)}_{C^0(S)} + \norm{\tilde{\nabla} \tilde{w}(t)}_{L^2(S,\tilde{g})} \leq \frac{C(C_1,C_2,\lambda,t_0)}{(t-t_1)^{1/q}}.
			\end{equation*}
	\end{enumerate}
\end{lem}

This solution $\tilde{w}$ satisfies all the requirements for $\partial_t u$ in Theorem \ref{thm:smoothing}. The last step for the proof of Theorem \ref{thm:smoothing} is
\begin{lem}
	\label{lem:same}
	$\tilde{w}$ constructed in Lemma \ref{lem:growth} is the same as $\partial_t u$ on $S\times [t_1,T]$.
\end{lem}

So the proof of Theorem \ref{thm:smoothing} reduces to the proof of Lemma \ref{lem:lp}, Lemma \ref{lem:growth} and Lemma \ref{lem:same}.

\begin{proof}[Proof of Lemma \ref{lem:lp}]
	$\partial_t u(t_1)$ is bounded away from the singular point $p$ by the interior estimate of \eqref{eqn:linear}. To prove Lemma \ref{lem:lp}, it suffices to consider the growth of $\tilde{\triangle} u(t_1)$ (near the singular point) in $B$. 
	By Theorem \ref{thm:dg}, we know $u(t_1)$ is $C^{\alpha'}(B)$ for some $\alpha'>0$, which means that in terms of polar coordinates
\begin{equation}\label{eqn:step11}
	\abs{u(\rho,\theta,t_1)-u(0,\theta,t_1)}\leq C \rho^{\alpha'}.
\end{equation}
For any $y=(\rho_0,\theta_0)$ ($\rho_0\ne 0$), there is constant $\sigma$ depending only on $\beta$ such that the geodesic ball $B_{\sigma \rho_0}(y)$ (with respect to $\tilde{g}$) is diffeomorphic to a disk embedded in $S\setminus \set{p}$ and $\tilde{g}$ restricted to $B_{\sigma \rho_0}(y)$ is flat. Let $(x_1,x_2)$ be the Euclidean coordinates in $B_{\sigma \rho_0}(y)$ with $x_i=0$ at $y$ and regard $u$ as a function of $x_1,x_2$ and $t$. Set
\begin{equation*}
	v(z_1,z_2,t)= u (\sigma \rho_0 z_1, \sigma \rho_0 z_2, (\sigma \rho_0)^2 t + t_1).
\end{equation*}
$v$ is defined on $B^z\times [-1,0]$ (here $B^z$ is the unit ball in $z$ plane) and satisfies
\begin{equation*}
	\partial_t v = a \triangle_z v + (\sigma \rho_0)^2 b v + (\sigma\rho_0)^2 f.
\end{equation*}

By the interior estimates of linear parabolic equations, we know that 
\begin{equation}\label{eqn:interpolation}
\norm{v(0)}_{C^{2,\alpha}(B^z_{1/2})}\leq C
\end{equation}
for some $C$ depending on the $\mathcal V^{0,\alpha,[0,T]}$ norm of $a,b,f$ and $\lambda$. \eqref{eqn:step11} implies that
\begin{equation}\label{eqn:osc}
	\mbox{osc}_{B^z} v(0)\leq C \rho_0^{\alpha'}.
\end{equation}
By \eqref{eqn:interpolation} and \eqref{eqn:osc}, the interpolation of H\"older norm\footnote{In case the reader needs a proof, we refer to Lemma \ref{lem:D2} proved in the appendix.} gives $\alpha''>0$ such that
\begin{equation*}
	\max_{B^z_{1/2}}\abs{\frac{\partial^2}{\partial z_i \partial z_j} v(0)}\leq C \rho_0^{\alpha''} \quad i,j=1,2
\end{equation*}
which implies that 
\begin{equation*}
	\abs{\tilde{\triangle}u (\rho_0,\theta_0,t_1)}\leq C \rho_0^{-2+\alpha''}.
\end{equation*}
Therefore, there is some $q>1$ such that $\tilde{\triangle}u(t_1)$ is in $L^q(S,\tilde{g})$.
\end{proof}

\begin{proof}[Proof of Lemma \ref{lem:growth}]
The problem we want to solve is a linear parabolic equation with nonzero non-homogeneous term and nonzero initial data. By the linearity, it suffices to solve the following two problems separately and add up the solutions:
\begin{equation}
	\left\{
		\begin{array}[]{ll}
			\partial_t w_1 = a\tilde{\triangle} w_1 + \tilde{b} w_1  + \tilde{f} & \quad \mbox{ on } S\times [t_1,T] \\
			w_1(0)=0 &
		\end{array}
		\right.
	\label{eqn:w1}
\end{equation}
and
\begin{equation}
	\left\{
		\begin{array}[]{ll}
			\partial_t w_2 = a\tilde{\triangle} w_2 + \tilde{b} w_2  & \quad \mbox{ on } S\times [t_1,T] \\
			w_2(0)= \partial_t u (t_1). &
		\end{array}
		\right.
	\label{eqn:w2}
\end{equation}
Here for simplicity, we have set
\begin{equation*}
	\tilde{b}=\partial_t a a^{-1} + b \quad \mbox{and} \quad \tilde{f} = \partial_t f + \partial_t b u - \partial_t a a^{-1} (bu +f).
\end{equation*}
Note that by the assumptions of Theorem \ref{thm:smoothing}, we know $\tilde{b},\tilde{f}\in \mathcal V^{0,\alpha,[t_1,T]}$. By Theorem \ref{thm:linear}, there exists a solution $w_1$ to \eqref{eqn:w1} satisfying
\begin{equation*}
	\norm{w_1}_{\mathcal V^{2,\alpha,[t_1,T]}}\leq C(C_1,C_2,\lambda,T),
\end{equation*}
so that to show Lemma \ref{lem:growth}, it suffices to find a solution $w_2$ to \eqref{eqn:w2} which satisfies (1) and (2) in Lemma \ref{lem:growth}.

The most natural way of solving \eqref{eqn:w2} is to consider an approximation to $\partial_t u(t_1)\in L^q(S,\tilde{g})$. For that purpose, let $w_{2,n,0}$ be a sequence of $\mathcal W^{2,\alpha}$ functions such that
\begin{equation*}
	\lim_{n\to \infty} \norm{w_{2,n,0}- \partial_t u(t_1)}_{L^q(S,\tilde{g})}=0
\end{equation*}
and for each compact domain $W\subset S\setminus \set{p}$
\begin{equation}\label{eqn:calphainitial}
	\lim_{n\to \infty} \norm{w_{2,n,0}- \partial_t u(t_1)}_{C^{2,\alpha}(W)}=0.	
\end{equation}
Note that in the last line above, we used the fact that $\partial_t u$ satisfies \eqref{eqn:lineardt} pointwisely away from the singular point so that $\partial_t u(t_1)$ is in $C^{2,\alpha}(W)$ by the Schauder estimate and the assumptions of Theorem \ref{thm:smoothing}. For each $w_{2,n,0}$, Theorem \ref{thm:linear} gives a weak solution $w_{2,n}(t)$ defined on $S\times [t_1,T]$ satisfying
\begin{equation*}
	\partial_t w_{2,n} = a\tilde{\triangle} w_{2,n} + \tilde{b} w_{2,n} \quad \mbox{and} \quad  w_{2,n}(t_0)= w_{2,n,0}.
\end{equation*}
We claim the following apriori estimates for $w_{2,n}(t)$:
\begin{enumerate}[(a)]
	\item For some $C>0$,
		\begin{equation}
			\max_{t\in [t_1,T]} \norm{w_{2,n}(t)}_{L^q(S,\tilde{g})}\leq C;
			\label{eqn:wnlp}
		\end{equation}

	\item For any compact domain $W\subset S\setminus \set{p}$, there is $C>0$ such that
		\begin{equation*}
			\norm{w_{2,n}}_{C^{2,\alpha}(W\times [t_1,T])}\leq C;
		\end{equation*}
	\item For any $t\in (t_1,T]$,
	\begin{equation}\label{eqn:smoothorder}
		\norm{w_{2,n}(t)}_{C^0(S)} + \norm{\tilde{\nabla} w_{2,n}(t)}_{L^2(S,\tilde{g})} \leq \frac{C}{(t-t_1)^{1/q}};
	\end{equation}
\item For any $\delta>0$, 
	\begin{equation*}
		\norm{w_{2,n}}_{\mathcal V^{2,\alpha,[t_1+\delta,T]}}\leq C(C_1,C_2,\lambda,T,t_1,\delta).
	\end{equation*}
\end{enumerate}
Before we start proving (a)-(d) above, we show how they imply Lemma \ref{lem:growth}. Recall that our aim is to find $w_2$ solving \eqref{eqn:w2} and check that (1) and (2) in Lemma \ref{lem:growth} holds for $w_2$ in the place of $\tilde{w}$.

The estimate (b) above and \eqref{eqn:calphainitial} imply that $w_{2,n}$ subconverges to a solution of \eqref{eqn:w2}. By taking the limit, the estimates (c) and (d) become (2) and (1) for $w_2$ in Lemma \ref{lem:growth}.

Now, let's turn to the proof of the claim, i.e. (a)-(d). Notice that (a) is a direct consequence of Lemma \ref{lem:fora}, (b) follows from (a) by applying known parabolic interior estimates to $w_{2,n}$ on $Q\times [t_1,T]$, where $W \subset Q\subset \bar{Q}\subset S\setminus \set{p}$ and (d) follows from (c), by regarding $w_{2,n}(t)$ as a weak solution to \eqref{eqn:w2} with initial data $w_{2,n}(t_1+\delta)$ (which by (c) belongs to $\mathcal W^{2,\alpha}$) and applying Theorem \ref{thm:unique}.

Before we prove (c), we describe a cover of ${S}$ by balls. For some $t$ fixed, there is a sequence of points $p_1,\cdots,p_N$ in $S$, where $p_1=p$ is the singular point such that the following is true. Let $B_1$ be $B_{\sqrt{t}}(p)$, the geodesic ball centered at $p_1=p$ with radius $\sqrt{t}$ and for $i=2,\cdots,N$, let $B_i$ be $B_{c_\beta \sqrt{t}}(p_i)$. Here $c_\beta\in (0,1)$ is a constant depending only on $\beta$ and the geometry of $(S,\tilde{g})$ so that each $B_i (i>1)$ is topologically a ball and $p\notin B_i$ for $i>1$. We also denote by $B^\lambda_i$ the geodesic ball centered at $p_i$ but with the radius being $\lambda$ times that of $B_i$. We require that $B^{1/4}_i$ cover $S$ and that each point in $S$ is covered by $\bar{B_i}$ for at most $c$ times for a universal constant $c$. 
We need the following lemma
\begin{lem}
	\label{lem:moser} Suppose that $a,b\in \mathcal V^{0,\alpha,[0,t]}(B_i)$ and that $u\in \mathcal V^{2,\alpha,[0,t]}(B_i)$ is a weak solution to
	\begin{equation*}
		\partial_t u =a(x,t) \tilde{\triangle} u + bu.
	\end{equation*}
	Assume 
	\begin{equation*}
		\max_{B_i\times [0,t]} \abs{\partial_t a} < C_1
	\end{equation*}
	for some $C_1$. Then there exists $C_2$ depending on $\mathcal V^{0,\alpha,[0,t]}$ norms of $a,b$ and $C_1$ such that
	\begin{equation}\label{eqn:moser1}
		\max_{B^{1/2}_i \times [t/2,t]} \abs{u} \leq {C_2} \frac{1}{t^{2/q}}\left( \int_0^t \int_{B_i} \abs{u}^q d\tilde{V} dt\right)^{1/q}
	\end{equation}
	and
	\begin{equation}\label{eqn:moser2}
		\norm{\tilde{\nabla} u(t)}_{L^2(B^{1/4}_i,\tilde{g})} \leq {C_2} \frac{1}{t^{2/q}} \left( \int_0^t \int_{B_i} \abs{u}^q d\tilde{V} dt\right)^{1/q}
	\end{equation}
\end{lem}
The proof \eqref{eqn:moser1} of this lemma is a Moser iteration, which is more or less standard, while the proof of \eqref{eqn:moser2} is a combination of \eqref{eqn:moser1} with Theorem \ref{thm:dg}. In order not to distract the readers from the proof of Theorem \ref{thm:smoothing}, we move it to the appendix. 

We apply Lemma \ref{lem:moser} to $w_{2,n}$ and see immediately from \eqref{eqn:moser1} and (a) that
\begin{equation*}
	\norm{w_{2,n}(t)}_{C^0(S)}\leq \frac{C_3}{(t-t_1)^{1/q}},
\end{equation*}
which is the $C^0$ part of \eqref{eqn:smoothorder}. For the other half, 
\begin{eqnarray*}
	\int_S \abs{\tilde{\nabla} w_{2,n}(t)}^2 d\tilde{V} &\leq& \sum_{i=1}^N \int_{B_i^{1/4}} \abs{\tilde{\nabla} w_{2,n}(t)}^2 d\tilde{V} \\
	&\leq& C^2_2 \sum_{i=1}^N \frac{1}{(t-t_1)^{4/q}} \left( \int_{t_1}^t \int_{B_i} \abs{w_{2,n}}^q d\tilde{V} dt \right)^{2/q} \\
	&\leq& \frac{C_4}{(t-t_1)^{4/q}} \left( \int_{t_1}^t \int_S \abs{w_{2,n}(t)}^q d\tilde{V} dt\right)^{2/q}.
\end{eqnarray*}
Here in the last line above, we used the elementary inequality that if $2/q>1$ (as we have assumed in Lemma \ref{lem:lp}), then for $x,y\in [0, \infty)$, we have $x^{2/q}+y^{2/q}\leq (x+y)^{2/q}$. Together with (a), this concludes the proof of part (c) of the claim and hence the proof of Lemma \ref{lem:growth}.
\end{proof}

Finally, we give a proof of Lemma \ref{lem:same}.
\begin{proof}[Proof of Lemma \ref{lem:same}]
For $t>t_1$, we set
\begin{equation*}
	\tilde{u}(t)= u(t_1)+ \int_{t_1}^t \tilde{w}(t) dt.
\end{equation*}
Obviously, $\tilde{u}(t_1)=u(t_1)$.
The proof consists of three steps. First, we show that $\tilde{u}-u$ satisfies some homogeneous equation with an error term, i.e. \eqref{eqn:tildeuu} below. Then we show that although $\tilde{u}-u$ is not a weak solution in $\mathcal V^{2,\alpha,[t_1,T]}$, there is still some control over its energy and time derivative when $t$ is close to $t_1$, so that in the final step, we can invoke an argument similar to the proof of Lemma \ref{lem:maximum} to show that $\tilde{u}=u$. 

{\bf Step 1.}  When $t=t_1$,
\begin{equation*}
	\partial_t \tilde{u}|_{t=t_1} = \tilde{w}(t_1)= \partial_t u(t_1)= a \tilde{\triangle} u +bu +f = a\tilde{\triangle} \tilde{u}+ b\tilde{u} +f.
\end{equation*}
For $t>t_1$, we set
\begin{equation}\label{eqn:defH}
	H=\partial_t \tilde{u}- a \tilde{\triangle} \tilde{u} -b\tilde{u} -f = \tilde{w}-a\tilde{\triangle} \tilde{u} -b\tilde{u} -f
\end{equation}
and compute 
\begin{eqnarray*}
	&& \partial_t H \\
	&=&\partial_t \tilde{w}  -\partial_t a \tilde{\triangle} \tilde{u}  -a \tilde{\triangle} \tilde{w} -\partial_tb \tilde u - b \tilde{w} -\partial_t f \\ 
	&=& \partial_t a a^{-1} H - \partial_t a \frac{\tilde{w} -b \tilde{u}-f}{a} + \partial_t \tilde{w} -a \tilde{\triangle} \tilde{w} -\partial_t b \tilde{u} - b\tilde{w} -\partial_t f \\ 
	&=& \partial_t a a^{-1} H + (\partial_t a a^{-1} b -\partial_t b ) (\tilde{u}-u).
\end{eqnarray*}
Here in the last line above, we used the fact that $\tilde{w}$ satisfies \eqref{eqn:lineardt}.
For each fixed $x\in S\setminus \set{p}$, consider $H$ as a function of $t$ alone satisfying the above ODE with initial data $H(0)=0$. There exists $C_4$ and $C_5$ depending on the $C^0$ norm of $\partial_t a, \partial_t b$ , $a^{-1}$ and $b$ such that
\begin{equation*}
	\abs{\partial_t H} \leq C_4 \abs{H} + C_5 \abs{\tilde{u}-u},
\end{equation*}
from which we obtain $C_6>0$ depending on $C_4,C_5$ and $T$ such that
\begin{equation}
	\abs{H}(t) \leq C_6 \int_{t_1}^t \abs{\tilde{u}-u} ds.
	\label{eqn:H}
\end{equation}
We then subtract \eqref{eqn:defH} and \eqref{eqn:linear} to get
\begin{equation}\label{eqn:tildeuu}
	\partial_t (\tilde{u}-u) = a\tilde{\triangle} (\tilde{u}-u) + b(\tilde{u}-u) +H.
\end{equation}
Intuitively, \eqref{eqn:H}, \eqref{eqn:tildeuu} and some maximum principle type argument shall prove $\tilde{u}=u$ for $t\in [t_1,T]$. However, there is some technical issue in the maximum principle type argument, which requires the following claim. 

{\bf Step 2.} We claim that $\tilde{u}$ satisfies
\begin{equation}
	\norm{\tilde{\nabla} \tilde{u}(t)}_{L^2(S,\tilde{g})} < \infty \quad \mbox{for } \, t>t_1 
	\label{eqn:tildeuenergy}
\end{equation}
and
\begin{equation}
	\int_{t_1}^T\int_S \abs{\partial_t \tilde{u}} d\tilde{V} dt < +\infty.
	\label{eqn:tildeudt}
\end{equation}
\eqref{eqn:tildeudt} follows from (2) of Lemma \ref{lem:growth} since $\partial_t \tilde{u}= \tilde{w}$. To see \eqref{eqn:tildeuenergy}, we take compact set $W\subset S\setminus \set{p}$. By (2) in Lemma \ref{lem:growth} and the definition of $\tilde{u}$, we know $\tilde{u}$ is bounded on $S\times [t_1,T]$, which in turn gives that $H$ is bounded on $S\times [t_1,T]$ by \eqref{eqn:H}. Since $\tilde{u}(t_1)-u(t_1)=0$, we can apply the usual $L^p$ estimate for $\tilde{u}-u$ to the equation \eqref{eqn:tildeuu} on $W'\times [t_1,T]$ to see that $\tilde{\nabla} \tilde{u}$ is bounded on $W\times [t_1,T]$. Here $W'$ is a compact domain in $S\setminus \set{p}$  whose interior set contains $W$. This together with (2) in Lemma \ref{lem:growth} shows that for any $t\in [t_1,T]$,
\begin{equation*}
	\int_{t_1}^t \int_W \abs{\tilde{\nabla} \tilde{u}} \abs{\tilde{\nabla} \tilde{w}} d\tilde{V} ds < \infty.
\end{equation*}
The Fubini theorem implies that $\int_W \abs{\tilde{\nabla} \tilde{u}}^2 d\tilde{V}$ is absolutely continuous function on $[t_1,T]$ and
\begin{eqnarray*}
	\frac{d}{dt}\int_W \abs{\tilde{\nabla} \tilde{u}}^2 d\tilde{V} &=&  2 \int_W \tilde{\nabla} \tilde{u} \cdot \tilde{\nabla} \tilde{w} d\tilde{V} \\
	&\leq& 2 \norm{\tilde{\nabla} \tilde{u}}_{L^2(W,\tilde{g})} \norm{\tilde{\nabla} \tilde{w}}_{L^2(W,\tilde{g})} \\ 
	&\leq & \frac{C}{ (t-t_1)^{1/q}} \norm{\tilde{\nabla} \tilde{u}}_{L^2(W,\tilde{g})}.
\end{eqnarray*}
Since we know $\norm{\tilde{\nabla} \tilde{u}(t_1)}_{L^2(W,\tilde{g})}\leq \norm{\tilde{\nabla} u(t_1)}_{L^2(S,\tilde{g})}\leq C < +\infty$, we can integrate the above differential inequality to get another constant $C'$ such that
\begin{equation*}
	\norm{\tilde{\nabla} \tilde{u}(t)}_{L^2(W,\tilde{g})}\leq C'
\end{equation*}
for all $t\in [t_1,T]$. By the arbitrariness of $W$, \eqref{eqn:tildeuenergy} is proved.

{\bf Step 3.} Let $F(t)$ be $\int_S (\tilde{u}-u)^2 a^{-1} d\tilde{V}$. \eqref{eqn:tildeudt} and Lemma \ref{lem:basic2} imply that $F$ is an absolutely continuous function of $t$ so that we can compute
\begin{eqnarray*}
	\frac{d}{dt} F(t) &\leq & 2\int_S (\tilde{u}-u) (\tilde{\triangle} (\tilde{u}-u) + a^{-1}b (\tilde{u}-u) + a^{-1}H) d\tilde{V} + C F(t)\\
	&\leq& C F(t)+ C \int_S H^2 d\tilde{V}.
\end{eqnarray*}
Here in the last line above, the integration by parts is justified by Lemma \ref{lem:basic}, \eqref{eqn:tildeuenergy} and the fact that $\tilde{\triangle} (\tilde{u}-u)$ is integrable for almost every $t\in [t_1,T]$, which is a consequence of \eqref{eqn:defH}, \eqref{eqn:H} and \eqref{eqn:tildeudt}.
On the other hand, by the H\"older's inequality and the Fubini's theorem,
\begin{eqnarray*}
	\int_S H^2 d\tilde{V} &\leq& C(T) \int_S \int_{t_1}^t (\tilde{u}-u)^2 ds d\tilde{V} \\
	&\leq & C(T) \int_{t_1}^t F(s) ds.
\end{eqnarray*}
In summary, we have
\begin{equation}\label{eqn:Fdt}
	\frac{d}{dt} F(t)\leq C F(t) + C \int_{t_1}^t F(s) ds
\end{equation}
and $F(t_1)=0$. One can conclude from the above inequality that $F\equiv 0$. To see this, set
\begin{equation*}
	\tilde{F}(t)=\max_{s\in [t_1,t]} F(s).
\end{equation*}
It turns out that $\tilde{F}$ is also absolutely continuous and 
\begin{equation*}
	\frac{d}{dt} \tilde{F} (t)\leq \max\set{0, \frac{d}{dt} F(t)} \leq \max \set{0, CF(t)+ C \int_{t_1}^t F(s)ds} \leq C \tilde{F}(t).
\end{equation*}
Integrating over $t$ and noticing that $\tilde{F}(t_1)=0$ gives $\tilde{F}\equiv 0$, which finishes the proof of Step 3, and hence Lemma \ref{lem:same}. 
\end{proof}

\section{Global existence of conical Ricci flow}\label{sec:global}
In this section, we prove Theorem \ref{thm:main1}. The local existence is proved in Section \ref{subsec:local} with an explicit iteration process. As a result of the local existence result (Theorem \ref{thm:local}), we show that the solution exists as long as the curvature remains bounded (see Lemma \ref{lem:blowup}). In Section \ref{subsec:apriori} and Section \ref{subsec:curvature}, we prove apriori estimates on the conformal factor and the curvature respectively, which are used to show Theorem \ref{thm:main1}.

\subsection{Local existence of the Ricci flow}\label{subsec:local}
In this section, we prove the local existence of the Ricci flow equation
\begin{equation}\label{eqn:localexist}
	\partial_t u = e^{-2u} \tilde{\triangle} u +\frac{r}{2}- e^{-2u}\tilde{K}.
\end{equation}
Here $r$ is some constant. We will show if $r$ is properly chosen according to the Gauss-Bonnet formula, then the flow preserves the volume. 

\begin{thm}
	\label{thm:local} Suppose that $u_0:S\to \Real$ and the curvature $K_0$ of $e^{2u_0}\tilde{g}$ are both in $\mathcal W^{2,\alpha}$. Then there is $T>0$ depending on $\norm{u}_{\mathcal W^{2,\alpha}}$ and $\norm{K_0}_{\mathcal W^{2,\alpha}}$ and a weak solution $u\in \mathcal V^{2,\alpha,[0,T]}$ to \eqref{eqn:localexist}. Moreover, $\partial_t u$ is in $\mathcal V^{2,\alpha,[0,T]}$. 
\end{thm}

The proof is an iteration. We start by defining $u_1(x,t)\equiv u_0(x)$.
Let $u_i(i>1)$ to be the weak solution (given by Theorem \ref{thm:linear}) of 
\begin{equation}\label{eqn:ui}
	\left\{
		\begin{array}[]{l}
			\partial_t u_i = e^{-2u_{i-1}} \tilde{\triangle} u_i +\frac{r}{2}-e^{-2u_{i-1}} \tilde{K}  \\
				u_i(0)=u_0.
		\end{array}
		\right.
\end{equation}
The goal is to show that this sequence converges and gives us the solution we want on small time interval.

\begin{lem} For $u_i$ defined above, if we write $w_i (i\geq 1)$ for $\partial_t u_i$, then $w_i (i\geq 2)$ is the weak solution of
\begin{equation}
	\left\{
		\begin{array}[]{l}
	\partial_t w_i = e^{-2u_{i-1}}\tilde{\triangle} w_i - 2 w_{i-1}(w_i -\frac{r}{2}) \\
	w_i(0)= e^{-2u_0} \tilde{\triangle} u_0 +\frac{r}{2} - e^{-2u_0} \tilde{K}. 
		\end{array}
		\right.
	\label{eqn:wi}
\end{equation}
\end{lem}
\begin{proof}
	The proof is by induction. For $i=2$, we try to apply Theorem \ref{thm:stronglinear} with $a= e^{-2u_0}$, $f=\frac{r}{2}-e^{-2u_0}\tilde{K}$ and 
	\begin{equation*}
		w_0= e^{-2u_0} \tilde{\triangle} u_0 + \frac{r}{2} - e^{-2u_0}\tilde{K}= \frac{r}{2}-K_0.
	\end{equation*}
	Since $\partial_t a =\partial_t f =0$, it remains to check $w_0\in \mathcal W^{2,\alpha}$, which follows from the assumption that $K_0$ is in $\mathcal W^{2,\alpha}$. Theorem \ref{thm:stronglinear} proves that $w_1$ is a weak solution to \eqref{eqn:wi} for $i=2$. 

	Assume the lemma is proved for $i-1$. By the induction hypothesis, $\partial_t(e^{-2u_{i-1}})$ and $\partial_t (\frac{r}{2}-e^{-2u_{i-1}}\tilde{K})$ are in $\mathcal V^{0,\alpha,[0,T]}$. $w_0$ remains as before, so that we can apply Theorem \ref{thm:stronglinear} again to conclude the proof of the lemma.
\end{proof}

To show that $u_i$ converges to the solution we want, we need several apriori estimates. We start with a uniform $C^0$ estimate.
Set
\begin{equation*}
	\tilde{C}= \max \set{\norm{u_0}_{C^0(S)}, \norm{e^{-2u_0}\tilde{\triangle} u_0 +\frac{r}{2}- e^{-2u_0}\tilde{K}}_{C^0(S)}}.
\end{equation*}
\begin{lem}\label{lem:time}
	There exists some $T>0$ depending only on $\tilde{C}$ such that
	\begin{equation*}
		\norm{u_i(t)}_{C^0(S)}, \norm{w_i(t)}_{C^0(S)}\leq \tilde{C} +1
	\end{equation*}
	uniformly for all $i$ and $t\in [0,T]$.
\end{lem}
\begin{proof}
	Note that the conclusion of the lemma holds trivially for $i=1$. Assume that it is true for $i-1$. 
	$u_i$ and $w_i$ are weak solutions to \eqref{eqn:ui} and \eqref{eqn:wi} respectively, and by the boundedness of $w_{i-1}$, the coefficients in front of $\tilde{\triangle}$ in both \eqref{eqn:ui} and \eqref{eqn:wi} have bounded time derivatives, which allows us to apply Theorem  \ref{thm:unique} to see that
	\begin{equation}\label{eqn:c01}
		\norm{u_i(t)}_{C^0(S)}\leq e^{C_1 t} \left( \norm{u_0}_{C^0(S)} + \int_0^t e^{-C_1s} C_2 ds \right)
	\end{equation}
	and
	\begin{equation}\label{eqn:c02}
		\norm{w_i(t)}_{C^0(S)}\leq e^{C_3 t} \left( \norm{e^{-2u_0}\tilde{\triangle} u_0 +\frac{r}{2}- e^{-2u_0}\tilde{K}}_{C^0(S)} + \int_0^t e^{-C_3s} C_4 ds \right).
	\end{equation}
	According to Theorem \ref{thm:unique}, we have 
	\begin{align*}
		C_1 &=  0 & C_2&=\norm{\frac{r}{2}-e^{-2u_{i-1}}\tilde{K}}_{C^0(S)} \\
		C_3 &=\norm{-2w_{i-1}}_{C^0(S)} & C_4&= \norm{r w_{i-1}}_{C^0(S)}
	\end{align*}
	By induction hypothesis again, we can bound $C_1, \cdots, C_4$ by a number depending only on $\tilde{C}$. It follows from \eqref{eqn:c01} and \eqref{eqn:c02} that we can choose $T$ small depending only on $\tilde{C}$ so that the lemma holds for $i$.
\end{proof}

The next lemma provides $C^{2,\alpha}$ estimates of $u_i$ and $w_i$ away from the singularity.  
\begin{lem}\label{lem:schauder}
	For the constant $T$ in Lemma \ref{lem:time}, we have 
	\begin{equation*}
		\norm{u_i}_{\mathcal P^{2,\alpha,[0,T]}}, \norm{w_i}_{\mathcal P^{2,\alpha,[0,T]}}\leq C
	\end{equation*}
	for some constant $C$ depending on $\tilde{C}$ and the $\mathcal W^{2,\alpha}$ norm of $u_0$ and $e^{-2u_0}\tilde{\triangle} u_0 +\frac{r}{2} -e^{-2u_0} \tilde{K}$.
\end{lem}
\begin{proof}
	Note that this lemma is not proved by induction. Instead, we have to make sure that in each step below, we obtain estimates that are uniform with respect to $i$.

	By the boundedness of $w_i$ and $u_i$ and \eqref{eqn:ui}, there is a constant $C$ depending on $\tilde{C}$ such that
	\begin{equation}\label{eqn:star2}
		\norm{\tilde{\triangle} u_i (t)}_{C^0(S)}\leq C(\tilde{C})
	\end{equation}
	for all $t\in [0,T]$. We claim that 
	\begin{equation}\label{eqn:claimalpha}
		\max_{t\in [0,T]}\norm{u_i(t)}_{C^{\alpha'}(S)}< C(\tilde{C})<\infty 
	\end{equation}
	 for some $\alpha'\in (0,1)$ depending only on $\beta$. It suffices to prove this in a neighborhood of $p$. Let $(x,y)$ be the conformal coordinates defined in $B$, then
	 \begin{equation*}
		 \tilde{\triangle} u_i(t) = (x^2+y^2)^{-\beta} \triangle u_i,
	 \end{equation*}
	 which implies (by \eqref{eqn:star2} and $\beta>-1$) that there is some $q>1$ such that $\triangle u_i\in L^q(B)$. By the $L^q$ estimate of $\triangle$ and the Sobolev embedding theorem, there is a function $v$ in $C^{\alpha'}(B)\cap W^{1,2}(B)$ with
	 \begin{equation*}
		 \triangle v = (x^2+y^2)^\beta \tilde{\triangle} u_i(t) \quad \mbox{on } \quad B
	 \end{equation*}
	 and
	 \begin{equation*}
		 v|_{\partial B}=u_i(t).
	 \end{equation*}
	 Since both $u_i(t)$ and $v$ are bounded on $B$ and their difference is a harmonic function on $B\setminus \set{0}$ which vanishes on $\partial B$, we know $u_i(t)=v$ and hence $u_i(t)\in C^{\alpha'}(B)$. This proves the claim.

	 Let's assume that $\alpha'<\alpha$ because the proof for the case $\alpha'\geq \alpha$ is simpler. \eqref{eqn:claimalpha} and the boundedness of $\partial_t u_i$ imply a uniform bound of
	 \begin{equation*}
		 \norm{u_i}_{\mathcal P^{0,\alpha',[0,T]}}< C(\tilde{C}),
	 \end{equation*}
	 which allows us to apply the interior Schauder estimate to \eqref{eqn:ui} to get a uniform bound of $\norm{u_i}_{\mathcal P^{2,\alpha',[0,T]}}$. Although $\alpha'<\alpha$, we have at least $\norm{u_i}_{\mathcal P^{0,\alpha,[0,T]}}$ is bounded, so that we can apply Schauder estimate again to \eqref{eqn:ui} to get the uniform bound $\norm{u_i}_{\mathcal P^{2,\alpha,[0,T]}}$ as we need. The estimate for $w_i$ follows easily by the Schauder estimate of \eqref{eqn:wi}.
\end{proof}

With Lemma \ref{lem:schauder}, $u_i$ subconverges. We now claim that if we choose $T$ to be even smaller, we can make $u_i$ converge without taking any subsequence.
\begin{lem}
	There exists $T>0$ (maybe smaller than given in Lemma \ref{lem:time}) such that $u_i$ is a Cauchy sequence in $C^0(S\times [0,T])$ norm.	
\end{lem}
\begin{proof}
	It follows from \eqref{eqn:ui} that
\begin{equation*}
	\pfrac{}{t}(u_{i+1}-u_i) = e^{-2 u_{i}} \tilde{\triangle} (u_{i+1}-u_i) + (e^{-2 u_{i}}-e^{-2u_{i-1}})\tilde{\triangle} u_i -(e^{-2 u_{i-1}}- e^{-2u_i})\tilde{K}.
\end{equation*}
Regard $u_{i+1}-u_i$ as a weak solution to
\begin{equation*}
	\partial_t (u_{i+1}-u_i) = a(x,t) \tilde{\triangle} (u_{i+1}-u_i) + f(x,t) 
\end{equation*}
where $a=e^{-2u_i}$ and $f=(e^{-2 u_{i}}-e^{-2u_{i-1}})\tilde{\triangle} u_i -(e^{-2 u_{i-1}}- e^{-2u_i})\tilde{K}$.
Since $w_i=\partial_t u_i$ and $u_i$ are uniformly bound, so is $\tilde{\triangle} u_i$ by \eqref{eqn:ui}, we have
\begin{equation*}
	\abs{f} \leq C \norm{u_i-u_{i-1}}_{C^0(S\times [0,T])}.
\end{equation*}
Moreover, $\partial_t a$ is bounded so that Theorem \ref{thm:unique} gives
\begin{equation*}
	\norm{u_{i+1}-u_i}_{C^0(S\times [0,T])}\leq T C \norm{u_i-u_{i-1}}_{C^0(S\times [0,T])}.
\end{equation*}
Hence, if we choose $T$ small, then the sequence $u_i$ is Cauchy in $C^0$ norm and the lemma is proved.
\end{proof}
Let $u$ be the limit of $u_i$. By Lemma \ref{lem:schauder}, the convergence of $u_i$ away from the singularity is in $C^{2,\alpha'}$ for any $\alpha'<\alpha$ and hence the limit $u$ satisfies the Ricci flow equation \eqref{eqn:rfu} pointwisely. To finish the proof of Theorem \ref{thm:local}, we still need
\begin{lem}
	For $T$ determined above,
	\begin{equation*}
		\abs{u_i}_{\mathcal V^{[0,T]}}, \abs{w_i}_{\mathcal V^{[0,T]}}\leq C(T).
	\end{equation*}
	In particular, $u$ is a weak solution.
\end{lem}
\begin{proof}
	Recall that $u_i$ and $w_i$ are the weak solutions to \eqref{eqn:ui} and \eqref{eqn:wi} respectively. Since $\partial_t (e^{-2u_{i-1}})$ is uniformly bounded, we obtain control of $\abs{u_i}_{\mathcal V^{[0,T]}}$ and $\abs{w_i}_{\mathcal V^{[0,T]}}$ by Theorem \ref{thm:unique}. Note that the constants $C_3$ and $C_4$ in \eqref{eqn:linear1energy} depend only on $C^0$ norm of the coefficients and initial data (for \eqref{eqn:ui} and \eqref{eqn:wi} respectively), which are uniformly bounded as in Lemma \ref{lem:time}.
\end{proof}

Based on Theorem \ref{thm:local}, we can discuss the blow-up criterion which serve as a starting point for the proof of long time existence. We start with a uniqueness result.
\begin{lem}
	\label{lem:uniqueRF}
	For $i=1,2$, suppose that $u_i\in V^{2,\alpha,[0,T]}$ is a weak solution to Ricci flow \eqref{eqn:rfu} for some $T>0$. If the Gauss curvature of $e^{2u_i}\tilde{g}$ are bounded on $S\times [0,T]$ and that $u_1(0)=u_2(0)$, then $u_1=u_2$ on $S\times [0,T]$.
\end{lem}
\begin{proof}
We subtract the equation of $u_i$ to get
\begin{equation*}
	\partial_t (u_1-u_2)= e^{-2u_1}\tilde{\triangle} (u_1-u_2) + (e^{-2u_1}-e^{-2u_2})(\tilde{\triangle} u_2 -\tilde{K}),
\end{equation*}
where
\begin{equation*}
	e^{-2u_1}-e^{-2u_2} = - 2 (u_1-u_2) \int_0^1 e^{-2u_2-2t(u_1-u_2)} dt.	
\end{equation*}
The curvature bound of $u_1$ and $u_2$ implies a bound for $\partial_t (e^{-2u_1})$ and $\tilde{\triangle} u_2$ respectively so that we can apply Lemma \ref{lem:maximum} to show $u_1\equiv u_2$.
\end{proof}

Suppose $u_0$ is an initial data satisfying the assumptions in Theorem \ref{thm:local}. Let $T_{max}$ be the supremum of $T$ such that there is a solution $u\in \mathcal V^{2,\alpha,[0,T]}$ satisfying $u(0)=u_0$ and $\partial_t u \in \mathcal V^{2,\alpha,[0,T]}$. Theorem \ref{thm:local} implies that $T_{max}>0$. The next lemma gives a characterization of $T_{max}$. 
\begin{lem}
	\label{lem:blowup}
	Suppose that $u_0$ satisfies the assumptions of Theorem \ref{thm:local} and $T_{max}$ is defined as above. Then there is a solution $u(t)$ defined on $S\times [0,T_{max})$ such that $u$ and $\partial_t u$ lie in $\mathcal V^{2,\alpha,[0,T_{max})}$. Moreover, if $T_{max}< +\infty$, then
		\begin{equation}\label{eqn:Kblow}
			\limsup_{t\to T_{max}} \norm{K(t)}_{C^0(S)}=+\infty
		\end{equation}
		where $K(t)$ is the Gauss curvature of $e^{2u}\tilde{g}$.
\end{lem}
In the rest of this paper, we shall call the solution given in the above lemma the maximal solution starting from $u_0$. 
\begin{proof}
	By the definition of $T_{max}$, for each $T<T_{max}$, there is $u_T\in \mathcal V^{2,\alpha,[0,T]}$ solving \eqref{eqn:rfu} with $u_T(0)=u_0$ and $\partial_t u_T\in \mathcal V^{2,\alpha,[0,T]}$. For any $t<T_{max}$, we define $u(t)=u_T(t)$ for any $t<T<T_{max}$, which is well defined by Lemma \ref{lem:uniqueRF}. It remains to show \eqref{eqn:Kblow}.

	If the lemma is not true, then there is $C>0$ such that
	\begin{equation*}
		\sup_{S\times [0,T_{max})} \abs{K}\leq C.
	\end{equation*}
	Since \eqref{eqn:rfu} is equivalent to $\partial_t u = \frac{r}{2}-K$, we know
	\begin{equation*}
		\sup_{S\times [0,T_{max})} \abs{u}\leq C
	\end{equation*}
	for possibly another $C>0$.
	Using a proof similar to Lemma \ref{lem:schauder}, we have 
	\begin{equation}\label{eqn:ext1}
		\norm{u}_{\mathcal P^{2,\alpha,[0,T]}} +\norm{\partial_t u}_{\mathcal P^{2,\alpha,[0,T]}}< C(T_{max}),
	\end{equation}
	for all $T<T_{max}$.

	We claim that for all $T<T_{max}$,
	\begin{equation}\label{eqn:ext2}
		\abs{u}_{\mathcal V^{[0,T]}} + \abs{\partial_t u}_{\mathcal V^{[0,T]}} < C(T_{max}).
	\end{equation}
	To see the claim, we regard $u$ as the weak solution of the linear equation
	\begin{equation*}
		\partial_t u = a(x,t) \tilde{\triangle} u + f(x,t)
	\end{equation*}
	where $a=e^{-2u}$ and $f= \frac{r}{2}-e^{-2u}\tilde{K}$. Since $\partial_t a$ is bounded, we apply Theorem \ref{thm:unique} to bound $\abs{u}_{\mathcal V^{[0,T]}}$ by $T$, the $C^0$ norm of $a,f,u$ and $a^{-1}$ and the $\mathcal W^{2,\alpha}$ norm of $u_0$. The same argument applies to $\partial_t u$.

	With \eqref{eqn:ext1} and \eqref{eqn:ext2}, we can extend the definition of $u$ to $T_{max}$ such that
	\begin{equation*}
		\norm{u(T_{max})}_{\mathcal W^{2,\alpha}}+ \norm{\tilde{\triangle} u(T_{max})}_{\mathcal W^{2,\alpha}}< +\infty.
	\end{equation*}
	Now Theorem \ref{thm:local} shows that we can extend the domain of $u$ to $T_{max}+\delta$, while keeping $\norm{u}_{\mathcal V^{2,\alpha,[0,T_{max}+\delta}]}$ and $\norm{\partial_t u}_{\mathcal V^{2,\alpha,[0,T_{max}+\delta}]}$ finite. This is a contradiction to the definition of $T_{max}$. 
\end{proof}

Finally, to conclude this section, we prove that for a natural choice of $r$, the maximal solution of \eqref{eqn:rfu} preserves the volume and the Gauss-Bonnet formula remains true as long as the solution exists. Following \cite{troyanov1991prescribing}, we set
\begin{equation*}
	\chi(S,\beta)= \chi(S) + \sum \beta_i,
\end{equation*}
where $\chi(S)$ is the Euler number of the underlying Riemann surface $S$.

\begin{prop}
	\label{prop:normalize}
	Suppose that $u_0$ satisfies the assumptions of Theorem \ref{thm:local} and that $u(t)$ is the maximal solution starting from $u_0$. If
	\begin{equation*}
		r=\frac{4\pi \chi(S,\beta)}{V_0}
	\end{equation*}
	where $V_0=\int_S e^{2u_0} d\tilde{V}$ is the volume of the initial metric, then
	\begin{equation*}
		V(t):=\int_S e^{2u(t)} d\tilde{V}
	\end{equation*}
	is a constant for $t\in [0,T_{max})$ and 
		\begin{equation}\label{eqn:GB}
			2\pi \chi(S,\beta)= \int_{S} K_t dV_t.
		\end{equation}
		Here $K_t$ and $V_t$ are the Gauss curvature and the volume form of $g(t)=e^{2u(t)}\tilde{g}$.
\end{prop}

\begin{proof}
	Since $\tilde{g}$ is the standard cone metric near the cone point, we can check by using the Gauss-Bonnet theorem involving the geodesic curvature on the boundary that 
	\begin{equation*}
		2\pi\chi(S,\beta) = \int_S \tilde{K} d\tilde{V}.
	\end{equation*}
	For any $t\in [0,T_{max})$, by $K_t= e^{-2u}(-\tilde{\triangle} u + \tilde{K})$, we have
		\begin{equation*}
			\int_S K_t dV_t = \int_S - \tilde{\triangle} u + \tilde{K} d\tilde{V}.
		\end{equation*}
		Since $u(t)$ is in $\mathcal W^{2,\alpha}$ (as in the definition of maximal solution) and $\tilde{\triangle} u$ is bounded, Lemma \ref{lem:basic} implies that
		\begin{equation*}
			\int_S \tilde{\triangle} u d\tilde{V}=0,
		\end{equation*}
		which proves \eqref{eqn:GB}.

		Since $\abs{u}_{\mathcal V^{[0,T]}}$ is finite for any $T<T_{max}$, Lemma \ref{lem:basic2} and Lemma \ref{lem:basic} allow us to compute
		\begin{equation*}
			\frac{d}{dt} V(t) = \int_S 2\tilde{\triangle} u + re^{2u} -2\tilde{K} d\tilde{V}= rV(t)- rV_0.
		\end{equation*}
		Since $V(0)=V_0$ by definition, we have $V(t)=V_0$ for all $t\in [0, T_{max})$.
\end{proof}
\subsection{Apriori estimate for the conformal factor}\label{subsec:apriori}
Lemma \ref{lem:blowup} implies that if $T_{max}$ for a maximal solution is finite, then the curvature $K(t)$ blows up as $t\to T_{max}$. The next lemma implies that at least the $C^0$ norm of conformal factor $u$ will stay bounded for any finite time interval.
\begin{lem}\label{lem:c2}
	Suppose $u_0$ is some initial data satisfying the assumptions of Theorem \ref{thm:local}. Let $u(t)$ be the maximal solution given by Lemma \ref{lem:blowup} with $T_{max}< +\infty$. There exists $C>0$ depending on $T_{max}$ and $u_0$ such that
	\begin{equation*}
		\norm{u}_{C^0(S\times [0,T_{max}) )}\leq C.
	\end{equation*}
\end{lem}

The proof follows some well known approach in K\"ahler geometry.  In the smooth case, if $\varphi(t)$ is the potential function in the sense that
\begin{equation*}
	\tilde{\triangle} \varphi(t)= e^{2u(t)}-\frac{V_t}{\tilde{V}} 
\end{equation*}
where $V_t$ is the volume of $g(t)$, then (up to some normalization) $\partial_t \varphi$ satisfies a linear parabolic equation, from which we obtain immediately $C^0$ apriori estimate of $u$. The rest of this section is to prove Lemma \ref{lem:c2} by showing that this argument works for conical surfaces as well. 

For the definition of potential function, we need
\begin{lem}
	\label{lem:poisson}
	Suppose that $f$ is a $\mathcal W^{2,\alpha}$ functions satisfying 
	\begin{equation*}
		\int_S f d\tilde{V} =0.
	\end{equation*}
	Then up to a constant, there is a unique $u\in \mathcal W^{4,\alpha}$ such that
	\begin{equation*}
		\tilde{\triangle} u = f.
	\end{equation*}
\end{lem}
The proof is elementary and not new (see Lemma 2.10 of \cite{yin2013ricci}) and hence is moved to the appendix.

By the choice of $r$ in Proposition \ref{prop:normalize}, Lemma \ref{lem:poisson} gives some $h_0 \in \mathcal W^{4,\alpha}$ satisfying
\begin{equation}\label{eqn:h0}
	\tilde{\triangle} h_0= \frac{r V_0}{2\tilde{V}}-\tilde{K}
\end{equation}
because
\begin{equation*}
	\int_S \frac{r V_0}{2 \tilde{V}} d\tilde{V} = \frac{r V_0}{2} = 2\pi \chi(S,\beta) = \int_S \tilde{K} d\tilde{V}.
\end{equation*}
Note that $h_0$ is determined only up to a constant.
The existence of potential function $\varphi(t)$ is given in the next lemma
\begin{lem}
	\label{lem:potential}
	Suppose that $u(t)$ is a maximal solution to \eqref{eqn:rfu} and that $h_0$ is defined as in \eqref{eqn:h0}. Then there exists $\varphi(t)$ such that
	\begin{equation}\label{eqn:varphiode}
		\varphi'- r\varphi= 2u(t)+ 2h_0
	\end{equation}
	and
	\begin{equation*}
		\tilde{\triangle} \varphi = e^{2u} -\frac{V_0}{\tilde{V}}.
	\end{equation*}
	Moreover, $\norm{\partial_t \varphi}_{C^0(S\times [0,T])}$ and $\abs{\partial_t \varphi}_{\mathcal V^{[0,T]}}$ are finite for any $T<T_{max}$.
\end{lem}

\begin{proof}
	By the definition of $V_0$ in Proposition \ref{prop:normalize} and Lemma \ref{lem:poisson}, we have $\varphi_0\in \mathcal W^{4,\alpha}$ solving
	\begin{equation}\label{eqn:varphi0}
		\tilde{\triangle} \varphi_0= e^{2u_0}-\frac{V_0}{\tilde{V}}.
	\end{equation}
	For $t\in [0,T_{max})$, we can define $\varphi(t)$ to be the solution of  \eqref{eqn:varphiode} with $\varphi(0)=\varphi_0$.
		By solving the ODE \eqref{eqn:varphiode}, we get
\begin{equation}\label{eqn:varphi}
	\varphi(t)=e^{rt}\left( \varphi(0)+ \int_0^t ( 2 u(s)+ 2h_0)e^{-rs}ds \right).
\end{equation}
We claim that
\begin{equation}\label{eqn:potential}
	\tilde{\triangle} \varphi(t)= e^{2u}-\frac{V_0}{\tilde{V}},
\end{equation}
which is true for $t=0$ by \eqref{eqn:varphi0}. To see the claim is true for $t>0$, we compute 
\begin{eqnarray*}
	&& \partial_t (\tilde{\triangle} \varphi - e^{2u} +\frac{V_0}{\tilde{V}}) \\
	&=& \tilde{\triangle} \partial_t \varphi - e^{2u} 2 \partial_t u \\
	&=& \tilde{\triangle} \partial_t \varphi -2 \left( \tilde{\triangle} u + \frac{r}{2}e^{2u} - \frac{r V_0}{2 \tilde{V}}+ \frac{r V_0}{2\tilde{V}}-\tilde{K} \right) \\
	&=& \tilde{\triangle} \left( \partial_t \varphi -2u - 2h_0 -  r\varphi \right) + r \left( \tilde{\triangle} \varphi -e^{2u}+ \frac{V_0}{\tilde{V}} \right) \\
	&=&  r \left( \tilde{\triangle} \varphi -e^{2u}+\frac{V_0}{\tilde{V}} \right) .
\end{eqnarray*}
Here in the above computation, we used \eqref{eqn:rfu}, \eqref{eqn:h0} and \eqref{eqn:varphiode}.

It remains to check that $\abs{\partial_t \varphi}_{V^{[0,T]}}$ and $\norm{\partial_t \varphi}_{C^0(S\times [0,T])}$ are finite for any $T<T_{max}$.  By \eqref{eqn:varphiode}, it suffices to show $\abs{\varphi}_{\mathcal V^{[0,T]}}$ and $\norm{\varphi}_{C^0(S\times [0,T])}$ are finite  because $\norm{h_0}_{\mathcal W^{4,\alpha}}<\infty$ and $\abs{u}_{\mathcal V^{[0,T]}}< \infty$ by the definition of the maximal solution and $T<T_{max}$. 

For this purpose, we derive from \eqref{eqn:varphi}
\begin{equation*}
	\max_{t\in [0,T]} \int_S \abs{\tilde{\nabla} \varphi}^2 d\tilde{V} < \infty
\end{equation*}
and
\begin{equation*}
	\max_{t\in [0,T]} \norm{\varphi(t)}_{C^0(S)} < \infty
\end{equation*}
by using the fact that $\varphi_0, h_0$ are in $\mathcal W^{4,\alpha}$ and $u(t)$ is in $\mathcal V^{2,\alpha,[0,T]}$. 
Finally, \eqref{eqn:varphiode} shows that $\varphi'$ is bounded on $S\times [0,T]$, which is stronger than
\begin{equation*}
	\iint_{S\times [0,T]} \abs{\partial_t \varphi}^2 d\tilde{V} ds <+\infty.
\end{equation*}
\end{proof}

Given the existence of the potential function $\varphi$, we move on to derive a uniform upper bound of $\partial_t \varphi$ (up to $T_{max}$). Using (\ref{eqn:potential}), we compute the equation satisfied by $\partial_t \varphi$ as follows
\begin{eqnarray*}
	\partial_t (\partial_t \varphi) &=& 2 \partial_t u + r \partial_t \varphi \\
	&=& e^{-2u}\left( \tilde{\triangle} (2u) + r e^{2u}-\frac{rV_0}{\tilde{V}} +\frac{rV_0}{\tilde{V}} -2\tilde{K} \right) + r\partial_t \varphi \\
	&=& e^{-2u} \tilde{\triangle} (2u+ r\varphi + 2h_0) + r \partial_ t \varphi \\
	&=& e^{-2u} \tilde{\triangle} (\partial_t \varphi) + r \partial_t \varphi.
\end{eqnarray*}
With this equation and the fact that $u\in \mathcal P^{2,\alpha,[0,T]}$ and $\partial_t \varphi(0)=r\varphi_0+2u(0)+2h_0\in \mathcal W^{2,\alpha}$, the interior Schauder estimate implies that $\partial_t \varphi$ is in $\mathcal P^{2,\alpha,[0,T]}$. Together with the finiteness of $\abs{\partial_t \varphi}_{\mathcal V^{[0,T]}}$, we know $\partial_t \varphi$ is a weak solution to the linear equation
\begin{equation*}
	\partial_t (\partial_t \varphi) = e^{-2u} \tilde{\triangle} (\partial_t \varphi) +r \partial_t \varphi.
\end{equation*}

The final step in the proof of Lemma \ref{lem:c2} is to realize that for each $T<T_{max}$, Theorem \ref{thm:unique} applies to $\partial_t \varphi$ as a weak solution to the above equation to give the required apriori $C^0$ bound.

\subsection{Curvature bound and global existence}\label{subsec:curvature}
In this section, we prove Theorem \ref{thm:main1}. Suppose $u_0$ satisfies the assumption of Theorem \ref{thm:local} and $r$ is chosen as in Proposition \ref{prop:normalize}. Let $u(t)$ be the maximal solution given in Lemma \ref{lem:blowup}. It suffices to show that $T_{max}< +\infty$ is not possible. If otherwise, Lemma \ref{lem:c2} gives a constant $C_1$ depending on $T_{max}$ and $u_0$ such that
\begin{equation}\label{eqn:c0bound}
	\norm{u}_{C^0(S\times [0,T_{max}))}\leq C_1 < +\infty.
\end{equation}
We will show in this section that this $C^0$ norm bound of $u$ contradicts Lemma \ref{lem:blowup}, which asserts that 
		\begin{equation*}
			\limsup_{t\to T_{max}} \norm{K(t)}_{C^0(S)}=+\infty.
		\end{equation*}
By \eqref{eqn:c0bound}, this is equivalent to
\begin{equation*}
	\limsup_{t\to T_{max}} \norm{\partial_t u}_{C^0(S)} = +\infty.
\end{equation*}
Hence, we can choose $x_i\in S\setminus \set{p}$ and $t_i\to T_{max}$ such that
	\begin{equation*}
		\lim_{i\to \infty} \abs{\partial_t u(x_i,t_i)}= \infty.
	\end{equation*}
	By modifying $x_i$ and $t_i$ if necessary, we may assume
	\begin{equation}\label{eqn:assume}
		\abs{\partial_t u}(x_i,t_i)\geq \frac{1}{2}\sup_S \abs{\partial_t u(t_i)}=\frac{1}{2}\sup_{t\in [0,t_i]}\sup_S \abs{\partial_t u}.
	\end{equation}

	For any $T<T_{max}$, we can apply Theorem \ref{thm:dgspecial} directly to $u$ as a function on $S\times [0,T]$ to see that there is $\alpha'$ depending on $C_1$ such that for small $\delta>0$,
	\begin{equation}\label{eqn:ucalpha}
		\norm{u}_{C^{\alpha'}(S\times [\delta,T_{max}) )}\leq C.
	\end{equation}
	Note that on the smooth part of $S$, the $C^0$ norm of $u$ is enough to bound any derivative of $u$ by applying the known theory of quasilinear parabolic equation to \eqref{eqn:rfu}, so we may assume that $x_i$ converges to the unique singular point $p$.  Let
	\begin{equation*}
		\lambda_i= \abs{\partial_t u(x_i,t_i)}\to \infty.
	\end{equation*}
	We compare the speed of $x_i\to p$ and $\lambda_i\to \infty$ and distinguish three cases.

	Case one: $d_{\tilde{g}}(x_i,p)^2 \lambda_i=\infty$. In fact, this case never happens because we can apply the theory of quasilinear parabolic equation to \eqref{eqn:rfu} on a ball centered at $x_i$ with the radius being a small multiple (depending only on $\beta$) of $d_{\tilde{g}}(x_i,p)$ to see that  
	\begin{equation*}
		\abs{\partial_t u}(x_i,t_i) \leq \frac{C}{d_{\tilde{g}}(x_i,p)^2}.
	\end{equation*}

	Case two: $0<d_{\tilde{g}}(x_i,p)^2 \lambda_i< \infty$. Let $(\rho,\theta)$ be the polar coordinates around $p$. Suppose $x_i=(\rho_i,\theta_i)$. By passing to some subsequence, we may assume (without loss of generality) that
	\begin{equation*}
		({\rho_i}{\lambda_i^{1/2}},\theta_i)\to (1,0).
	\end{equation*}
Set
\begin{equation*}
	w_i(\rho,\theta,t)=u(\frac{\rho}{\lambda_i^{1/2}},\theta,t_i+\frac{t}{\lambda_i}),
\end{equation*}
which satisfies
\begin{equation}\label{eqn:nonzero}
	\abs{\partial_t w_i}(\rho_i \lambda_i^{1/2},\theta_i,0)=1
\end{equation}
and
\begin{equation}
	\pfrac{w_i}{t}(\rho,\theta,t)= e^{-2w_i} \tilde{\triangle} w_i (\rho,\theta,t)+ \frac{1}{\lambda_i}\left[ \frac{r}{2}-e^{-2w_i }\tilde{K} \right].
	\label{eqn:localw2}
\end{equation}
We can apply the Schauder estimate in a neighborhood of $(\rho,\theta,t)=(1,0,0)$ to see that $w_i$ converges in $C^2$ to a limit $w_\infty$ with
\begin{equation*}
	\partial_t w_\infty (1,0,0)=1.
\end{equation*}
This is a contradiction to \eqref{eqn:ucalpha}, which implies that $w_\infty$ must be a constant.

Case three: $d_{\tilde{g}}(x_i,x_0)^2\lambda_i=0$. Let $w_i$ be defined as in Case two so that \eqref{eqn:nonzero} holds. 
In this case, $\rho_i \lambda_i^{1/2}$ converges to zero. Taking $t$-derivative of the equation satisfied by $w_i$, we have
\begin{equation*}
	\partial_t(\partial_t w_i)= e^{-2w_i} \tilde{\triangle }(\partial_t w_i) + (-2 \partial_t w_i)\left[ \partial_t w_i -\frac{1}{\lambda_i}\left( \frac{r}{2}-e^{-2w_i}\tilde{K} \right) \right] + \frac{1}{\lambda_i}\left( 2\partial_t w_i e^{-2w_i}\tilde{K} \right).
\end{equation*}
By (\ref{eqn:assume}), the term
\begin{equation*}
	(-2 \partial_t w_i)\left[ \partial_t w_i -\frac{1}{\lambda_i}\left( \frac{r}{2}-e^{-2w_i}\tilde{K} \right) \right] + \frac{1}{\lambda_i}\left( 2\partial_t w_i e^{-2w_i}\tilde{K} \right)
\end{equation*}
is uniformly bounded on $\set{(\rho,\theta,t)|\, \rho<2, t\in [-1,0]}$. By the scaling invariance of $\abs{\cdot}_{\mathcal V}$ and the definition of maximal solution (Lemma \ref{lem:blowup}), we know $\partial_t w_i$ is a weak solution defined on $\set{(\rho,\theta,t)|\, \rho<2, t\in [-1,0]}$. Theorem \ref{thm:dg} then implies the existence of $\alpha'\in (0,1)$ and $C_1>0$ (independent of $i$) such that
\begin{equation*}
	\norm{\partial_t w_i(0)}_{C^{\alpha'}(\set{(\rho,\theta)|\, \rho<1})}\leq C_1.
\end{equation*}
\begin{rem}
	Note that we do not have uniform control over $\abs{\partial_t w_i}_{\mathcal V^{[-1,0]} (\set{\rho<2})}$. The point is that Theorem \ref{thm:dg} only requires that it is finite and the constant $C_1$ does not depend on the particular value of it.	
\end{rem}
This together with (\ref{eqn:nonzero}) gives (for $i$ large)
\begin{equation*}
	\abs{\partial_t w_i} (\tilde{\rho},0,0)\geq 1/2, 
\end{equation*}
where $\tilde{\rho}= \left( \frac{1}{4C_1} \right)^{1/\alpha'}$. We can then obtain a contradiction as in Case two.

\section{Higher regularity of conical Ricci flow}\label{sec:higher}
In previous sections, we proved the global existence of a Ricci flow solution. For any $T< \infty$, we know that $u$ and $\partial_t u$ (or equivalently $K$) are in $\mathcal V^{2,\alpha,[0,T]}$. In this section, we show that

\begin{lem}\label{lem:higher}
	Suppose $u$ is the solution in Theorem \ref{thm:main1}. If for some $C_1>0$ and $T>1$, we have
	\begin{equation*}
		\norm{u}_{\mathcal V^{2,\alpha,[T-1,T]}} + \norm{\partial_t u}_{\mathcal V^{2,\alpha,[T-1,T]}}\leq C_1,	
	\end{equation*}
	then for any $k>1$, there exists $C_2(k)$ depending only on $C_1$ (not on $T$) such that 
	\begin{equation*}
		\norm{\partial_t^k u}_{\mathcal V^{2,\alpha,[T-1/2,T]}}\leq C_2(k).	
	\end{equation*}
\end{lem}
\begin{rem}
	For any constant $\delta>0$, we may replace $T>1$ by $T>\delta$, $[T-1,T]$ by $[T-\delta,T]$ and $[T-1/2,T]$ by $[T-\delta/2,T]$ in the above lemma, which still holds with $C_2$ depending on $\delta$.
\end{rem}
Before we start the proof, we note that since $\partial_t u =\frac{r}{2}-K$, it is equivalent to bound $\partial_t^{k-1}K$.

\subsection{Regularity of $\partial_t K$}\label{subsec:dtK}

\begin{lem}\label{lem:dtK}
	Let $u$ be the solution in Lemma \ref{lem:higher}. For any $0<\delta<1$, we have
	\begin{equation*}
		\norm{\partial_t K}_{\mathcal V^{2,\alpha,[T-\delta,T]}}< C
	\end{equation*}
	for some $C$ depending on $C_1$ in Lemma \ref{lem:higher} and $\delta$.
\end{lem}

We study the evolution equation of $K$ instead of $u$,
\begin{equation}
	\partial_t K = e^{-2u} \tilde{\triangle} K + K(2K -r).
	\label{eqn:flowK}
\end{equation}

The proof is very similar to that of Theorem \ref{thm:smoothing}. For $\delta\in (0,1)$ in Lemma \ref{lem:dtK}, set $t_0=T-\frac{1+\delta}{2}$. The same proof as in Lemma \ref{lem:lp} shows
\begin{equation}\label{eqn:wlp}
	\partial_t K(t_0)\in L^q(S,\tilde{g})
\end{equation}
for some $q>1$.

Next, we compute the evolution equation of $w=\partial_t K$. Taking $t$-derivative of \eqref{eqn:flowK} gives
\begin{equation*}
	\partial_t w = e^{-2u} \tilde{\triangle} w + e^{-2u}(-2\partial_t u) \tilde{\triangle} K + w(4K-r).
\end{equation*}
Using $\partial_t u = -K +r/2$ and \eqref{eqn:flowK}, we get
\begin{equation*}
	\partial_t w = e^{-2u} \tilde{\triangle} w + (w-K(2K-r))(2K-r) + w(4K-r),
\end{equation*}
which is simplified to
\begin{equation}\label{eqn:w}
	\partial_t w =e^{-2u} \tilde{\triangle} w + w(6K-2r) -K(2K-r)^2.
\end{equation}
We take \eqref{eqn:w} as a linear parabolic equation of $w$, while the coefficients are in $\mathcal V^{2,\alpha,[T-1,T]}$ and $\partial_t e^{-2u}$ lies also in $\mathcal V^{2,\alpha,[T-1,T]}$. Together with \eqref{eqn:wlp}, Lemma \ref{lem:growth}\footnote{We use here a different linear equation, but the coefficients satisfy the same assumption.} gives us a solution $\tilde{w}$ to the initial value problem 
\begin{equation*}
	\left\{
		\begin{array}[]{l}
			\partial_t \tilde{w} = e^{-2u} \tilde{\triangle} \tilde{w} + \tilde{w}(6K-2r) - K(2K-r)^2 \\
			\tilde{w}(t_0)= \partial_t K(t_0).
		\end{array}
		\right.
\end{equation*}
Moreover, Lemma \ref{lem:growth} guarantees that
\begin{equation}\label{eqn:weakw}
	\norm{\tilde{w}}_{\mathcal V^{2,\alpha,[t_0+\eta,T]}}\leq C(\eta) \quad \mbox{for } 0<\eta<T-t_0
\end{equation}
and
\begin{equation}\label{eqn:goodw}
	\norm{\tilde{w}(t)}_{C^0(S)} + \norm{\tilde{\nabla} \tilde{w}(t)}_{L^2(S,\tilde{g})}\leq \frac{C}{(t-t_0)^{1/q}}
\end{equation}
for $t\in (t_0,T]$.

We define for $t\in [t_0,T]$
\begin{equation*}
	\tilde{K}(t)=K(t_0)+ \int_{t_0}^t \tilde{w}(s)ds.
\end{equation*}
The proof of Lemma \ref{lem:dtK} is done if we can show that $\tilde{K}\equiv K$ for any $t\in [t_1,T]$. To show this, we follow the proof of Lemma \ref{lem:same}.

By the fact that $\tilde{w}(t_0)=\partial_t K(t_0)$ and $\tilde{K}(t_0)=K(t_0)$, $\tilde{K}$ satisfies the (\ref{eqn:flowK}) at $t_0$. For later time, we compute $\partial_t H$ for
\begin{eqnarray*}
	H&:=&   \partial_t \tilde{K} - e^{-2u} \tilde{\triangle}\tilde{K} -\tilde{K}(2\tilde{K}-r),
\end{eqnarray*}
\begin{eqnarray*}
	\partial_t H &=& \partial_t \tilde{w} - e^{-2u} \tilde{\triangle} \tilde{w} - \tilde{w}(4\tilde{K}-r) + e^{-2u} \tilde{\triangle} \tilde{K} (2 \partial_t u) \\
	&=& (2\partial_t u)(-H) + (r-2K)\left( \tilde{w}-\tilde{K}(2\tilde{K}-r) \right) + \partial_t \tilde{w} -e^{2u}\tilde{\triangle} \tilde{w} - \tilde{w}(4\tilde{K}-r) \\
	&=& (2\partial_t u) (-H) + 4 \tilde{w}(K-\tilde{K}) + (2K-r) \left( \tilde{K}(2\tilde{K}-r)-K(2K-r) \right).
\end{eqnarray*}
Here in the second line above, we used $2\partial_t u = r- 2K$ and the definition of $H$; in the last line above, we used the equation satisfied by $\tilde{w}$.  Due to \eqref{eqn:goodw}, we know $\tilde{K}$ is bounded on $S\times [t_1,T]$, while $K$ is bounded by the assumption of the lemma so that the last term in the above equation is bounded by
\begin{equation*}
	\abs{(2K-r) \left( \tilde{K}(2\tilde{K}-r)-K(2K-r) \right)}\leq C \abs{\tilde{K}-K}.
\end{equation*}
On the other hand, by \eqref{eqn:goodw},
\begin{equation*}
	\abs{4\tilde{w}(K-\tilde{K})}\leq \frac{C}{ (t-t_0)^{1/q}} \abs{\tilde{K}-K}.
\end{equation*}
In summary, we obtained
\begin{equation*}
	\abs{\partial_t H}\leq C \abs{H} + \frac{C}{(t-t_0)^{1/q}} \abs{\tilde{K}-K},
\end{equation*}
from which we get by integration (using $H(t_0)=0$)
\begin{equation}\label{eqn:H2}
	\abs{H}(t)\leq C \int_{t_0}^t \frac{1}{(s-t_0)^{1/q}}\abs{\tilde{K}-K} ds
\end{equation}
for $t\in [t_0,T]$. 
Set
\begin{equation*}
	F(t):= \sup_{S} \abs{\tilde{K}(t)-K(t)} \quad \mbox{and} \quad F_H(t):= \sup_S \abs{H}.
\end{equation*}
\eqref{eqn:H2} implies that
\begin{equation*}
	F_H(t) \leq C\int_{t_0}^t \frac{1}{(s-t_0)^{1/q}} F(s) ds.
\end{equation*}
By (\ref{eqn:flowK}) and the definition of $H$, we have
\begin{equation}\label{eqn:KK}
	\partial_t (\tilde{K}-K)=H + e^{-2u} \tilde{\triangle}(\tilde{K}-K) + (\tilde{K}-K)(2\tilde{K}+2K -r).
\end{equation}

In order to apply Lemma \ref{lem:odecompare} to \eqref{eqn:KK}, we check that \eqref{eqn:assmax} holds. In fact, as in Step 2 of the proof of Lemma \ref{lem:same}, we have
\begin{equation*}
	\norm{\tilde{\nabla} \tilde{K}}_{L^2(S,\tilde{g})}\leq C' \quad \mbox{for } t\in [t_1,T]
\end{equation*}
and
\begin{equation*}
	\int_{t_1}^T \int_S \abs{\partial_t \tilde{K}} d\tilde{V} dt < \infty.
\end{equation*}
Lemma \ref{lem:odecompare} then gives
\begin{equation}\label{eqn:iterate}
	F(t)\leq C_1 \int_{t_0}^t F_H(t) dt \leq C'_1 \int_{t_0}^t \int_{t_0}^s \frac{1}{(r-t_0)^{1/q}} F(r)dr ds.
\end{equation}
It follows from \eqref{eqn:iterate} that $F\equiv 0$ for $t\in [t_0,T]$. To see this, we notice $F\leq C_2$ for some $C_2$ and integrate the right hand side to see
\begin{equation*}
	F(t)\leq C'_1 C_2 \frac{1}{(1-1/q)(2-1/q)} (t-t_0)^{2-1/q}.
\end{equation*}
Plugging this back into \eqref{eqn:iterate} will give
\begin{equation*}
	F \leq (C'_1)^2 C_2 \frac{1}{(1-1/q)(2-1/q)(3-2/q)(4-2/q)} (t-t_0)^{4-2/q}.
\end{equation*}
Repeating this process gives $F\equiv 0$ and proves Lemma \ref{lem:dtK}.
\subsection{Higher order regularity}\label{subsec:higher}

In the previous section, we have shown that $\partial_t K$ (or equivalently, $\partial^2_t u$) is in $\mathcal V^{2,\alpha,*}$. For higher $t$-derivatives, we can apply Theorem \ref{thm:smoothing} directly, because for any $l\geq 2$, the evolution equation of $\partial_t^l u$ is a linear equation whose coefficients involve only lower $t$-derivatives, which we may assume to be in $\mathcal V^{2,\alpha,*}$ by induction.

To be precise, we claim that for $l\geq 2$
\begin{equation*}
	\partial_t (\partial^l_t u) = e^{-2u} \tilde{\triangle} (\partial_t^l u) + P_l\cdot \partial_t^l u + Q_l
	\eqno(E_l)
\end{equation*}
where $P_l$ and $Q_l$ are polynomials of $\partial_t u,\cdots, \partial_t^{l-1}u$ with constant coefficients.
To see this, we compute directly to get
\begin{equation*}
	\partial_t (\partial_t u) = e^{-2u} \tilde{\triangle} (\partial_t u) - 2\partial_t u (\partial_t u -\frac{r}{2}) 
\end{equation*}
and
\begin{equation*}
	\partial_t (\partial_t^2 u) = e^{-2u} \tilde{\triangle} (\partial_t^2 u) + \partial_t^2 u \left( -6 \partial_t u +r \right)  -2 (\partial_t u)^2 (2\partial_t u -r),
\end{equation*}
which confirms the claim for $l=2$.
Assume the claim is true for $l$. Taking one more $t$-derivative gives
\begin{eqnarray*}
	\partial_t (\partial_t^{l+1} u) &=&  e^{-2u} \tilde{\triangle}(\partial_t^{l+1} u) + (-2 \partial_t u) (e^{-2u} \tilde{\triangle} \partial_t^l u)  \\
	&& + P_l \cdot \partial_t^{l+1} u + (\partial_t P_l) \cdot \partial_t^l u + \partial_t Q_l \\
	&=& e^{-2u} \tilde{\triangle}(\partial_t^{l+1} u) + (-2 \partial_t u) (\partial_t^{l+1} u- P_l \cdot \partial_t^l u -Q_l)  \\
	&& + P_l \cdot \partial_t^{l+1} u + (\partial_t P_l) \cdot \partial_t^l u + \partial_t Q_l.
\end{eqnarray*}
Hence, the claim is proved if we take
\begin{equation*}
	P_{l+1}= -2\partial_t u + P_l
\end{equation*}
and
\begin{equation*}
	Q_{l+1} = 2\partial_t u (P_l \cdot \partial_t^l u + Q_l) + (\partial_t P_l) \cdot \partial_t^l u +\partial_t Q_l.
\end{equation*} 

Given ($E_l$), we may prove Lemma \ref{lem:higher} by induction. Starting with ($E_2$), we can apply Theorem \ref{thm:smoothing} directly to it, because $u,\partial_t u, \partial_t^2 u$ are in $\mathcal V^{2,\alpha, [T-\delta,T]}$. Hence, for any $\delta'<\delta$, we have
\begin{equation*}
	\partial_t^3 u \in \mathcal V^{2,\alpha,[T-\delta',T]}.
\end{equation*}
The proof for higher order derivatives is similar and omitted.

\section{Asymptotic expansion of the solution}\label{sec:expansion}
The aim of this section is to prove Theorem \ref{thm:main3}. The proof is built on the previous knowledge that $\partial_t^l u$ is bounded for all $l=0,1,2,\cdots$. All discussions in this section are local, hence we take the polar coordinates $(\rho,\theta)$ on $B$ and regard $u(t)$ as a function of $(\rho,\theta)$.

\subsection{Formal consideration}\label{subsec:formal}
Since our aim is to study the expansion of $u$, we must first decide what terms should be included in the expansion. On one hand, we need to include sufficiently many terms so that $u(t)$ can be expanded as a series of such terms. On the other hand, we do not want to include more than what is absolutely necessary, because that will weaken our understanding on the regularity. The consideration in this subsection is a little formal, but it shall be fully justified when we prove Theorem \ref{thm:main3} in later subsections and it explains the reason why a particular term appears in the expansion.

First, let's recall the expansion of bounded harmonic functions defined on $B\setminus \set{0}$.
\begin{lem}
	\label{lem:harmonic}
	Suppose $u$ is a bounded harmonic function defined on $B\setminus \set{0}$, i.e. $\tilde{\triangle} u=0$. Then we have
	\begin{equation*}
		u(\rho,\theta) = a_0 + \sum_{k=1}^\infty \left( a_k \rho^{\frac{k}{\beta+1}} \cos k\theta + b_k \rho^{\frac{k}{\beta+1}} \sin k\theta \right)
	\end{equation*}
	for $\rho\in (0,1)$. Here $a_k$ and $b_k$ are real numbers determined by $u$.
\end{lem}
The proof is a well known argument of separation of variables and is omitted.

This is the starting point of our consideration. Namely, we should consider linear combinations of the terms in 
\begin{equation*}
	\mathcal T_h = \set{ \rho^{\frac{k}{\beta+1}}\cos k\theta, \rho^{\frac{k}{\beta+1}}\sin k\theta \,|\, k=0,1,2,\cdots}.
\end{equation*}
We denote the set of finite linear combination of terms in $\mathcal T_h$ by $\mbox{Span}(\mathcal T_h)$ and similar conventions apply to $\mathcal T_a$ and $\mathcal T$ to be defined later.

Next, we would like to include more terms so that some basic algebraic operations are closed. We define
\begin{equation*}
	\mathcal T_a =\set{ \rho^{\frac{k}{\beta+1}}\cos l \theta, \rho^{\frac{k}{\beta+1}}\sin l \theta \,|\, l=0,1,2,\cdots; \frac{k-l}{2}\in \mathbb N \cup \set{0}}.
\end{equation*}
We characterize $\mathcal T_a$ in the following lemma.
\begin{lem}
	\label{lem:algebra}
	$\mbox{Span}(\mathcal T_a)$ is the smallest vector space of functions which contains $\mbox{Span}(\mathcal T_h)$ and is multiplicatively closed.
\end{lem}
\begin{proof}
	It is straightforward to check that $\mbox{Span}(\mathcal T_a)$ is multiplicatively closed by computing the product of two terms in $\mathcal T_a$. Moreover, $\mathcal T_a$ contains $\mathcal T_h$ trivially. It suffices to show that it is the smallest set satisfying these properties. To see this, we compute
	\begin{equation*}
		\rho^{\frac{1}{\beta+1}}\cos \theta \cdot \rho^{\frac{1}{\beta+1}}\cos \theta = \rho^{\frac{2}{\beta+1}} \frac{\cos 2\theta +1}{2},
	\end{equation*}
	which implies that $\rho^{\frac{2}{\beta+1}}$ should be in $\mbox{Span}(\mathcal T_a)$. Multiplying $\rho^{\frac{2}{\beta+1}}$ to the terms in $\mathcal T_h$ repeatedly gives all terms in $\mathcal T_a$.
\end{proof}

Finally we define
\begin{equation*}
	\mathcal T =\set{ \rho^{2j+\frac{k}{\beta+1}}\cos l \theta, \rho^{2j+\frac{k}{\beta+1}}\sin l \theta \,|\, l,j,k=0,1,2,\cdots; \frac{k-l}{2}\in \mathbb N\cup\set{0}}.
\end{equation*}
Note that if $\beta\in \mathbb Q$, it is possible that there exists $j_1\ne j_2$ and $k_1\ne k_2$ such that
\begin{equation*}
	2j_1+\frac{k_1}{\beta+1} = 2j_2 + \frac{k_2}{\beta+1}.
\end{equation*}
The motivation behind the definition of $\mathcal T$ is explained in the next lemma.
\begin{lem}
	\label{lem:inverse} $\mbox{Span}(\mathcal T)$ is the smallest vector space of functions containing $\mbox{Span}(\mathcal T_h)$ such that (1) it is multiplicatively closed; (2)for each $u\in \mbox{Span}(\mathcal T)$, there is $v\in \mbox{Span}(\mathcal T)$ such that
	\begin{equation*}
		\tilde{\triangle} v = u.
	\end{equation*}
\end{lem}
\begin{proof}
(1) can be checked with direct computation. For (2), for each $u=\rho^\sigma \cos l\theta$ in $\mathcal T$, we compute
\begin{eqnarray*}
	\tilde{\triangle} \rho^{\sigma+2} \cos l\theta &=& (\partial_\rho^2 +\frac{1}{\rho}\partial_\rho - \frac{l^2}{\rho^2(\beta+1)^2}) \rho^{\sigma+2} \cos l\theta \\
	&=& ( (\sigma+2)^2 - \frac{l^2}{(\beta+1)^2}) \rho^\sigma \cos l\theta. 
\end{eqnarray*}
	By the definition of $\mathcal T$, $\sigma+2> \frac{l}{\beta+1}$ so the right hand side above in not zero, hence we may take
	\begin{equation*}
		v=( (\sigma+2)^2 - \frac{l^2}{(\beta+1)^2})^{-1}  \rho^{\sigma+2} \cos l\theta.
	\end{equation*}
	The computation works as well if we replace $\cos$ by $\sin$. Obviously, $\mathcal T_h\subset \mathcal T$ and the above computation also shows  that $Span(\mathcal T)$ is the smallest vector space with the required properties. 
\end{proof}

\subsection{Finite expansion}\label{subsec:finite}

As in the formulation of a Taylor expansion of a smooth function on $\Real^n$, we need to be precise about the difference between a smooth function and a Taylor polynomial. For that purpose, we shall define a class of functions $\tilde{O}(q)$ for any nonnegative real number $q$. In addition to the restriction on the decay of the function itself, we put some restrictions to the derivatives of the function, which is quite natural in our setting.

\begin{defn}\label{defn:O}
	A function $u$ defined in $B_{1/2}\setminus \set{0}$ is said to be in $\tilde{O}(q)$ for $q\in [0,\infty)$ if and only if there are constants $C_k$ for each $k=0,1,2,\cdots$ such that
	\begin{equation}\label{eqn:qk}
		\abs{\tilde{\nabla}^k u} \leq C_k \rho^{q - k} \quad\mbox{on}\quad B_{1/2}\setminus \set{0}.
	\end{equation}
\end{defn}
\begin{rem}
	We note that \eqref{eqn:qk} is equivalent to 
	\begin{equation*}
		\abs{(\rho\partial_\rho)^{k_1} \partial_{\theta}^{k_2} u}\leq C(k_1,k_2) \rho^q.
	\end{equation*}
\end{rem}

To define an expansion up to order $q>0$, we consider only the linear combination of functions in 
\begin{equation*}
	\mathcal T^q =\set{ \rho^{2j+\frac{k}{\beta+1}}\cos l \theta, \rho^{2j+\frac{k}{\beta+1}}\sin l \theta \,|\, l,j,k=0,1,2,\cdots; \frac{k-l}{2}\in \mathbb N \cup \set{0}; 2j+\frac{k}{\beta+1}< q}.
\end{equation*}
In other words, it is the subset of $\mathcal T$ which decays strictly faster than $\rho^q$ when $\rho\to 0$. 

\begin{defn}\label{defn:expansion}
	A function $u$ is said to have an expansion up to order $q$ if and only if there is a set of real numbers $a_v$ for each $v\in \mathcal T^q$ such that
	\begin{equation*}
		u= \sum_{v\in \mathcal T^q} a_v v + \tilde{O}(q) \quad\mbox{on}\quad B_{1/2}\setminus \set{0}.
	\end{equation*}
\end{defn}

As an example, we prove
\begin{lem}
	\label{lem:harmonicexp} Each bounded harmonic function $u$ on $B\setminus \set{0}$ has expansion up to order $q$ for any $q>0$.
\end{lem}
\begin{proof}
	Let $l_0$ be the smallest integer such that $\frac{l_0}{\beta+1}\geq q$. By Lemma \ref{lem:harmonic}, it suffices to prove
	\begin{equation*}
		R:=\sum_{l\geq l_0} \left( a_l \rho^{\frac{l}{\beta+1}-q} \cos l\theta + b_l \rho^{\frac{l}{\beta+1}-q}\sin l\theta \right) \in \tilde{O}(0).
	\end{equation*} 
	For any $k_1$ and $k_2$ in $\mathbb N\cup \set{0}$, we need to show that $(\rho\partial_\rho)^{k_1} \partial_\theta^{k_2} R$ is bounded on $B_{1/2}\setminus \set{0}$. Since $u$ and hence $R$ is smooth away from $0$, the trigonometric series converges nicely so that for $\rho\in (0,1)$
	\begin{equation}\label{eqn:series}
		(\rho \partial_\rho^{k_1} \partial_\theta^{k_2})R:=\sum_{l\geq l_0} \left( \tilde{a}_l \rho^{\frac{l}{\beta+1}-q} \cos^{(k_2)} l\theta + \tilde{b}_l \rho^{\frac{l}{\beta+1}-q}\sin^{(k_2)} l\theta \right),
	\end{equation} 
	where the exact formula for $\tilde{a}_l$ and $\tilde{b}_l$ is not important.
	Each term in the series of \eqref{eqn:series} is a continuous function of $(\rho,\theta)$ defined on $[0,1/2]\times S^1$. If we can show that the series converges uniformly on $[0,1/2]\times S^1$, then we are done. To see this, we use Abel's uniform convergence test and write for $\rho_0=1/2$
	\begin{equation*}
		\tilde{a_l} \rho^{\frac{l}{\beta+1}-q} \cos^{(k_2)} l\theta = \tilde{a_l} \rho_0^{\frac{l}{\beta+1}-q} \cos^{(k_2)} l\theta \cdot (\rho/\rho_0)^{\frac{l}{\beta+1}-q}.
	\end{equation*}
	The series of $\tilde{a}_l \rho_0^{\frac{l}{\beta+1}-q}\cos^{(k_2)}l\theta$ (forgetting about $\sin$ for simplicity) converges uniformly as functions (trivial in $\rho$) defined on $[0,1/2]\times S^1$, while the sequence of functions $(\rho/\rho_0)^{\frac{l}{\beta+1}-q}$ is uniformly bounded and decreases in $l$.
\end{proof}

\begin{lem}
	\label{lem:goodxi}
	If $f_1$ and $f_2$ both have expansions up to order $q$, then so do $f_1\pm f_2$, $f_1\cdot f_2$ and $e^{f_1}$.
\end{lem}
\begin{proof}
	The claim holds trivially for $f_1\pm f_2$. For $f_1\cdot f_2$, it suffices to notice that
	\begin{itemize}
		\item for any $\sigma\geq 0$ and $l\in \mathbb N\cup \set{0}$, $\rho^\sigma \cos l \theta$ is in $\tilde{O}(q)$ if and only if $\sigma\geq q$;
		\item for any $\sigma\geq 0$ and $l\in \mathbb N\cup \set{0}$, if $u\in \tilde{O}(q)$, then $\rho^\sigma \cos l \theta \cdot u \in \tilde{O}(q+\sigma)$;
		\item for $q_1,q_2\geq 0$, if $u_i\in \tilde{O}(q_i)$ for $i=1,2$, then $u_1\cdot u_2 \in \tilde{O}(q_1+q_2)$.
	\end{itemize}
	Instead of showing that $e^{f_1}$ has the required expansion, we prove something a little stronger. For any smooth function $F:\Real \to \Real$, we claim that $F\circ f_1$ has expansion up to order $q$. By changing $F(\cdot)$ to $F(c+\cdot)$, we may assume without loss of generality that the constant term in the expansion of $f_1$ vanishes. Namely, we can find $\xi \in Span(\mathcal T^q)$ with no constant term such that $f_1=\xi+\tilde{O}(q)$. 

	Recall that we have the Taylor expansion formula with the integral remainder,
	\begin{equation*}
		F(x)=\sum_{l=0}^n \frac{F^{(l)}(0)}{l!}x^l + \frac{1}{n!} \left( \int_0^1 F^{(n+1)}(tx) (1-t)^n dt \right) x^{n+1},
	\end{equation*}
	in which we choose $n$ so that $(n+1) \min\set{2,\frac{1}{\beta+1}}> q$. By what has been proved so far, we know
\begin{equation*}
	\sum_{l=0}^n \frac{F^{(l)}(0)}{l!} (f_1)^n
\end{equation*}
has an expansion up to order $q$. It remains to show 
\begin{equation}\label{eqn:remains}
	\left( \int_0^1 F^{(n+1)}(tf_1) (1-t)^n dt \right) (f_1)^{n+1}\in \tilde{O}(q).
\end{equation}
If $q< \min\set{2,\frac{1}{\beta+1}}$, then $\xi$ must be zero because $\mathcal T^q$ contains nothing but a constant function. In this case, $f_1\in \tilde{O}(q)$. If $q\geq \min\set{2,\frac{1}{\beta+1}}$, by our choice of $n$ and the fact that $f_1\in \tilde{O}( \min\set{2,\frac{1}{\beta+1}})$, we have $f_1^{n+1}$ is in $\tilde{O}(q)$. Hence, \eqref{eqn:remains} is reduced to
\begin{equation}\label{eqn:remains1}
	\left( \int_0^1 F^{(n+1)}(tf_1) (1-t)^n dt \right) \in \tilde{O}(0).
\end{equation}
By direct computation, one can check that \eqref{eqn:remains1} is true because of the smoothness of $F$ and the fact that $f_1\in \tilde{O}(0)$. 
\end{proof}

\subsection{Expansion of Conical Ricci flow solution}\label{subsec:expansion}
Assume that we have a solution $u$ given in Theorem \ref{thm:main1}.  By Theorem \ref{thm:main2}, for any $t>0$, we know that
\begin{equation*}
	\abs{\partial_t^l u(t)}\leq  C_k
\end{equation*}
for $l\in \mathbb N\cup\set{0}$.

Recall that we work on $B$ where $\tilde{K}=0$ and $\partial_t^l u$ satisfies the equations (after some rearrangement)
\begin{align*}
	\tilde{\triangle} u &= e^{2u}(\partial_t u -\frac{r}{2});
	\tag{$E_0$} \\
	\tilde{\triangle} (\partial_t u) &=e^{2u}\left( \partial_t (\partial_t u)+ \partial_t u (2\partial_t u -r)\right) ;
	\tag{$E_1$} \\
\intertext{and for $l\geq 2$} 
\tilde{\triangle} (\partial_t^l u)&=e^{2u}\left( \partial_t (\partial^l_t u) - P_l\cdot \partial_t^l u - Q_l\right),
	\tag{$E_l$}
\end{align*}
where $P_l$ and $Q_l$ are polynomials of $\partial_t u,\cdots, \partial_t^{l-1}u$ with constant coefficients.

We consider the following family of claims.

\noindent
{\bf Claim $\mathcal C^q$:} for each $l=0,1,2,\cdots$, $\partial_t^l u$ has an expansion up to order $q$.

Here are two easy observations. First, by Lemma \ref{lem:higher}, we know that the claim $\mathcal C^0$ is true. If $q_1<q_2$, then $\mathcal C^{q_2}$ is a stronger statement than the claim $\mathcal C^{q_1}$, hence, to show the claim $\mathcal C^q$ holds for any $q>0$, it suffices to justify $\mathcal C^{q_i}$ for a sequence $q_i$ going to $\infty$. This is done by a bootstrapping argument applied to the whole family of equations $(E_l)$.

\begin{lem}\label{lem:orderup}
	Suppose that $w$ and $f$ are smooth functions on $\bar{B}_1\setminus \set{0}$ and that $w$ is bounded and $f$ has an expansion up to order $q\geq 0$.
	If 
	\begin{equation*}
		\tilde{\triangle} w =f \qquad \mbox{on} \quad \bar{B}_1 \setminus \set{0},
	\end{equation*}
	then $w$ has an expansion up to order $q'$ for any $q'< 2+q$. When $q\ne \frac{k}{\beta+1}-2$ for any $k\in \mathbb N\cup \set{0}$, $w$ has an expansion up to order $q'=2+q$.
\end{lem}

Before the proof, we show how this lemma implies Theorem \ref{thm:main3}. 
If we know that the claim $\mathcal C^q$ is true, then Lemma \ref{lem:goodxi} implies that the right hand side of ($E_l$) for $l=0,1,\cdots$ has an expansion up to order $q$. By Lemma \ref{lem:orderup}, $\partial_t^l u$ has an expansion up to order $q'$ with $q'>q+1$ for all $l$. Hence, the claim $\mathcal C^{q'}$ for some $q'>q+1$ is true. Theorem \ref{thm:main3} then follows by repeatedly using the above argument.

The proof of Lemma \ref{lem:orderup} requires
\begin{lem}
	\label{lem:ode} If $f_o$ is a smooth function defined on $\bar{B}_1\setminus \set{0}$ and $f_o\in \tilde{O}(q)$ for some $q\geq 0$, then there exists another smooth function $w_o \in C^\infty(\bar{B}_1\setminus \set{0}) \cap \tilde{O}(q')$ with
	\begin{equation*}
		\tilde{\triangle} w_o =f_0.
	\end{equation*}
	Here $q'=2+q$ if $q\ne \frac{k}{\beta+1}-2$ for any $k\in \mathbb N\cup \set{0}$ and $q'$ can be any number smaller than $2+q$ if otherwise.
\end{lem}

\begin{proof}[Proof of Lemma \ref{lem:orderup}] By Definition \ref{defn:expansion}, there is $\xi_f\in \mbox{Span}(\mathcal T^q)$ such that
	\begin{equation*}
		f=\xi_f + \tilde{O}(q).
	\end{equation*}
	Lemma \ref{lem:inverse} implies the existence of $\xi_w\in \mbox{Span}(\mathcal T^{q+2})$ with
	\begin{equation*}
		\tilde{\triangle} \xi_w = \xi_f.
	\end{equation*}
	Since $f_o:=f-\xi_f$ is in $\tilde{O}(q)$, Lemma \ref{lem:ode} gives $w_o\in \tilde{O}(q')$ with $\tilde{\triangle} w_o =f_o$. Since $w,\xi_w$ and $w_o$ are all bounded on $B_1\setminus \set{0}$, we know that $w_h:= w- \xi_w - w_o$ is a bounded harmonic function on $B_1\setminus \set{0}$. By Lemma \ref{lem:harmonicexp}, all bounded harmonic functions have expansion up to any order. Since $w=w_h+\xi_w+w_o$, $w$ has an expansion up to order $q'$ for $q'$ given in Lemma \ref{lem:ode}. 
\end{proof}

The rest of this section is devoted to the proof of Lemma \ref{lem:ode}.

Let's first do some formal computation to motivate the proof. We shall justify later that this gives the solution we want. Assume we have an expansion
\begin{equation*}
	w_o=\sum_{l=0}^\infty A_l(\rho) \cos (l\theta) + B_l(\rho) \sin (l\theta),
\end{equation*}
which is convergent in some suitable sense such that
\begin{equation*}
	\tilde{\triangle} w_o = \sum_{l=0}^{\infty} (L_l A_l) \cos (l\theta) + (L_l B_l) \sin (l\theta)
\end{equation*}
where
\begin{equation*}
	L_l := \partial_\rho^2 + \frac{1}{\rho}\partial_\rho - \frac{l^2}{\rho^2 (\beta+1)^2}.
\end{equation*}
Hence, we are motivated to solve the equations
\begin{equation*}
	L_l A_l = a_l\quad \mbox{and} \quad L_l B_l=b_l,
\end{equation*}
if $a_l$ and $b_l$ are given by
\begin{equation*}
	f_o= \sum_{l=0}^\infty a_l(\rho) \cos (l\theta) + b_l(\rho) \sin (l\theta).
\end{equation*}
Notice that it follows from the theory of trigonometric series that $f_o$ is in $\tilde{O}(q)$ if and only if 
\begin{equation}\label{eqn:goodfo}
	\abs{l^{k_2} (\rho\partial_\rho)^{k_1} a_l(\rho)}+\abs{l^{k_2} (\rho\partial_\rho)^{k_1} b_l(\rho)}\leq C(k_1,k_2) \rho^q
\end{equation}
for constants $C(k_1,k_2)$ depending on $f_o$.

\begin{lem}\label{lem:lkvw}
Suppose $a_*:(0,1]\to \Real$ satisfies that
\begin{equation}
	\abs{(\rho\partial_\rho)^{k_1} a_*(\rho)} \leq C_1(k_1) \rho^q\qquad \mbox{for} \quad k_1=0,1,2,\cdots.	
	\label{eqn:goodastar}
\end{equation}
If 
\begin{equation}
	-\frac{l}{\beta+1}+2+q \ne 0,
	\label{eqn:goodalpha}
\end{equation}
then there exist constants $C_2(k_1)$ depending on $\abs{-\frac{l}{\beta+1}+2+q}$ and $C_1(\cdot)$ and a function $A_*:(0,1]\to \Real$ such that
\begin{equation}\label{eqn:Lkv}
	L_l A_*=a_*
\end{equation}
and
\begin{equation}\label{eqn:goodAstar}
	\abs{(\rho\partial_\rho)^{k_1} A_*(\rho)} \leq C_2(k_1) \rho^{q+2}\qquad \mbox{for} \quad k_1=0,1,2,\cdots.	
\end{equation}
\end{lem}

\begin{rem}\label{rem:c1c2}
The dependence of $C_2(\cdot)$ on $C_1(\cdot)$ is linear in the sense that if we multiply every one of $C_1(\cdot)$ by a positive constant $\lambda$, then $C_2(\cdot)$ is multiplied by the same constant. It follows either from the proof below by tracing the dependency of constants carefully or from the linearity of the statement of the lemma, i.e. we can apply it to $\lambda a_*$ instead of $a_*$.
\end{rem}

\begin{proof}
	For simplicity, we write $c$ for $\frac{l}{\beta+1}$. First, we note that the solution of \eqref{eqn:Lkv} can be explicitly written down. We assume that the solution is of the form $A_*(\rho)=h(\rho) \rho^c$, then \eqref{eqn:Lkv} is equivalent to
	\begin{equation*}
		(2c+1) h' \rho^{c-1} + h'' \rho^c = a_*.
	\end{equation*}
	Hence,
	\begin{equation*}
		h'(\rho)= \rho^{-2c-1}\left( h'(1) + \int_1^\rho a_*(t)t^{c+1}dt \right).
	\end{equation*}
	Here $h'(1)$ is some constant. We can choose it to be anything we want, since it suffices for the proof of the lemma to give one solution. We choose $h'(1)$ so that
	\begin{equation*}
		h'(1) + \int_1^0 a_*(t)t^{c+1}dt=0.
	\end{equation*}
Hence,
	\begin{equation*}
		h'(\rho)= \rho^{-2c-1}\int_0^{\rho} a_*(t) t^{c+1}dt.
	\end{equation*}
	By \eqref{eqn:goodastar}, we get
	\begin{equation}\label{eqn:hprime}
		\abs{h'(\rho)}\leq \frac{C_1(0)}{c+2+q } \rho^{-c+1+q}.
	\end{equation}
	On the other hand, we have
	\begin{equation*}
		A_*(\rho)=\rho^c \left( h(1)+\int_1^\rho h'(t)dt \right).
	\end{equation*}
	Here $h(1)$ is another constant at our disposal. 

	By \eqref{eqn:goodalpha}, we have two possible cases:

	{\bf Case 1:} $c>2+q$. Take $h(1)=0$ and compute
	\begin{eqnarray*}
		\abs{A_*}&\leq& \rho^c \frac{C_1(0)}{(c+2+q)\abs{-c+2+q}} \left( \rho^{-c+2+q}-1 \right)\\
		&\leq& C_2(0) \rho^{2+q},
	\end{eqnarray*}
	where $C_2(0)=\frac{C_1(0)}{(c+2+q)\abs{-c+2+q}} $.

	{\bf Case 2:} $c<2+q$. Take $h(1)$ satisfying
	\begin{equation*}
		h(1)+\int_1^0 h'(t) dt =0.
	\end{equation*}
Hence,
\begin{equation*}
	\abs{A_*}\leq \rho^c \frac{C_1(0)}{(c+2+q)\abs{-c+2+q}} \rho^{-c+2+q} \leq C_2(0)\rho^{2+q}
\end{equation*}
for the same $C_2(0)$ as in Case 1.

For the derivatives of $A_*$, we compute
\begin{equation*}
	(\rho \partial_\rho) A_*(\rho)= c A_* + \rho^{c+1} h'(\rho),
\end{equation*}
which has the correct order of decay by \eqref{eqn:hprime}. Moreover, $C_2(1)$ in \eqref{eqn:goodAstar} is a linear combination of $C_2(0)$ and $C_1(0)$.

For $k_1>1$, we rewrite the equation $L_l A_*=a_*$ as 
\begin{equation*}
	(\rho\partial_\rho)^2 A_* - \frac{l^2}{(\beta+1)^2} A_* = \rho^2 a_*,
\end{equation*}
The estimate \eqref{eqn:goodAstar} for $k_1=2$ follows from the above equation and \eqref{eqn:goodastar} directly, while for the case $k_1>2$, we take $\rho\partial_\rho$ repeatedly on both sides of the above equation. 
\end{proof}

Now, we apply Lemma \ref{lem:lkvw} to both $a_l$ and $b_l$ to get $A_l$ and $B_l$. More precisely, for $a_l$ and any $k_2$, \eqref{eqn:goodfo} implies
\begin{equation*}
	\abs{(\rho\partial_\rho)^{k_1} a_l(\rho)}\leq \frac{C(k_1,k_2)}{l^{k_2}} \rho^q.
\end{equation*}
By Lemma \ref{lem:lkvw} and Remark \ref{rem:c1c2}, we have $A_l(\rho)$ such that
\begin{equation*}
	L_l A_l=a_l
\end{equation*}
and
\begin{equation*}
	\abs{(\rho\partial_\rho)^{k_1} A_l(\rho)}\leq \frac{C'(k_1,k_2)}{l^{k_2}} \rho^q.
\end{equation*}
Similar arguments work for $b_l$ as well. In summary, we get
\begin{equation}\label{eqn:goodwo}
	\abs{l^{k_2} (\rho\partial_\rho)^{k_1} A_l(\rho)}+\abs{l^{k_2} (\rho\partial_\rho)^{k_1} B_l(\rho)}\leq C'(k_1,k_2) \rho^q.
\end{equation}
This not only justifies the convergence of the series
\begin{equation*}
	w_o= \sum_{l=0}^\infty A_l(\rho) \cos l\theta + B_l(\rho)\sin l\theta,
\end{equation*}
but also the computation
\begin{equation*}
	\tilde{\triangle} w_o = \sum_{l=0}^\infty (L_l A_l) \cos l\theta + (L_l B_l) \sin l\theta,
\end{equation*}
so that $\tilde{\triangle }w_o=f_o$. Moreover, the decay of $w_o$, $ (\rho \partial_\rho)^{k_1}\partial_\theta^{k_2} w_o$ also follows from \eqref{eqn:goodwo}.

\appendix

\section{Proof of Theorem \ref{thm:dg1}}\label{app:dg1}
Similar to the proof of Theorem 10.1 in Chapter III of \cite{ladyzhenskaya1968linear}, the proof relies on the results in Section 7 Chapter II in the same book. In fact, the authors wrapped up the argument of Di Giorgi iteration in Theorem 7.1 there and the assumptions needed for the argument were summarized in the definition of the class of functions, $\mathcal B_2(Q_T,M,\gamma,r,\delta,\kappa)$. Our strategy of proving Theorem \ref{thm:dg1} is to show that $u$ is a function in $\mathcal B_2(Q_T,M,\gamma,r,\delta,\kappa)$ and apply Theorem 7.1 there.

To be precise, we write down the definition of $\mathcal B_2$ and Theorem 7.1 explicitly and simplify the definition and statement a little since we are not interested in equations of the most general form. For the convenience of readers who may want to consult the book \cite{ladyzhenskaya1968linear}, we use the notations there as much as possible. 

\subsection{The class of functions $\mathcal B_2$}

Suppose $\Omega$ is an open set in $\Real^2$, which is just $B$ for our purpose. Let $Q_T=\Omega\times [0,T]$ for some $T>0$. For fixed $(x_0,t_0)\in Q_t$, $K_\rho$ is the set $\set{\abs{x-x_0}<\rho}$ and $Q(\rho,\tau)=K_\rho \times (t_0,t_0+\tau)$ and it is said to have a vertex at $(x_0,t_0)$.  Following \cite{ladyzhenskaya1968linear}, we define
\begin{equation*}
	\norm{u}_{2,\Omega}= \left( \int_\Omega u^2 dx \right)^{1/2}
\end{equation*}
and
\begin{equation*}
	\norm{u}_{Q_T} = \sup_{0\leq t\leq T} \norm{u(t)}_{2,\Omega} + \norm{u_x}_{2,Q_T}.
\end{equation*}
Here $u_x$ means the partial derivatives of $u$ with respect to $x_1$ and $x_2$.
\begin{defn}
	\label{defn:b2}
	A function $u:Q_T\to \Real$ is said to be in the class $\mathcal B_2(Q_T,M,\gamma)$ if it satisfies
	\begin{enumerate}[(a)]
		\item there is constant $C>0$ such that
			\begin{equation*}
				\norm{u}_{Q_T}\leq C
			\end{equation*}
			and that $\int_\Omega u^2 dx$ is a continuous function of $t$;
		\item $\sup_{Q_T} \abs{u}\leq M$;

		\item for $w(x,t)=\pm u(x,t)$, we have 
			\begin{eqnarray}\label{eqn:71}
				&& \max_{t_0\leq t\leq t_0+\tau} \norm{w^{(k)(x,t)}}_{2,K_{(1-\sigma_1)\rho}}^2\\ \nonumber
				&&\leq \norm{w^{(k)}(x,t)}_{2,K_\rho}^2
				 + \gamma \left[ (\sigma_1 \rho)^2 \norm{ w^{(k)}}_{2,Q(\rho,\tau)}^2 + \mu^{3/4}(k,\rho,\tau) \right]
			\end{eqnarray}
			and
			\begin{eqnarray}
				\label{eqn:72}
				&&\norm{w^{(k)}}^2_{Q( (1-\sigma_1)\rho,(1-\sigma_2)\tau)} \\ \nonumber
				&\leq& \gamma \left( [ (\sigma_1 \rho)^{-2} + (\sigma_2\tau)^{-1}] \norm{w^{(k)}}^2_{2,Q(\rho,\tau)} + \mu^{3/4}(k,\rho,\tau) \right),
			\end{eqnarray}
			Here in the above $w^{(k)}=\max\set{w-k,0}$, $\sigma_1,\sigma_2$ are any number in $(0,1)$, $(x_0,t_0)\in Q_T$, $\rho,\tau$ are any positive number such that $Q(\rho,\tau)\subset Q_T$ and
			\begin{equation*}
				\mu(k,\rho,\tau)= \int_{t_0}^{t_0+\tau} \abs{A_{k,\rho}(t)} dt
			\end{equation*}
			where $A_{k,\rho}(t)=K_\rho\cap \set{w(x,t)>k}$ and $\abs{A_{k,\rho}(t)}$ is the measure of $A_{k,\rho}(t)$.
	\end{enumerate}
\end{defn}

\begin{rem}
	This is the same as the definition in Section 7 Chapter III of \cite{ladyzhenskaya1968linear} except that we have chosen $n=2$, $\delta=+\infty$, $r=q=4$ and $\kappa=1/2$.
\end{rem}

Here is a simplified version of Theorem 7.1 in \cite{ladyzhenskaya1968linear}.
\begin{thm}\label{thm:71}
	[Theorem 7.1 in Chapter II of \cite{ladyzhenskaya1968linear}] Suppose $u$ is in $\mathcal B_2(Q_T,M,\gamma)$. There are $\theta>0$ and $\alpha\in (0,1)$ depending only on $\gamma$ such that the following holds. For some $\rho_0>0$ and $(x_0,t_0)\in Q_T$ such that $Q(\rho_0,\theta \rho_0^2)$ with vertex at $(x_0,t_0)$ lies in $Q_T$, we have for $\rho<\rho_0$,
	\begin{equation*}
		\mbox{osc}\set{u,Q_\rho}\leq c\rho^\alpha \rho_0^{-\alpha},
	\end{equation*}
	where $\mbox{osc}\set{u,Q_\rho}$ is the oscillation of $u$ on $Q_\rho:=Q(\rho,\theta\rho^2)$.
	Here $c$ depends on $\gamma, \rho_0$ and $M$.
\end{thm}

There is a remark after that theorem in \cite{ladyzhenskaya1968linear}: if in the definition of $\mathcal B_2(Q_T,M,\gamma)$, \eqref{eqn:71} and \eqref{eqn:72} holds only for $Q(\rho,\tau)$ with $\tau\leq \theta_2 \rho^2 \leq \theta_2 \rho_0^2$, then the result of the above theorem only holds for $\rho\leq \rho_0$ with the $\theta$ in the definition of $Q_\rho$ replaced by $\min \set{\theta_2,\theta}$.

\subsection{Check the assumptions in Definition \ref{defn:b2}}
Let $u$ be as in Theorem \ref{thm:dg1}. It is easy to see that (a) and (b) in Definition \ref{defn:b2} holds, except that we need to apply Lemma \ref{lem:basic2} to show $\int_{\Omega} u^2 dx$ is an absolutely continuous function of $t$. The rest of this section is devoted to check (c). Moreover, we want to make sure that $\gamma$ there depends only on $\lambda$.

By \eqref{eqn:weak}, we apply Lemma \ref{lem:basic2} with $\Psi(u)=(u^{(k)})^2$ where $u^{(k)}:=\max\set{u-k,0}$ and $\varphi(x,t)= \xi^2(x,t) (\sqrt{g})^{-1}$ to get
\begin{eqnarray}\label{eqn:b2weak}
	 &&\frac{d}{dt} \int_{B} (u^{(k)}\xi)^2 (\sqrt{g}) dx \\ \nonumber 
	&=& \int_B 2 u^{(k)} (\partial_i (g^{ij}\sqrt{g}\partial_j u) +(\sqrt{g}) b u+  (\sqrt{g}) f) \xi^2  dx \\ \nonumber
	&& + \int_B (u^{(k)}\xi )^2 (\partial_t (\sqrt{g})) dx + \int_B (u^{(k)})^2 2 \xi \partial_t \xi (\sqrt{g}) dx.
\end{eqnarray}
We study each term above separately.
\begin{eqnarray}\label{eqn:b21}
	&& \int_B 2 u^{(k)} \partial_i (g^{ij}\sqrt{g} \partial_j u) \xi^2 dx\\ \nonumber
	&=&  - \int_B g^{ij}\partial_i u^{(k)} \partial_j u \sqrt{g} \xi^2 +4 u^{(k)} g^{ij} \sqrt{g} \partial_j u \xi \partial_i \xi dx \\ \nonumber
	&\leq& - c_1 \int_B \abs{\partial_i u^{(k)}}^2 \xi^2 dx + C_3 \int_B (u^{(k)})^2 \abs{\partial_i \xi}^2 dx.
\end{eqnarray}
Here $c_1$ and $C_3$ are two constants depending only on $\lambda$. We also note that the integration by parts are justified by the fact that $\int_B \abs{\partial_i u}^2 dx < \infty$ and a similar argument as in Lemma \ref{lem:basic}.

\begin{eqnarray}
	\label{eqn:b22}
	&& \int_B 2u^{(k)} \sqrt{g} (bu+f) \xi^2 + ( u^{(k)} \xi)^2 (\partial_t \sqrt{g}) dx
	\\ \nonumber &\leq& C_4 (M^2+1) \int_{\set{u\geq k}} \xi dx. 
\end{eqnarray}
Here $C_4$ depends on $\lambda$ and $C^0$ norm of $\partial_t \sqrt{g}$, $b$ and $f$ and $M$ is $\norm{u}_{C^0(B\setminus \set{0}\times [0,T])}$.

Putting \eqref{eqn:b2weak}, \eqref{eqn:b21} and \eqref{eqn:b22} together yields
\begin{eqnarray*}
	&& \frac{d}{dt} \int_{B} (u^{(k)}\xi)^2 (\sqrt{g}) dx \\
	&\leq& - c_1 \int_B \abs{\partial_i u^{(k)}}^2 \xi^2 dx + C_5 \int_B (u^{(k)})^2 ( \abs{\partial_i \xi}^2+ \xi \abs{\partial_t \xi}) dx\\
	&&   + C_4 (M^2+1) \int_{\set{u\geq k}}\xi dx.
\end{eqnarray*}
Here $C_5$ depends only on $\lambda$.

Let $\rho_0$ be a positive number to be determined later. Consider $K_\rho$ as in Definition \ref{defn:b2} with $\rho\leq \rho_0$ and assume that $\xi:\Omega \times [t_0,t_0+\tau] \to [0,1]$ vanishes in a neighborhood of $\partial K_\rho \times [t_0,t_0+\tau]$ and outside $K_\rho\times [t_0,t_0+\tau]$. Now, we integrate the above inequality from $t_0$ to $t_0+\tau$ to get
\begin{eqnarray*}
	&& \int_{K_\rho} ( u^{(k)}(x,t_0+\tau)\xi(x,t_0+\tau))^2 dx + c'_1 \iint_{Q(\rho,\tau)} \abs{\partial_i u^{(k)}}^2 \xi^2 dx \\ 
	&\leq& \int_{K_\rho} ( u^{(k)}(x,t_0)\xi(x,t_0))^2 dx + C'_5 \iint_{Q(\rho,\tau)} (u^{(k)})^2 ( \abs{\partial_i \xi}^2+ \xi \abs{\partial_t \xi}) dx \\
	&& + C'_4 (M^2+1) \int_{t_0}^{t_0+\tau} \int_{\set{u\geq k}} \xi dx dt.
\end{eqnarray*}
Here $c'_1$, $C'_4$ and $C'_5$ is obtained from $c_1$, $C_4$ and $C_5$ by multiplying a constant depending only on $\lambda$.
Now we pick $\rho_0$ satisfying 
\begin{equation*}
	C'_4 (M^2+1) (T \abs{K_{\rho_0}})^{1/4} =1,
\end{equation*}
so that
\begin{eqnarray*}
	&& \int_{K_\rho} ( u^{(k)}(x,t_0+\tau)\xi(x,t_0+\tau))^2 dx + c'_1 \iint_{Q(\rho,\tau)} \abs{\partial_i u^{(k)}}^2 \xi^2 dx \\ 
	&\leq& \int_{K_\rho} ( u^{(k)}(x,t_0)\xi(x,t_0))^2 dx + C'_5 \iint_{Q(\rho,\tau)} (u^{(k)})^2 ( \abs{\partial_i \xi}^2+ \xi \abs{\partial_t \xi}) dx \\
	&& + \left( \int_{t_0}^{t_0+\tau} \int_{\set{u\geq k}} \xi dx dt\right)^{3/4}.
\end{eqnarray*}
This is exactly (7.5) in Chapter II of \cite{ladyzhenskaya1968linear} for our choice of $r,q,\kappa$. By Remark 7.2 there, \eqref{eqn:71} and \eqref{eqn:72} follows from it directly. Now Theorem \ref{thm:71} (Theorem 7.1 in Chapter II of \cite{ladyzhenskaya1968linear} and the remark following it) implies Theorem \ref{thm:dg1}.


\section{Proof of Lemma \ref{lem:moser}}\label{app:moser}
First of all, we note that it suffices to prove for $B_1$, because $B_i (i>1)$ contains no singular point and the result should be considered known. On the other hand, the proof given below for $B_1$ works for $B_i(i>1)$ as well with no modification at all. 

Recall that on $B_1$ for $t<1$, the metric is standard cone metric and both \eqref{eqn:moser1} and \eqref{eqn:moser2} are invariant under parabolic scaling, so it suffices to prove the lemma on $B$, i.e. the unit disc on $\Real^2$ equipped with the metric
\begin{equation*}
	d\rho^2+ \rho^2(1+\beta)^2 d\theta^2.
\end{equation*}
Note that when we scale up from $B_1=B_{\sqrt{t}}(p)$ to $B$, the norm of $b$ and $\partial_t a$ becomes smaller so the assumptions remain true. More precisely, we prove 
\begin{lem}
	Suppose that $a,b\in \mathcal V^{0,\alpha,[0,1]}(B)$ and $u\in \mathcal V^{2,\alpha,[0,1]}(B)$ is a weak solution to
	\begin{equation*}
		\partial_t u =a(x,t) \tilde{\triangle} u + bu.
	\end{equation*}
	Assume 
	\begin{equation*}
		\max_{B \times [0,1]} \abs{\partial_t a} + \abs{a^{-1}} < C_1
	\end{equation*}
	for some $C_1$. Then for any $q>1$, there is $C_2$ depending only on $C_1$ and the $\mathcal V^{0,\alpha,[0,1]}$ norms of $a$ and $b$ such that
	\begin{equation}\label{eqn:app1}
		\max_{B_{1/2} \times [3/4,1]} \abs{u} \leq {C_2} \left( \int_0^1 \int_{B} \abs{u}^q d\tilde{V} dt\right)^{1/q}
	\end{equation}
	and
	\begin{equation}\label{eqn:app2}
		\norm{\tilde{\nabla} u(1)}_{L^2(B_{1/4},\tilde{g})} \leq {C_2}  \left( \int_0^1 \int_{B} \abs{u}^q d\tilde{V} dt\right)^{1/q}.
	\end{equation}
\end{lem}

We prove \eqref{eqn:app1} first for $q\geq 2$. The proof is well known except that we need to justify some integration by parts and switching of the order of integrals.  

Take some $\sigma\in (0,1)$ fixed. Let $q_i=(3/2)^{i-1}q$, $t_i=(2\sigma-\sigma^2)(1-\frac{1}{2^{i-1}})$ and $r_i=(1-\sigma)+ \frac{\sigma}{2^{i-1}}$. For simplicity, denote $Q_i=B_{r_i}\times [t_i,1]$. Suppose that  $\varphi_i(x,t)$ is a smooth cut-off function defined in $Q_i$ satisfying (1) $\varphi_i(x,t)\equiv 0$ for $t= t_i$ or $\abs{x}=r_i$; (2) $\varphi_i(t)\equiv 1$ for $(x,t)\in Q_{i+1}$; (3) $0\leq \partial_t \varphi_i + \varphi_i^{-1} \abs{\tilde{\nabla} \varphi_i}^2\leq \sigma^{-2}{4^{i+2}}$.

For $i\geq 1$, multiplying both sides of the equation by $\varphi_i u_+^{q_i-1} a^{-1}$ and integrating over $B_{r_i}\times [t_{i},t]$ for any $t\in [t_{i},1]$ gives
\begin{equation*}
	\iint_{B_{r_i}\times [t_i,t]} \partial_tu u_+^{q_i-1} \varphi_i a^{-1} d\tilde{V} ds = \iint_{B_{r_i}\times [t_i,t]} \varphi_i u_+^{q_i-1} \tilde{\triangle} u + \varphi_i  a^{-1} b u_+^{q_i} d\tilde{V} ds.
\end{equation*}
First, by the definition of $\mathcal V^{2,\alpha,[0,1]}$, the integral on the left hand side (hence on the right hand side) is absolutely integrable. Hence, the Fubini theorem allows us to write 
\begin{equation}\label{eqn:write}
	\int_{B_{r_i}} \int_{t_{i}}^t \partial_tu u_+^{q_i-1} \varphi_i a^{-1} ds d\tilde{V} = \int_{t_{i}}^t \int_{B_{r_i}} \varphi_i u_+^{q_i-1} \tilde{\triangle} u + \varphi_i  a^{-1} b u_+^{q_i} d\tilde{V} ds.
\end{equation}
An integration by parts by Lemma \ref{lem:basic} (see the footnote in the proof of Lemma \ref{lem:maximum}) gives
\begin{eqnarray*}
	\int_{t_{i}}^t \int_{B_{r_i}} \varphi_i u_+^{q_i-1} \tilde{\triangle}u d\tilde{V} ds &\leq &  - \frac{4(q_i-1)}{q_i^2} \int_{t_{i}}^t \int_{B_{r_i}}  \varphi_i \abs{\tilde{\nabla} u_+^{q_i/2}}^2 d\tilde{V} ds \\
	&& + \int_{t_{i}}^t\int_{B_{r_i}} \tilde{\nabla} \varphi_i \tilde{\nabla} u_+ u_+^{q_i-1} d\tilde{V} ds \\ 
	&\leq&  - \frac{2(q_i-1)}{q_i^2} \int_{t_{i}}^t \int_{B_{r_i}}  \varphi_i \abs{\tilde{\nabla} u_+^{q_i/2}}^2 d\tilde{V} ds \\
	&& + q_i^2 \int_{t_i}^t \int_{B_{r_i}} \varphi_i^{-1} \abs{\tilde{\nabla} \varphi_i}^2 u_+^{q_i} d\tilde{V}ds.
\end{eqnarray*}
Using the bound of $\partial_t a$ and Lemma \ref{lem:basic2}, the left hand side of \eqref{eqn:write} is estimated by
\begin{eqnarray*}
	\int_{B_{r_i}}\int_{t_i}^t \partial_t u u_+^{q_i-1} \varphi_i a^{-1} ds d\tilde{V} &\geq & \left. \int_{B_{r_i}} u_+^{q_i} \varphi_i a^{-1} d\tilde{V}\right|^t_{t_i} \\
	&& - C\sigma^{-2} 4^i \iint_{Q_i}  u_+^{q_i} d\tilde{V} ds.
\end{eqnarray*}
By the arbitrariness of $t\in [t_i,1]$ and the boundedness of $a,a^{-1},b$, we get
\begin{equation*}
	\max_{t\in [t_{i+1},1]} \int_{B_{r_{i+1}}} u_+^{q_i} d\tilde{V} + \int_{t_{i+1}}^{1} \int_{B_{r_{i+1}}} \frac{2(q_i-1)}{q_i^2} \abs{\tilde{\nabla} u_+^{q_i/2}}^2 d\tilde{V} ds \leq C(q) \sigma^{-2} (16)^i\iint_{Q_i}  u_+^{q_i} d\tilde{V} ds.
\end{equation*}
Here we have used the definition of $\varphi_i$ and estimated $q_i^2$ from above by $4^i q^2$.

For simplicity, we write $w_i= u_+^{q_i/2}$ and by the H\"older and the Sobolev inequalities,
\begin{eqnarray*}
	\int_{B_{r_{i+1}}} w_i^3 d\tilde{V} &\leq& \left( \int_{B_{r_{i+1}}} w_i^4 d\tilde{V} \right)^{1/2} \left( \int_{B_{r_{i+1}}} w_i^2 \right)^{1/2} \\
	&\leq& C\left( \int_{B_{r_{i+1}}} \abs{\tilde{\nabla} w_i}^2 + w_i^2 d\tilde{V}\right) \left( \int_{B_{r_{i+1}}} w_i^2 \right)^{1/2}.
\end{eqnarray*}
A simply way to see that the Sobolev inequality holds is to notice that the cone metric on $B_{r_i}$ is quasi-isometric to the flat metric on a ball of radius $r_i$ in $\Real^2$. Of cause the constant above depends on $\beta$.

Integrating from $t_{i+1}$ to $1$, we get
\begin{eqnarray*}
	\iint_{Q_{i+1}} u_+^{q_{i+1}} d\tilde{V} &\leq& \left(\int_{t_{i+1}}^{1}\int_{B_{r_{i+1}}} \abs{\tilde{\nabla} u_+^{q_i/2}}^2 + u_+^{q_i} d\tilde{V}\right) \cdot \left( \max_{t\in [t_{i+1},1]} \int_{B_{i+1}} u_+^{q_i} d\tilde{V}\right)^{1/2} \\
	&\leq&C(q) (256)^i \sigma^{-3} \left( \iint_{Q_i} u_+^{q_i} d\tilde{V} ds\right)^{3/2}.
\end{eqnarray*}
Some routine iteration yields that
\begin{equation*}
	\norm{u_+}_{C^0(B_{1-\sigma}\times [1-(1-\sigma)^2,1])}\leq C \sigma^{-6/q} \norm{u_+}_{L^q(B_1\times [0,1])}.
\end{equation*}
The same applies to the negative part so that (for $q\geq 2)$
\begin{equation}\label{eqn:q2}
	\norm{u}_{C^0(B_{1-\sigma}\times [1-(1-\sigma)^2,1])}\leq C \sigma^{-6/q} \norm{u}_{L^q(B_1\times [0,1])}.
\end{equation}
There is also a standard iteration process to prove the case for $q\in (0,2)$ as we learned from Li and Schoen \cite{li1984p}. For completeness, we outline the argument below. We recycle $r_i$ and $Q_i$ by setting for $i=1,2,\cdots$
\begin{equation*}
	r_i=1-\frac{1}{2^i}\qquad Q_i= B_{r_i}\times [1-r_i^2,1].
\end{equation*}
We take $\sigma_i$ such that
\begin{equation*}
	(1-\sigma_i) r_{i+1}= r_{i},
\end{equation*}
namely
\begin{equation*}
	\sigma_i =\frac{1}{2^{i+1}-1}.	
\end{equation*}
We apply \eqref{eqn:q2} with $q=2$ to get
\begin{eqnarray*}
	\norm{u}_{C^0(Q_1)}&\leq & C \sigma_1^{-3} \norm{u}_{L^2(Q_2)} \\
	&=& C \sigma_1^{-3} \norm{u}_{C^0(Q_2)}^{\frac{2-q}{2}} \norm{u}_{L^q(B_1\times [0,1])}^{\frac{q}{2}} \\
	&\leq& C (\sigma_1^{-3}) (\sigma_2^{-3})^{\frac{2-q}{2}}(\sigma_2^{-3})^{(\frac{2-q}{2})^2}  \cdots \norm{u}_{L^q(B_1\times [0,1])}^{\frac{q}{2}(1+ (\frac{2-q}{2}) +(\frac{2-q}{2})^2 \cdots )}.
\end{eqnarray*}
\eqref{eqn:app1} is proved by checking the series in the above equation converges.

With \eqref{eqn:app1}, we apply Theorem \ref{thm:dg} to see that there is some $\alpha'>0$ and $C>0$ such that
\begin{equation*}
	\norm{u(1)}_{C^{\alpha'}(B_{1/4})}\leq C \norm{u}_{C^0(B_{1/2}\times [3/4,1])}.
\end{equation*}
Using the $(\rho,\theta)$ coordinates,
\begin{equation*}
	\abs{u(\rho,\theta,1)- u(0,\theta,1)}\leq C \norm{u}_{C^0(B_{1/2}\times [3/4,1])} \rho^\alpha.
\end{equation*}
On the other hand, parabolic interior estimate shows that there is another $C>0$ depending on $\mathcal V^{0,\alpha}$ norm of $a$ and $b$ and $C^0$ norm of $a^{-1}$ such that for $\rho\leq 1/4$,
\begin{equation*}
	\abs{\tilde{\nabla}^2 u (\rho,\theta,1)}\leq C \rho^{-2}  \norm{u}_{C^0(B_{1/2}\times [3/4,1])}.
\end{equation*}
Interpolation gives 
\begin{equation*}
	\abs{\tilde{\nabla} u (\rho,\theta,1)} \leq \frac{C}{\rho^{1-\alpha/2}} \norm{u}_{C^0(B_{1/2}\times [3/4,1])},
\end{equation*}
which we integrate over $B_{1/4}$ to get \eqref{eqn:app2}.

\section{Proof of Lemma \ref{lem:poisson}}\label{app:poisson}

\begin{proof}
	By definition of $\tilde{g}$, there is a smooth metric $\bar{g}$ on $S$ such that
	\begin{equation*}
		\tilde{g}= w g \quad \mbox{on } S
	\end{equation*}
	and
	\begin{equation*}
		w= \rho^{2\beta}
	\end{equation*}
	in a neighborhood of $p$. The equation we want to solve is equivalent to
	\begin{equation}\label{eqn:barg}
		\bar{\triangle} u = w f,
	\end{equation}
	which lies in $L^p(S,\bar{g})$ for some $p>1$ because $f$ is bounded and $\beta>-1$. It is well known that there is $u$ in $W^{2,p}(S,\bar{g})$ solving \eqref{eqn:barg} because
	\begin{equation*}
		\int_S wf d\bar{V} = \int_S f d\tilde{V} =0.
	\end{equation*}
	Here $W^{2,p}(S,\bar{g})$ is the Sobolev space on $S$ with respect to $\bar{g}$. By the Sobolev embedding theorem, we know $u$ is bounded and 
	\begin{equation*}
		\int_S \abs{\tilde{\nabla} u}^2 d\tilde{V} =\int_S \abs{\bar{\nabla} u}^2 d\bar{V}=0.
	\end{equation*}
	The Schauder interior estimate then implies that $u\in \mathcal E^{4,\alpha}$, which concludes the proof of the existence part of the lemma. If $u_1$ and $u_2$ are two such solutions, then $u_1-u_2$ is a bounded harmonic function on $S$ with respect to both the conical metric $\tilde{g}$ and the smooth metric $\bar{g}$, which has to be a constant.  
\end{proof}

\section{An interpolation of H\"older norm}
We need several lemmas on the interpolation of H\"older norms, which should be well known. We include a proof for completeness.

\begin{lem}\label{lem:D1}
	There is a universal constant $C$ such that if $u$ is in $C^{2,\alpha}(B)$ and satisfies
	\begin{equation*}
		\norm{u}_{C^0(B)}\leq 1
	\end{equation*}
	and
	\begin{equation*}
		[u]_{2,\alpha,B}:= \sup_{x,y\in B} \frac{\abs{\nabla^2 u(x)- \nabla^2 u(y)}}{\abs{x-y}^\alpha}\leq 1,
	\end{equation*}
	then
	\begin{equation*}
		\norm{\nabla^2 u}_{C^0(B)}\leq C.
	\end{equation*}
\end{lem}
\begin{proof}
	If the lemma is not true, then there exists a sequence $u_i$ satisfying
	\begin{enumerate}[(1)]
		\item $\norm{u_i}_{C^0(B)}\leq 1$;
		\item $\norm{\nabla^2 u_i}_{C^0(B)}\geq i$;
		\item $[u_i]_{2,\alpha,B}\leq 1$.
	\end{enumerate}
	By setting
	\begin{equation*}
		v_i= \frac{u_i}{\norm{\nabla^2 u_i}_{C^0(B)}},
	\end{equation*}
	we obtain
	\begin{enumerate}[(a)]
		\item $\norm{v_i}_{C^0(B)}\to 0$;
		\item $\norm{\nabla^2 v_i}_{C^0(B)}=1$;
		\item $[v_i]_{2,\alpha,B}\to 0$,
	\end{enumerate}
	which altogether imply that $\norm{v_i}_{C^{2,\alpha}(B)}\leq C$ for $i$ large. By taking subsequence if necessary, we may assume that $v_i$ converges strongly to $v$ in $C^2(B)$ norm. We get a contradiction because (a) implies that $v$ is constant, while (b) implies that
	\begin{equation*}
		\norm{\nabla^2 v}_{C^0(B)}=1.
	\end{equation*}
\end{proof}

\begin{lem}\label{lem:D2}
	Suppose that $u$ is a $C^{2,\alpha}$ function from $B$ to $\Real$ satisfying $\norm{u}_{C^{2,\alpha}(B)}\leq 1$. We have
	\begin{equation*}
		\sup_{B_{1/2}} \abs{\nabla^2 u} \leq C \norm{u}_{C^0(B)}^{\frac{\alpha}{2+\alpha}}.
	\end{equation*}
\end{lem}
\begin{proof}
	We may assume that $\lambda:=\norm{u}_{C^0(B)}$ is smaller than $1/8$, otherwise there is nothing to prove. Set
	\begin{equation*}
		v(x) = \lambda^{-1} u(\lambda^{\frac{1}{2+\alpha}} x),
	\end{equation*}
	which is defined on $B':=B_{\lambda^{-\frac{1}{2+\alpha}}}$. Notice that $B'\supset B_2$ since we have assumed that $\lambda<1/8$. The advantage of $v$ is that
	\begin{equation}\label{eqn:vc0}
		\norm{v}_{C^0(B')}=1
	\end{equation}
	and
	\begin{equation}\label{eqn:vc2alpha}
		\sup_{x,y\in B'} \frac{\abs{\nabla^2 v(x)- \nabla^2 v(y)}}{\abs{x-y}^\alpha}= \sup_{\tilde{x},\tilde{y}\in B} \frac{\abs{\nabla^2 u(\tilde{x})- \nabla^2 u(\tilde{y})}}{\abs{\tilde{x}-\tilde{y}}^\alpha} \leq 1.
	\end{equation}
	We may apply Lemma \ref{lem:D1} to $v$ (restricted to $B_1(x)$ for each $x\in B_{\frac{1}{2}\lambda^{-\frac{1}{2+\alpha}}}$) to show
	\begin{equation*}
		\norm{\nabla^2 v}_{C^0(B_{\frac{1}{2}\lambda^{-\frac{1}{2+\alpha}}})}\leq C,
	\end{equation*}
	which is exactly what we need when translated back into the inequality of $u$.
\end{proof}

\bibliography{foo}

\begin{thebibliography}{10}

\bibitem{bahuaud2014yamabe}
E.~Bahuaud and B.~Vertman.
\newblock Yamabe flow on manifolds with edges.
\newblock {\em Mathematische Nachrichten}, 287(2-3):127--159, 2014.

\bibitem{chen2014long}
X.~Chen and Y.~Wang.
\newblock {On the long time behaviour of the Conical K\"ahler-Ricci flows}.
\newblock {\em arXiv:1402.6689}, 2014.

\bibitem{chen2015bessel}
X.~Chen and Y.~Wang.
\newblock {Bessel Functions, heat kernel and the conical K{\"a}hler--Ricci
  flow}.
\newblock {\em Journal of Functional Analysis}, 269(2):551--632, 2015.

\bibitem{donaldson2012kahler}
S.~Donaldson.
\newblock K{\"a}hler metrics with cone singularities along a divisor.
\newblock In {\em Essays in mathematics and its applications}, pages 49--79.
  Springer, 2012.

\bibitem{trudinger1983elliptic}
D.~Gilbarg and N.~Trudinger.
\newblock {\em Elliptic partial differential equations of second order}.
\newblock Springer-Verlag Berlin, 1983.

\bibitem{hamilton1988ricci}
R.~Hamilton.
\newblock {The Ricci flow on surfaces}.
\newblock {\em Contemp. Math}, 71(1):237--261, 1988.

\bibitem{jeffres2000uniqueness}
T.~Jeffres.
\newblock {Uniqueness of K\"ahler-Einstein cone metrics}.
\newblock {\em Publicacions matematiques}, 44(2):437--448, 2000.

\bibitem{jeffres2003regularity}
T.~Jeffres and P.~Loya.
\newblock Regularity of solutions of the heat equation on a cone.
\newblock {\em International Mathematics Research Notices}, 2003(3):161--178,
  2003.

\bibitem{jeffres2011k}
T.~Jeffres, R.~Mazzeo, and Y.~Rubinstein.
\newblock {K\"ahler-Einstein metrics with edge singularities}.
\newblock {\em Annals of Mathematics}, 183:95--176, 2016.

\bibitem{ladyzhenskaya1968linear}
O.~A. Ladyzhenskaya, V.~A. Solonnikov, and N.~N. Ural'tseva.
\newblock Linear and quasilinear equations of parabolic type, {Translations of
  Mathematical Monographs, Vol. 23}.
\newblock {\em American Mathematical Society, Providence}, 1968.

\bibitem{li1984p}
P.~Li and R.~Schoen.
\newblock {$L^p$ and mean value properties of subharmonic functions on
  Riemannian manifolds}.
\newblock {\em Acta Mathematica}, 153(1):279--301, 1984.

\bibitem{lieberman1996second}
G.~Lieberman.
\newblock {\em Second order parabolic differential equations}.
\newblock World scientific, 1996.

\bibitem{mazzeo1991elliptic}
R.~Mazzeo.
\newblock {Elliptic theory of differential edge operators I}.
\newblock {\em Communications in Partial Differential Equations},
  16(10):1615--1664, 1991.

\bibitem{mazzeo2013ricci}
R.~Mazzeo, Y.~Rubinstein, and N.~Sesum.
\newblock Ricci flow on surfaces with conic singularities.
\newblock {\em Analysis and PDE}, 8(4):839--882, 2015.

\bibitem{mooers1999heat}
E.~Mooers.
\newblock Heat kernel asymptotics on manifolds with conic singularities.
\newblock {\em Journal d'Analyse Mathematique}, 78(1):1--36, 1999.

\bibitem{phong2014ricci}
D.~Phong, J.~Song, J.~Sturm, and X.~Wang.
\newblock {The Ricci flow on the sphere with marked points}.
\newblock {\em arXiv:1407.1118}, 2014.

\bibitem{phong2015convergence}
D.~Phong, J.~Song, J.~Sturm, and X.~Wang.
\newblock {Convergence of the conical Ricci flow on $S^2$ to a soliton}.
\newblock {\em arXiv:1503.04488}, 2015.

\bibitem{ramos2011smoothening}
D.~Ramos.
\newblock Smoothening cone points with ricci flow.
\newblock {\em arXiv:1109.5554}, 2011.

\bibitem{simon2001deformation}
M.~Simon.
\newblock {Deformation of Lipschitz Riemannian metrics in the direction of
  their Ricci curvature}.
\newblock In {\em Proceedings of the International Conference held to honour
  the 60th Birthday of AM Naveira}, volume~8, page~14. World Scientific, 2001.

\bibitem{simon2013local}
M.~Simon.
\newblock {Local smoothing results for the Ricci flow in dimensions two and
  three}.
\newblock {\em Geometry \& Topology}, 17(4):2263--2287, 2013.

\bibitem{topping2010ricci}
P.~Topping.
\newblock Ricci flow compactness via pseudolocality, and flows with incomplete
  initial metrics.
\newblock {\em Journal of the European Mathematical Society}, 12(6):1429--1451,
  2010.

\bibitem{troyanov1991prescribing}
M.~Troyanov.
\newblock Prescribing curvature on compact surfaces with conical singularities.
\newblock {\em Transactions of the American Mathematical Society},
  324(2):793--821, 1991.

\bibitem{wang2014smooth}
Y.~Wang.
\newblock {Smooth approximations of the Conical K\"ahler-Ricci flows}.
\newblock {\em arXiv:1401.5040}, 2014.

\bibitem{yin2013ricci}
H.~Yin.
\newblock {Ricci flow on surfaces with conical singularities, II}.
\newblock {\em arXiv:1305.4355}, 2013.

\bibitem{yin2016}
H.~Yin and K.~Zheng.
\newblock {Expansion formula for complex Monge-Amp\`ere equation along cone
  singuarlties}.
\newblock {\em preprint}, 2016.

\bibitem{yin2010ricci}
Hao Yin.
\newblock Ricci flow on surfaces with conical singularities.
\newblock {\em Journal of Geometric Analysis}, 20(4):970--995, 2010.

\end{thebibliography}
\bibliographystyle{plain}
\end{document}